\documentclass[12pt,a4paper,reqno]{amsart}
\usepackage[noadjust]{cite}
\usepackage{amsfonts}
\usepackage{amsthm}
\usepackage{amsmath}
\usepackage{amscd}
\usepackage{comment}
\usepackage[applemac]{inputenc}
\usepackage{t1enc}
\usepackage[mathscr]{eucal}
\usepackage{indentfirst}
\usepackage[dvipsnames]{xcolor}
\usepackage{graphicx}
\usepackage{graphics}
\usepackage{pict2e}
\usepackage{epic}
\numberwithin{equation}{section}
\usepackage[section]{placeins}
\usepackage[margin=2.9cm]{geometry}
\usepackage{epstopdf}
\usepackage{hyperref} 

\theoremstyle{plain}
\newtheorem{Th}{Theorem}[section]
\newtheorem{Lemma}[Th]{Lemma}
\newtheorem{Corollary}[Th]{Corollary}
\newtheorem{Proposition}[Th]{Proposition}

\theoremstyle{definition}
\newtheorem{Definition}[Th]{Definition}

\newtheorem{Remark}[Th]{Remark}
\newtheorem{?}[Th]{Problem}

\newcommand{\barC}{{\overline{C}}}

\begin{document}

\title{Clifford Prolate Spheroidal wave Functions}

\author[H. Baghal Ghaffari, J.A. Hogan, J.D. Lakey]{Hamed Baghal Ghaffari, Jeffrey A. Hogan, Joseph D. Lakey}

%\address{Massachusetts Institute of Technology \\ Department of Mathematics \\
%Cambridge MA 02139 \&  E\"{o}tv\"{o}s Lor\'{a}nd University \\ Department of Computer 
%Science \\ H-1117 Budapest
%\\ P\'{a}zm\'{a}ny P\'{e}ter s\'{e}t\'{a}ny 1/C \\ Hungary} 

%\email{peter.csikvari@gmail.com}

%\subjclass[2010]{Primary: 05C??. Secondary: 05C??}

\keywords{Clifford Prolate Spheroidal Wave Functions, Spectral Concentrations
	\newline
	AMS Classification is $15A66$
}

\begin{abstract}
In the present paper, we introduce the multidimensional Clifford prolate spheroidal wave functions (CPSWFs) defined on the unit ball as eigenfunctions of a Clifford differential operator and provide a Galerkin method for their computation as linear combinations of Clifford-Legendre polynomials. We show that these functions are 
 eigenfunctions of the truncated Fourier transformation. Then we investigate the role of the CPSWFs in the spectral concentration problem associated with balls in the space and frequency domains, the behaviour of the eigenvalues of the time-frequency limiting operator and their spectral accumulation property.
\end{abstract}

\bigskip
\maketitle
\section{Introduction}

Prolate spheroidal wave functions (PSWFs) are special functions that have long been used in mathematical physics. They are real-valued functions on the line which arise when solving the Helmholtz equation by separation of variables in prolate spheroidal coordinates (playing the role of Legendre polynomials in spherical coordinates). In a beautiful series of papers published in the Bell Labs Technical Journal in the 1960s, \cite{slepian1961prolate} Slepian, Pollak and Landau observed that these same functions were the solutions of the spectral concentration problem which is of enormous importance in communications technologies. This observation allowed for the efficient computation of PSWFs and the eventual incorporation of their digital counterparts (the discrete PSWFs) in computer hardware. Analyticity and asymptotic properties of PSWFs, including numerical evaluations and applications to quadrature and  interpolation are investigated in
\cite{xiao2001prolate,beylkin2002generalized,boyd2003large,glaser2007fast,hogan2015frame,moore2004prolate,osipov2013prolate,schmutzhard2015numerical}. Other applications of PSWFs can be found in
\cite{thomson2007jackknifing,chen2008mimo,dilmaghani2003novel,gosse2013compressed,hu2014doa,hogan2010sampling,khare2003sampling,lindquist2008spatial,senay2009reconstruction}.

In higher dimensions, the computation of PSWFs is more problematic, and the differential equation from which they arise is singular, causing instabilities\cite{slepian1964prolate} \cite{lederman2017numerical}. Furthermore, the higher dimensional PSWFs, like the one-dimensional PWSFs, are real-valued. Here we seek natural multi-channel versions of the PSWFs of \cite{slepian1964prolate} with a view to applications in the treatment of multi-channel signals such as colour images and electromagnetic fields.

This paper is organized as follows. We will review the preliminaries about Clifford Analysis in second section. The third section will be about the definition of the multidimensional Clifford prolate spheroidal wave functions (CPSWFs). In section four, we will have a look at the Sturm-Liouville operator related to the radial part of the CPSWFs. Section five, will be about the proofs stating the CPSWFs are the eigenfuntions of the finite Fourier transformations. In section 6, we will investigate the essential features of the new multidimensional prolate spheroidal wave functions such as the decrease of eigenvalues. In the current section, we will talk about spectral concentration problem. Lastly, we see the spectrum accumulation properties CPSWFs in $m$ dimensions.

\section{Clifford analysis}
Let
$\mathbb{R}^{m}$
be 
$m$-dimensional 
euclidean space and let
$\{e_{1},e_{2},\dots ,e_{m}\}$
be an orthonormal basis for
$\mathbb{R}^{m}.$
We endow these vectors with the multiplicative properties
\begin{align*}
	e_{j}^{2}&=-1,\; \; j=1,\dots , m,\\
	e_{j}e_{i}&=-e_{i}e_{j}, \;\; i\neq j, \;\; i,j=1,\dots , m.
\end{align*}
For any subset
$A=\{j_{1},j_{2},\dots, j_{h}\}\subseteq \{1,\dots ,	m\}=Q_m,$ with $j_1<j_2<\cdots <j_h$
we consider the formal product
$e_{A}=e_{j_{1}}e_{j_{2}}\dots e_{j_{h}}.$
Moreover for the empty set
$\emptyset$
one puts
$e_{\emptyset}=1$ (the identity element). The Clifford algebra ${\mathbb R}_m$ is then the $2^m$-dimensional real associative algebra 
$${\mathbb R}_m=\bigg\{\sum\limits_{A\subset Q_m}\lambda_Ae_A:\, \lambda_A\in{\mathbb R}\bigg\}.$$
Similarly, the Clifford algebra ${\mathbb C}_m$ is the $2^m$-dimensional complex associative algebra
$${\mathbb C}_m=\bigg\{\sum\limits_{A\subset Q_m}\lambda_Ae_A:\, \lambda_A\in{\mathbb C}\bigg\}.$$
Every element $\lambda=\sum\limits_{A\subset Q_m}\lambda_Ae_A\in{\mathbb C}_m$ may be decomposed as  
$\lambda=\sum\limits_{k=0}^{m}[\lambda]_{k},$
where 
$[\lambda]_{k}=\sum\limits_{\vert A\vert=k}\lambda_{A}e_{A}$
is the so-called 
$k$-vector
part of 
$\lambda\, (k=0,1,\dots ,m).$

Denoting by 
$\mathbb{R}_{m}^{k}$
the subspace of all 
$k$-vectors
in
$\mathbb{R}_{m},$
i.e., the image of 
$\mathbb{R}_{m}$
under the projection operator 
$[\cdot]_{k},$
one has the multi-vector decomposition
$\mathbb{R}_{m}=\mathbb{R}_{m}^{0}\oplus \mathbb{R}_{m}^{1}\oplus\cdots \oplus \mathbb{R}_{m}^{m},$
leading  to the identification of
$\mathbb{R}$
with the subspace of real scalars
$\mathbb{R}_{m}^{0}$
and of
$\mathbb{R}^{m}$
with the subspace of real Clifford vectors 
$\mathbb{R}_{m}^{1}.$ The latter identification is achieved by identifying the point
$(x_{1},\dots,x_{m})\in{\mathbb R}^m$
with the Clifford number
$x=\sum\limits_{j=1}^{m}e_{j}x_{j}\in{\mathbb R}_m^1$.
The Clifford number 
$e_{M}=e_{1}e_{2}\cdots e_{m}$
is called the pseudoscalar; depending on the dimension 
$m,$
the pseudoscalar commutes or anti-commutes with the 
$k$-vectors
and squares to 
$\pm 1.$
The Clifford conjugation on ${\mathbb C}^m$is the conjugate linear mapping $\lambda\mapsto\bar\lambda$ of ${\mathbb C}_m$ to itself satisfying
\begin{align*}
	\overline{\lambda \mu}&=\bar\mu\bar{\lambda},\;\;\;\; \textnormal{for all}\;\lambda,\mu\in\mathbb{C}_{m},\\
	\overline{\lambda_{A}e_{A}}&=\overline{\lambda_{A}}\overline{e_{A}},\;\;\; \lambda_A\in\mathbb{C},\\
	\overline{e_{j}}&=-e_{j},\;\; j, \;\; j=1,\cdots , m.
\end{align*}
The Clifford conjugation leads to a Clifford inner product $\langle \cdot,\cdot\rangle$ and an associated norm $|\cdot |$ on 
$\mathbb{C}_{m}$
given respectively by
$$\langle \lambda,\mu\rangle=[\bar{\lambda}\mu]_{0}\;\;\;\textnormal{and}\;\;\; \vert\lambda\vert^{2}=[\bar{\lambda}\lambda]_{0}=\sum\limits_{A}\vert\lambda_{A}\vert^{2},$$
for $\lambda=\sum_{A\subset Q_M}\lambda_Ae_A\in{\mathbb C}_m$. 

The product of two vectors $x$, $y\in{\mathbb R}_m^1$ can be decomposed as the sum of a scalar and a 2-vector, also called a bivector:
$$xy=-\langle x, y\rangle +x\wedge y,$$
where
$\langle x,y\rangle=-\sum\limits_{j=1}^{m}x_{j}y_{j}\in \mathbb{R}^{0}_{m}$,
and,
$x\wedge y=\sum\limits_{i=1}^{m}\sum\limits_{j=i+1}^{m}e_{i}e_{j}(x_{j}y_{j}-x_{j}y_{i})\in\mathbb{R}^{2}_{m}$.
Note that the square of a vector variable 
$x$
is scalar-valued and equals the norm squared up to minus sign:
$$x^{2}=-\langle x,x\rangle=-\vert x\vert^{2}.$$
Clifford analysis offers a function theory which is a higher-dimensional analogue of the theory of holomorphic functions of one complex variable. The functions considered are defined in the Euclidean space 
$\mathbb{R}^{m}$
and take their values in the Clifford algebra 
$\mathbb{R}_{m}.$

The central notion in Clifford analysis is monogenicity, which is a multidimensional counterpart to that of holomorphy in the complex plane. 
\begin{Definition}
	A function 
	$f(x)=f(x_{1},\dots, x_{m})$
	defined and continuously
	differentiable in an open region of 
	$\mathbb{R}^{m}$
	and taking values in 
	$\mathbb{C}_{m}$
	is said to be left monogenic in that region if 
	$$\partial_{x}f=0,$$
	where
	$\partial_{x}=\sum_{j=1}^me_j\partial_{x_j}$
	is the Dirac operator 
	%$$\partial_{x}=\sum_{j=1}^{m}e_{j}\partial_{x_{j}},$$
	and 
	$\partial_{x_{j}}$
	is the partial differential operator
	$\frac{\partial}{\partial x_{j}}.$
	We also define the Euler differential operator by
	$$E=\sum_{j=1}^{m}x_{j}\partial_{x_{j}}.$$
	The Laplace operator is factorized by the Dirac operator as follows:
	\begin{equation}
	\Delta_{m}=-\partial_{x}^{2}.
	\end{equation}
\end{Definition}
The notion of right monogenicity is defined in a similar way by letting the Dirac operator act from the right. A ${\mathbb C}_m$-valued function 
$f(x)=\sum_{A\subset Q_m}f_A(x)e_A$ (where each $f_A$ takes complex values)
is left monogenic if and only if its Clifford conjugate 
$\bar{f}(x)=\sum_{A\subset Q_m}\overline{f_A}(a)\overline{e_A}$
is right monogenic. In fact, $\overline{\partial f}=-\overline{f}\partial$.
\begin{Definition}\label{definition3.2}
	A left, respectively right, monogenic homogeneous polynomial
	$P_{k}$
	of degree 
	$k\; (k\geq 0)$
	in
	$\mathbb{R}^{m}$
	is called a left, respectively right, solid inner spherical monogenic of order 
	$k.$
	The set of all left, respectively right, solid inner spherical monogenic of order 
	$k$
	will be denoted by
	$M_{l}^{+},$
	respectively
	$M_{r}^{+}.$
	It can be shown 
	\cite{delanghe2012clifford}
	that the dimension of 
	$M_{l}^{+}(k)$
	is given by 
	$$\dim M_{l}^{+}(k)=\frac{(m+k-2)!}{(m-2)!k!}.$$
	A left, respectively right, monogenic homogeneous function $Q_k$ of degree $-(k+m-11)$ in $R^m \setminus\{0\}$ is called a left, respectively right, solid outer spherical monogenic of order $k.$
\end{Definition}
\begin{Lemma}
	We can see that 
	$$M_{l}^{+}(k)\cap M_{l}^{-}(k)=\{0\}.$$
\end{Lemma}
\begin{Proposition}
	For any two monogenic, homogeneous functions
	$Y_{k}(x)$ 
	and 
	$Y_{l}(x),$
	we have that
	$$\langle Y_{k}(x),Y_{l}(x)\rangle_{L^{2}(s^{m-1})}=\delta_{k,l},$$
	where $Y_{k}(x),Y_{l}(x)\in M_{l}^{+}(k)\cup M_{l}^{-}(k).$
\end{Proposition}
Therefore, we can define the monogenic functions
$$M_{l}(k):=M_{l}^{+}(k)\oplus M_{l}^{-}(k).$$
It's possible to see that if $P_{k}(x)\in M_{l}^{+}(k)$ then $Q_{k}=\dfrac{x}{\vert x\vert^{m}}P_{k}(\dfrac{x}{\vert x\vert^2})\in M_{l}^{-}(k).$

Below, we make some remarks which will be required in proving main theorems. We will need also the following lemma which is easy to obtain by direct calculation.
\begin{Lemma}\label{diracderivativelemma}
	For
	$P_{k}\in M_{l}^{+}(k)$
	and 
	$s\in \mathbb{N}$
	the following fundamental formula holds:
	\begin{equation*}
	\partial_{x}[x^{s}P_{k}]=
	\left\lbrace \begin{array}{l}
	-sx^{s-1}P_{k}\hspace*{3.55cm} \textnormal{for}\;s\;\textnormal{even},\\
	-(s+2k+m-1)x^{s-1}P_{k}\;\;\;\;\;\; \textnormal{for}\;s\;\textnormal{odd.}\\
	\end{array} \right.
	\end{equation*}
\end{Lemma}
\begin{proof}
For the proof see \cite{delanghe2012clifford}.
\end{proof}
\begin{Definition}
	A real-valued polynomial
	$S_{k}$
	of degree
	$k$
	on
	$\mathbb{R}^{m}$
	satisfying
	$$\Delta_{m}S_{k}(x)=0,\;\;\;\;\; \textnormal{and,}\;\;\;\;\; S_{k}(tx)=t^{k}S_{k}(x)\quad (t>0),$$
	is called a solid spherical harmonic of degree $k$. The collection of solid spherical harmonics of degree $k$ on
	$\mathbb{R}^{m}$
	is denoted
	$\mathcal{H}(k)$
	(or $\mathcal{H}(m,k)$).
\end{Definition}
Since
$\partial^{2}_{x}=-\Delta_{m}$,
we have that
$$M_{l}^{+}(k)\subset \mathcal{H}(k),\;\;\;\;\; \textnormal{and,}\;\;\;\;\; M_{r}^{+}(k)\subset \mathcal{H}(k).$$
Let
$H_{(r)}$
be a unitary right Clifford-module, i.e. 
$(H_{(r)},+)$
is an abelian group
and a law of scalar multiplication
$(f,\lambda)\to f\lambda$
from
$H_{(r)}\times\mathbb{C}_{m}$
into
$H_{r}$
is defined such that for all
$\lambda,\mu\in\mathbb{C}_{m}$
and
$f,g\in H_{(r)}:$
\begin{align*}
	&(i)\;f(\lambda+\mu)=f\lambda+f\mu,\hspace*{10cm}\\
	&(ii)\;f(\lambda\mu)=(f\lambda)\mu,\\
	&(iii)\;(f+g)\lambda=f\lambda+g\lambda,\\
	&(iv)\;fe_{\emptyset}=f.
\end{align*}
Note that
$H_{(r)}$
becomes a complex vector space if
$\mathbb{C}$
is identified with
$\mathbb{C}e_{\emptyset}\subset\mathbb{C}_{m}.$
Then a function
$\langle\cdot,\cdot\rangle:H_{(r)}\times H_{(r)}\to\mathbb{C}_{m}$
is said to be an inner product on
$H_{(r)}$
if for all
$f,g,h\in H_{(r)}$
and
$\lambda\in\mathbb{C}_{m}:$
\begin{align*}
	&(i)\;\langle f,g\lambda+h\rangle=\langle f,g\rangle\lambda+\langle f,h\rangle,\hspace*{10cm}\\
	&(ii)\;\langle f,g\rangle=\overline{\langle g,f\rangle},\\
	&(iii)\;[\langle f,f\rangle]_{0}\geq 0 \;\; \textnormal{and}\;\; [\langle f,f\rangle]_{0}= 0 \;\;\textnormal{if and only if}\;\; f=0.
\end{align*}
From this
${\mathbb C}_m$-valued
inner product
$(\cdot,\cdot)$,
one can recover the complex inner product
$$( f,g) =[\langle f,g\rangle]_{0},$$
on
$H_{r}$. Putting for each
$f\in H_{(r)}$
\begin{equation}
\Vert f\Vert^{2}=( f,f) ,\label{ip norm}
\end{equation}
$\Vert \cdot\Vert$
becomes a norm on 
$H_{r}$
turning it into a normed right Clifford-module. 

Now, let
$H_{(r)}$
be a unitary right Clifford-module provided with an inner product
$\langle\cdot,\cdot\rangle.$
Then it is called a right Hilbert Clifford-module if
$H_{(r)}$
considered as a complex vector space provided with the complex inner product 
$(\cdot,\cdot)$
is a Hilbert space. 

We consider the ${\mathbb C}_m$-valued inner product of the functions 
$f,g:{\mathbb R}^m\to{\mathbb C}_m$
%and
%$g$
%defined in
%$\mathbb{R}^{m}$
%and taking values in the Clifford algebra 
%$\mathbb{R}_{m}$
by
$$ \langle f,g\rangle=\int\limits_{\mathbb{R}^{m}}\overline{f(x)}g(x)\, dx,$$
where
$dx$
is Lebesgue measure on 
$\mathbb{R}^{m}$
and moreover the associated norm given by (\ref{ip norm}).
%$$\Vert f\Vert=[\langle f,f\rangle]_{0}.$$
The right Clifford-module of ${\mathbb C}_m$-valued measurable functions on
$\mathbb{R}^{m}$
for which 
$\Vert f\Vert^{2}<\infty$
is a right Hilbert Clifford-module which we denote by 
$L^{2}(\mathbb{R}^{m},{\mathbb C}_m).$
Therefore, we obtain the right Hilbert Clifford-module of square integrable functions:
$L^{2}(\mathbb{R}^{m},{\mathbb C}_m)$ of functions $f:{\mathbb R}^m\to{\mathbb C}_m$ for which each component of $f$ is measurable and 
$$\Vert f\Vert_{2}=\left(\int\limits_{\mathbb{R}^{m}}\vert f(x)\vert^{2}\, dx\right)^{\frac{1}{2}}<\infty.$$

The standard tensorial multi-dimensional Fourier transform given by:
\begin{equation*}
\mathcal{F}f(\xi)=\int\limits_{\mathbb{R}^{m}}e^{-2\pi i\langle x,\xi\rangle}f(x)\, dx.
\end{equation*}
Now, we have two definitions related to the operators.
\begin{Definition}
Let $L^{2}(\mathbb{R}^{m},{\mathbb C}_m)$ be right Hilbert Clifford-module. For any $f,g\in L^{2}(\mathbb{R}^{m},{\mathbb C}_m),$ the operator $T$ is self-adjoint if 
$$\langle Tf,g\rangle=\langle f,Tg\rangle.$$
\end{Definition}

\begin{Definition}
Let $L^{2}(\mathbb{R}^{m},{\mathbb C}_m)$ be right Hilbert Clifford-module. The operator $T$ is compact, if $\{T(f_{n})\}$ has a convergent subsequence for every bounded sequence of $\{f_{n}\}\in L^{2}(\mathbb{R}^{m},{\mathbb C}_m)$.
\end{Definition}

Now we state some fundamental results without proof which will be useful later. 
\begin{Th}\label{planchereltheorem}
	The Fourier transform 
	$\mathcal{F}$
	is an isometry on the space of the square integrable functions, in other words, for all
	$f,g\in L_{2}(\mathbb{R}^{m},{\mathbb C}_m)$
	the Parseval formula holds:
	$$\langle f,g\rangle=\langle \mathcal{F}f,\mathcal{F}g\rangle.$$
	In particular, for each 
	$f\in L_{2}(\mathbb{R}^{m},{\mathbb C}_m)$
	one has $\Vert f\Vert_{2}=\Vert \mathcal{F}f\Vert_{2}$.
\end{Th}
\begin{proof}
	See the proof in
	\cite{delanghe2012clifford}.
\end{proof}
\begin{Proposition}
	For any two monogenic, homogeneous polynomials 
	$Y_{k}$ of homogeneous degree $k$
	and 
	$Y_{\ell}$ of homogeneous degree $\ell$
	we have 
$$\langle Y_{k},Y_{l}\rangle_{L^{2}(S^{m-1})}:=\int\limits_{S^{m-1}}\overline{Y_k(\omega )}Y_\ell (\omega )\, d\omega=0$$
if $k\neq\ell$.
\end{Proposition}
\begin{Th}(\textbf{Clifford-Stokes theorem})\label{Clifford_Stokes}
	Let
	$f,g\in C^{1}(\Omega ,{\mathbb C}_m)$.
Then for each compact orientable 
	$m-$dimensional manifold $C\subset\Omega$ with boundary
	$\partial C$,
	$$\int\limits_{\partial C}f(x)n(x)g(x)d\sigma(x)=\int\limits_{C}[(f(x)\partial_{x})g(x)+f(x)(\partial_{x}g(x))]dx.$$
	where 
	$n(x)$
	is the outward-pointing unit normal vector on
	$\partial C.$
\end{Th}
\begin{proof}
	For the proof, see
	\cite{delanghe2012clifford}.
\end{proof}

The following result may be obtained as a simple application of the Clifford-Stokes theorem.
\begin{Lemma}\label{property of two monogenic and x between}
	Let $f$, $g$ be defined on a neighbourhood $\Omega $ of the unit ball in ${\mathbb R}^m$ and suppose $f$ is right monogenic on $\Omega$ while $g$ is left monogenic on $\Omega$.
	Then
	\begin{equation}
	\int\limits_{S^{m-1}}f(\omega )\omega g(\omega )\, d\omega =0.
	\end{equation}
\end{Lemma}
Now here we introduce the representation for functions in $L^{2}(\mathbb{R}^{m},\mathbb{R}_{m})$. A proof can be found in \cite{delanghe2012clifford}
\begin{Th}\label{Representation_f_all_monogenics}
	For every function $f\in L^{2}(\mathbb{R}^{m},\mathbb{R}_{m}),$ we have the following representation 
	\begin{equation}\label{Representation_f_all_monogenics_equation}
	f(x)=\sum_{k=0}^{\infty}\sum_{l=1}^{d_{k}}\left[f_{k}^{(l)}(\vert x\vert)Y_{k}^{(l)}(x)+g_{k}^{(l)}(\vert x\vert)\frac{x}{\vert x\vert^{m}}Y_{k}^{(\ell )}(\frac{x}{\vert x\vert^2})\right],
	\end{equation}
	where $Y_{k}\in M_{l}^{+}(k)$ and the radial functions $f_k^{(\ell )}$ and $g_k^{(\ell )}$ satisfy
%\end{Th}
%\begin{proof}
%	See the proof at \cite{delanghe2012clifford}.
%\end{proof}
%\begin{Lemma}
%	The radial functions $f_{k}^{(l)}(\vert x\vert)$ and $g_{k}^{(l)}(\vert x\vert)$ defined in  \ref{Representation_f_all_monogenics} satisfy the followings
$$\int\limits_{0}^{\infty} r^{m+2k-1}\vert f_{k}^{(l)}(r)\vert^{2}dr<\infty,\quad\int\limits_{0}^{\infty} r^{1-m-2k}\vert g_{k}^{(l)}(r)\vert^{2}dr<\infty.$$
\end{Th}
%\begin{proof}
%	Let $f\in L^{2}(\mathbb{R}^{m},\mathbb{R}_{m})$ so f enjoys that representations \eqref{Representation_f_all_monogenics_equation}. On the contrary, 
%	$$ \langle \sum_{k=0}^{\infty}\sum_{l=1}^{d_{k}}[f_{k}^{(l)}(\vert x\vert)Y_{k}^{(l)}(x)],\sum_{k=0}^{\infty}\sum_{l=1}^{d_{k}}[g_{k}^{(l)}(\vert x\vert)\frac{x}{\vert x\vert^{m}}Y_{k}(\frac{x}{\vert x\vert^2})] \rangle=0. $$ 
%	Since $\Vert f\Vert_{2}^{2}<\infty,$ then 
%	$\Vert \sum_{k=0}^{\infty}\sum_{l=1}^{d_{k}}[f_{k}^{(l)}(\vert x\vert)Y_{k}^{(l)}(x)]\Vert_{2}^{2}<\infty.$ Now since $\{ Y_{k}^{l}(x)\}$ are orthonormal at $L^{2}(S^{m-1}),$ therefore
%	$$\sum_{k=0}^{\infty}\sum_{l=1}^{d_{k}}\int\limits_{0}^{\infty} r^{m+2k-1}\vert f_{k}^{(l)}(r)\vert^{2}dr<\infty,$$
%	which shows that $\int\limits_{0}^{\infty} r^{m+2k-1}\vert f_{k}^{(l)}(\vert x\vert)\vert^{2}dr<\infty.$ Similarly, we can show that 
%	\newline
%	$\int\limits_{0}^{\infty} r^{1-m-2k}\vert g_{k}^{(l)}(\vert x\vert)\vert^{2}dr<\infty.$
%\end{proof}
In what follows, $B(c)$ is the closed ball of radius $c> 0$ in $\mathbb{R}^{m}$.  

\begin{Lemma}\label{support_Lemma_f_summand}
	Let $f\in L^{2}(\mathbb{R}^{m},\mathbb{R}_{m}),$ supported on ${B(1)}.$ Then the radial functions $f_{k}^{(l)}(\vert x\vert)$ and $g_{k}^{(l)}(\vert x\vert)$ defined in Theorem \ref{Representation_f_all_monogenics} are supported at $[0,1].$
\end{Lemma}
\begin{proof}
	Since $f$ is supported on ${B(1)}$, we have 
	\begin{align*}
	0&=\int\limits_{{B(1)}^{c}}\left| \sum_{k=0}^{\infty}\sum_{l=1}^{d_{k}}[f_{k}^{(l)}(\vert x\vert)Y_{k}^{(l)}(x)+g_{k}^{(l)}(\vert x\vert)\frac{x}{\vert x\vert^{m}}Y_{k}^{(l)}(\frac{x}{\vert x\vert^2})]\right|^{2} \,dx\\
	&=\int\limits_{{B(1)}^{c}}\sum_{k=0}^{\infty}\sum_{l=1}^{d_{k}}f_{k}^{(l)}(\vert x\vert)\overline{Y_{k}^{(l)}(x)}\sum_{k'=0}^{\infty}\sum_{l'=1}^{d_{k'}}f_{k'}^{(l')}(\vert x\vert)Y_{k'}^{(l')}(x)\,dx\\
	&+\int\limits_{{B(1)}^{c}}\sum_{k=0}^{\infty}\sum_{l=1}^{d_{k}}f_{k}^{(l)}(\vert x\vert)\overline{Y_{k}^{(l)}(x)}\sum_{k'=0}^{\infty}\sum_{l'=1}^{d_{k'}}g_{k'}^{(l')}(\vert x\vert)\frac{x}{\vert x\vert^{m}}Y_{k'}^{(l')}(\frac{x}{\vert x\vert^2})\,dx\\
	&+\int\limits_{{B(1)}^{c}}\sum_{k=0}^{\infty}\sum_{l=1}^{d_{k}}g_{k}^{(l)}(\vert x\vert)\overline{\frac{x}{\vert x\vert^{m}}Y_{k}^{(l)}(\frac{x}{\vert x\vert^2})}\sum_{k'=0}^{\infty}\sum_{l'=1}^{d_{k'}}f_{k'}^{(l')}(\vert x\vert)Y_{k'}^{(l')}(x)\,dx\\
	&+\int\limits_{{B(1)}^{c}}\sum_{k=0}^{\infty}\sum_{l=1}^{d_{k}}g_{k}^{(l)}(\vert x\vert)\overline{\frac{x}{\vert x\vert^{m}}Y_{k}^{(l)}(\frac{x}{\vert x\vert^2})}\sum_{k'=0}^{\infty}\sum_{l'=1}^{d_{k'}}g_{k'}^{(l')}(\vert x\vert)\frac{x}{\vert x\vert^{m}}Y_{k'}^{(l')}(\frac{x}{\vert x\vert^2})\,dx.
	\end{align*}
	The second and third integrals are zero due to the Lemma \ref{property of two monogenic and x between}. By the orthonormality of $\{ Y_{k}^{l}:\, k\geq 0,\ 1\leq\ell\leq d_k\}$ we have 
	$$\sum_{k=0}^{\infty}\sum_{l=1}^{d_{k}}\int\limits_{1}^{\infty} r^{m+2k-1}\vert f_{k}^{(l)}(r)\vert^{2}\, dr+\sum_{k=0}^{\infty}\sum_{l=1}^{d_{k}}\int\limits_{1}^{\infty} r^{1-m-2k}\vert g_{k}^{(l)}(r)\vert^{2}\, dr=0.$$
	Hence, $ f_{k}^{(l)}(r)=g_{k}^{(l)}(r)=0$ for $r>1$. This completes the proof.
\end{proof}
We now give relevant background on the Clifford-Legendre polynomials, which we use to build the CPSWFs. Given $\alpha >-1$ and $Y_k\in M_l^+(k)$, the Clifford-Legendre polynomial $C_{n,m}^\alpha (Y_k)$ is defined by
$$C_{n,m}^\alpha (Y_k)=(1+x^2)^{-\alpha}\partial_x^n[(1+x^2)^{\alpha +n}Y_k(x)].$$
Let ${\mathcal L}_\alpha$ be the Clifford differential operator
$${\mathcal L}_\alpha f=(1+x^2)\partial_x^2f-2(\alpha +1)x\partial_xf$$
which acts on $C^2({\mathbb R}^m,{\mathbb R}_m)$.
\begin{Th}(\textbf{Differential equation for the Gegenbauer polynomials})\label{differentialforClifford}
	For all
	$n,k\in\mathbb{N}$
	and
	$\alpha>-1$, the Clifford-Gegenbauer polynomial $C_{n,m}^{\alpha}(Y_{k})(x)$ is an eigenfunction of the differential operator ${\mathcal L}_\alpha$
	with real eigenvalue 
	$C(\alpha,n,m,k)$, i.e.,
	\begin{equation}\label{differential_equation_of_CL_Poly}
	{\mathcal L}_\alpha (C_{n,m}^\alpha (Y_k))=C(\alpha ,n,m,k)C_{n,m}^\alpha (Y_{k}),
	\end{equation}
	where
	$$C(\alpha,n,m,k)=\begin{cases}
	n(2\alpha+n+m+2k)&\text{ if $n$ is even,}\\
	(2\alpha+n+1)(n+m+2k-1)&\text{ if
		$n$ is odd,}
	\end{cases}$$
and $\{Y_{k}^{i}\}_{i=1}^{d_k}$ is any orthonormal basis for $M_l^+(k)$.	
\end{Th}
\begin{proof}
	See
	\cite{delanghe2012clifford}.
\end{proof}
When $\alpha=0,$ the Clifford-Gegenbauer polynomials are known as the Clifford-Legendre polynomials and are eigenfunctions of the Clifford differential operator $L_0$:
$$L_{0}(C_{n,m}^0 (Y_{k}))=C(0,n,m,k)C_{n,m}^0 (Y_{k}),$$
where $L_0 f=(1+x^2)\partial_x^2f-2x\partial_xf.$

Let $Y_k\in M_l^+(k)$ be normalized in the sense that $\int\limits_{S^{m-1}}|Y_k(\omega )|^2\, d\omega =1$. Using a result from \cite{propertiesofcliffordlegendre}, we now define the normalized Clifford-Legendre polynomials $\barC_{n,m}^0(Y_k)$ by 
\begin{equation}
\barC_{n,m}^0(Y_k)=\frac{\sqrt{2k+2n+m}}{2^nn!}C_{n,m}^0(Y_k)\label{normalized C_L}.
\end{equation}

We now state some useful results from \cite{propertiesofcliffordlegendre} which we will use in this paper. \begin{Th}\label{thm: Bonnet formula}
	The Bonnet formula for the normalized Clifford-Legendre polynomials is as follows:
	\begin{enumerate}
		\item[(a)] When $n=2N$ is even,
\begin{equation}
x\barC_{2N,m}^0(Y_k)(x)=A_{N,k,m}\barC_{2N+1,m}^0(Y_k)(x)+B_{N,k,m}\barC_{2N-1,m}^0(Y_k)(x)\label{even Bonnet}
\end{equation}
		where
		\begin{align*}
		A_{N,k,m}&=\frac{-(\frac{m}{2}+N+k)\sqrt{m+4N+2k}}{(\frac{m}{2}+2N+k)\sqrt{m+4N+2k+2}},\\
		B_{N,k,m}&=\frac{N\sqrt{m+4N+2k}}{(\frac{m}{2}+2N+k)\sqrt{m+4N+2k-2}},
		\end{align*}
		\item[(b)] When $n=2N+1$ is odd
\begin{equation}
x\barC_{2N+1,n}^0(Y_k)(x)=A_{N,k,m}'\barC_{2N+2,m}^0(Y_k)(x)+B_{N,k,m}'\barC_{2N,m}^0(Y_k)(x),\label{odd Bonnet}
\end{equation}
		where
		\begin{align*}
		A_{N,k,m}'&=\frac{-(N+1)\sqrt{m+4N+2k+2}}{(\frac{m}{2}+2N+k+1)\sqrt{m+4N+2k+4}},\\
		B_{N,k,m}'&=\frac{(\frac{m}{2}+N+k)\sqrt{m+4N+2k+2}}{(\frac{m}{2}+2N+k+1)\sqrt{m+4N+2k}}.
		\end{align*}
	\end{enumerate}
\end{Th}
\begin{Th}
	The Fourier transform of the restriction of the Clifford-Legendre polynomial
	$C_{n,m}^{0}(Y_{k})(x)$
	to the unit ball
	$B(1)$
	is given by
	\begin{equation}
	\mathcal{F}(C_{n,m}^{0}(Y_{k}))(\xi)=(-1)^{k}i^{n+k}2^{n}n!\xi^n\frac{J_{k+\frac{m}{2}+n}(2\pi \vert\xi\vert)}{\vert \xi\vert^{\frac{m}{2}+n+k}}Y_{k}(\xi).
	\end{equation}
\end{Th}
\begin{Th}\label{Even and Odd Clifford Legendre}
	In dimension $m=2$, the normalised Clifford-Legendre polynomials satisfy
	$$\barC_{2N+1,2}^{0}(Y_{k})(x)=-e_{1}\barC_{2N,2}^{0}(Y_{k+1})(x).$$
	
\begin{Lemma}\label{partial_Cl_sum_of_Cl}
There are real constant $a_{n,m,j}^k$ such that
\begin{equation}
\partial_{x} C_{n,m}^0(Y_k)(x)=4n(n-1+k+\frac{m}{2})C_{n-1,m}^0(Y_k)(x)+\sum_{j=0}^{\lceil \frac{n}{2}\rceil-1}a_{n,m,j}^kC_{n-2-2j,m}^0(Y_k)(x).
\label{Dirac of CLP1}
\end{equation}
\end{Lemma}

\begin{proof} The proof is by induction on $n$. By direct computation, we find that 
$$\partial_{x} C_{1,m}^0(Y_k)(x)=(m+2k)Y_k(x)=(m+2k)C_0^0(Y_k)(x),$$ so that (\ref{Dirac of CLP1}) is verified when $n=0$. Suppose now that (\ref{Dirac of CLP1}) holds when $n=N$. By Corollary 3.10 from \cite{propertiesofcliffordlegendre} we have
\begin{equation}
\partial_{x} C_{n+1,m}^0(Y_k)(x)=4(n+1)[(n+k+\frac{m}{2})C_{n,m}^0(Y_k)(x)-n\partial_{x} C_{n-1,m}^0(Y_k)(x)].\label{Dirac of CLP2}
\end{equation}
With $n=2N+2,$ we have that
\begin{align}
\partial_{x} C_{2N+3,m}^0(Y_k)(x)&=4(2N+3)[(2N+2+k+\frac{m}{2})C_{2N+2,m}^0(Y_k)(x)\nonumber\\
&\hspace*{4.7cm}-(2N+2)\partial_{x} C_{2N+1,m}^0(Y_k)(x)]\nonumber\\
&=4(2N+3)[(2N+2+k+\frac{m}{2})C_{2N+2,m}^0(Y_k)(x)\nonumber\\
&\hspace*{4.7cm}-(2N+2)\sum_{j=0}^{N}a_{N,j}^kC_{2N-2j,m}^0(Y_k)(x)]\nonumber\\
&=\sum_{j=0}^{N+1}a_{N+1,j}^kC_{2N+2-2j,m}^0(Y_k)(x)\label{odd_version_partial_Cl}
\end{align}
where $a_{N+1,0}^k=4(2N+3)(2N+2+k+\frac{m}{2})$ and $a_{N+1,j}^k=-4(2N+3)(2N+2)a_{N,j}^k$ for $1\leq j\leq N-1)$. Now let's assume that $n=2N+1$. Then with the same process we will get 
\begin{align}\label{even_version_partial_Cl}
\partial_{x} C_{2N+2,m}^0(Y_k)(x)&=4(2N+2)(2N+1+k+\frac{m}{2})C_{2N+1,m}^0(Y_k)(x)\nonumber\\
&+\sum_{j=1}^{N}b_{N+1,m,j}^kC_{2N+1-2j,m}^0(Y_k)(x).
\end{align}
Now the combination \eqref{odd_version_partial_Cl} and \eqref{even_version_partial_Cl} will complete the proof.
\end{proof}
	
\end{Th}
We now review relevant facts from the theory of Sturm-Liouville problems, proofs of which can be found in 

\cite{christensen2010functions,al2008sturm} 
\begin{Definition}\label{def_SL}
Let $p(x),\, q(x),$ and $r(x)$ be functions on $\mathbb{R}$ or a subinterval of $\mathbb{R};$ assume that the function
$p(x)$ is differentiable, and that the functions $p(x), q(x),$ and $r(x)$ are continuous. A differential equation that can be written on the form
\begin{equation}\label{SL_form}
[p(x)u']'+[q(x)+\lambda r(x)]u=0,
\end{equation}
for some parameter 
$\lambda\in \mathbb{R},$
is called a Sturm-Liouville differential equation and a solution $u$ of (\ref{SL_form}) is known as an eigenfunction of the Sturm-Liouville form.
\end{Definition}
\begin{Th}\label{firstSL}
Let $u_n$ and $u_m$ be real-valued eigenfunctions for a  Sturm-Liouville form defined in \eqref{SL_form} corresponding to different eigenvalues $\lambda_{n}, \,\lambda_{m}.$ Let $p(x)>0$ and $r(x)>0$ for all $x\in (a,b),$ and also let $p(a)=p(b)=0.$
Then $u_{n},$ and, $u_{m}.$ are orthogonal in $L_{r}(a,b)$, i.e., 
$$\int\limits_a^br(x)u_n(x)u_m(x)\, dx=0.$$
\end{Th}
\begin{Th}\label{secondSL}
Let $p_i, q_i,\, i = 1, 2,$ be real-valued continuous functions on the interval $[a, b]$ and let
$$(p_{1}(x)y')'+q_{1}(x)y=0,$$
$$(p_{2}(x)y')'+q_{2}(x)y=0,$$
be two homogeneous linear second order differential equations in self-adjoint form with
$0<p_{2}(x)\leq p_{1}(x),$
and,
$q_{1}(x)\leq q_{2}(x).$
Let $u$ be a non-trivial solution of the first of these equations with successive roots at $z_1$ and $z_2$ and let $v$ be a non-trivial solution of second equation. Then, there exists an $x \in (z_1, z_2)$ such that $v(x) = 0.$
\end{Th}

\section{The multidimensional Clifford Prolate Spheroidal Wave Functions}
The goal of this section is to introduce the Clifford prolate spheroidal wave functions and to compute them using Boyd's algorithm \cite{boyd2005algorithm}.
\begin{Definition}
	Given $c\geq 0$, the Clifford differential operator 
	$L_{c}$
	acts of 
	$C^{2}(B(1),\mathbb{R}_{m})$
	as follows:
	\begin{equation}\label{Definition_of_Lc_operator}
	L_{c}f(x)=\partial_{x}((1-\vert x\vert^{2})\partial_{x}f(x))+4\pi^{2}c^{2}\vert x\vert^{2}f(x)
	\end{equation}
	where 
	$B(1)$
	is the unit ball in
	$\mathbb{R}^{m}$
	and 
	$\partial_{x}$
	is the Dirac Operator. We define the Clifford Prolate Spheroidal Wave Functions (CPSWFs) as the eigenfunctions of 
	$L_{c}.$
\end{Definition}
\begin{Proposition}\label{proposition3.3}
	The operator 
	$L_{c}$
	defined in 
	\eqref{Definition_of_Lc_operator}
	is self-adjoint.
\end{Proposition}
\begin{proof} First note that $L_c=L_0+M_c$ where $M_cf(x)=4\pi^2c^2|x|^2f(x)$. Clearly $M_c$ is self-adjoint, so we need only show that $L_0$ is self-adjoint.
	An application of \ref{Clifford_Stokes} yields
	\begin{align*}
		\langle f,L_{0}g\rangle&=\int\limits_{B(1)}\overline{f(x)}[\partial_{x}((1-\vert x\vert^{2})\partial_{x}g(x)\, dx\\
%		=& \int\limits_{B(1)}\overline{f(x)}[\partial_{x}((1-\vert x\vert^{2})\partial_{x}g(x))]dx+\int\limits_{B(1)}\overline{f(x)}[4\pi^{2}c^{2}\vert x\vert^{2}g(x)]dx\\
		&= \int\limits_{\partial B(1)}\overline{f(x)}\frac{x}{\vert x\vert}(1-\vert x\vert^{2})(\partial_{x}g)(x)\, d\sigma(x)- \int\limits_{B(1)}(\overline{f}\partial_{x})(x)(1-\vert x\vert^{2})(\partial_{x}g)(x)\, dx\\%+\int\limits_{B(1)}\overline{4\pi^{2}c^{2}\vert x\vert^{2}f(x)}g(x)dx\\
		&=-\int\limits_{B(1)}[((1-\vert x\vert^{2})(\overline{f}\partial_{x})](\partial_{x}g)(x)\,dx\\
&=-\int\limits_{\partial B(1)}(1-|x|^2)(\overline{f}\partial_x)(x)\frac{x}{|x|}g(x)\, d\sigma (x)+\int\limits_{B(1)}([(1-|x|^2)(\overline{f}\partial_x)]\partial_x)g(x)\, dx\\
&=\int\limits_{B(1)}\overline{\partial_x[(1-|x|^2)\partial_xf](x)}g(x)\, dx=\langle L_0f,g\rangle,
\end{align*}
%		
%		
%		
%		%+\int\limits_{B(1)}\overline{4\pi^{2}c^{2}\vert x\vert^{2}f(x)}g(x)dx\\
%%		\;\;\;\;\;\;\;\;\;\;=&-\{ \int\limits_{\partial B(1)}((1-\vert x\vert^{2})\overline{f(x)}\partial_{x})\frac{x}{\vert x\vert} g(x)d\sigma (x)\\
%		-& \int\limits_{B(1)}([(1-\vert x\vert^{2})(\overline{f(x)}\partial_{x})]\partial_{x})g(x)\,dx\}\\%+\int\limits_{B(1)}\overline{4\pi^{2}c^{2}\vert x\vert^{2}f(x)}g(x)dx\\
%		=&\int\limits_{B(1)}([(1-\vert x\vert^{2})(\overline{f(x)})\partial_{x}]\partial_{x})g(x)\,dx\\%+\int\limits_{B(1)}\overline{4\pi^{2}c^{2}\vert x\vert^{2}f(x)}g(x)dx\\
%		=&\int\limits_{B(1)}([(1-\vert x\vert^{2})(\overline{\partial_{x}f(x)})]\partial_{x})g(x)\, dx\\%+\int\limits_{B(1)}\overline{4\pi^{2}c^{2}\vert x\vert^{2}f(x)}g(x)dx\\
%		=&\int\limits_{B(1)}(\overline{\partial_{x}[(1-\vert x\vert^{2})(\partial_{x}f(x))]})g(x)\, dx\\%+\int\limits_{B(1)}\overline{4\pi^{2}c^{2}\vert x\vert^{2}f(x)}g(x)dx\\
%		=&\int\limits_{B(1)}(\overline{\partial_{x}[(1-\vert x\vert^{2})(\partial_{x}f(x))]}g(x)\, dx\\%+4\pi^{2}c^{2}\vert x\vert^{2}f(x)})g(x)dx\\
%		=&\int\limits_{B(1)}\overline{L_{0}f(x)}g(x)\, dx=\langle L_{0}f,g\rangle
%	\end{align*}
	which completes the proof.
\end{proof}
We now outline an algorithm for the computation of the CPSWFs.

Recall that the  first Bonnet formula (\ref{even Bonnet}) of Theorem \ref{thm: Bonnet formula} takes the form
$$x\barC_{2i,m}^0(Y_{k}^{j})(x)=A_{i,k,m}\barC_{2i+1,m}^0(Y_{k}^{j})(x)+B_{i,k,m}\barC_{2i-1,m}^0(Y_{k}^{j})(x).$$
Multiplying both sides of this equation by $-x$ on both sides and applying the second Bonnet formula (\ref{odd Bonnet}) yields
\begin{equation}\label{twice_application_bonnet}
\vert x\vert^{2}\barC_{2i,m}^0(Y_{k}^{j})(x)=a_{i}\barC_{2i+2,m}^0(Y_{k}^{j})(x)+b_{i}\barC_{2i,m}^0(Y_{k}^{j})(x)+c_{i}\barC_{2i-2,m}^0(Y_{k}^{j})(x),
\end{equation}
where
$a_{i}=-A_{i,k,m}A'_{i,k,m}$
$b_{i}=-(A_{i,k,m}B'_{i,k,m}+A'_{i-1,k,m}B_{i,k,m}).\;$
$c_{i}=-B_{i,k,m}B'_{i-1,k,m}.$ 

By direct computation, it may be shown that if $X_k^e$ is the collection of functions of the form $F(|x|^2)Y_{k}^{j}(x)$ with $F\in C^\infty ([0,1],{\mathbb R})$ and $Y_{k}^{j}$ a spherical monogenic of degree $k$ on ${\mathbb R}^m$, then $X_k^e$ is invariant under $L_c$. In fact, a direct calculation, using the monogenicity and homogeneity of $Y_{k}^{j}$, gives
$$L_0[F(|x|^2)Y_{k}^{j}(x)]=[(-2m+(4+4k)|x|^2)F'(|x|^2)-4k|x|^2F''(|x|^2)]Y_{k}^{j}(x).$$
Similarly, if $X_k^o$ is the collection of functions of the form $xG(|x|^2)Y_{k}^{j}$ with $G\in C^\infty ([0,1],{\mathbb R})$ and $Y_{k}^{j}$ a spherical monogenic of degree $k$ on ${\mathbb R}^m$, then $X_k^0$ is also invariant under $L_c$. Hence, when searching for eigenfunctions of $L_c$, we may search within the spaces $X_k^e$ and $X_k^o$. Furthermore, the spaces $X_k^e$ and $X_\ell^o$ are orthogonal for all values of $k$ and $\ell$. The collection $\{\barC_{2i,m}^0(Y_{k}^{j})\}_{i=0}^\infty$ lies in $X_k^e$ while $\{\barC_{2i+1,m}^0(Y_{k}^{j})\}_{i=0}^\infty$ lies in $X_k^e$. We may therefore assume that eigenfunctions of $L_c$ will take one of two forms:
\begin{equation}
f_{N,k}(x)=\sum_{j=0}^\infty C_{2i,m}^0(Y_{k}^{j})(x)\alpha_{i,N,m}^{k};\quad g_{N,k}(x)=\sum_{i=0}^\infty C_{2i+1,m}^0(Y_{k}^{j})(x)\beta_{i,N,m}^k.\label{eigenfunction form}
\end{equation}
%Since $\{\barC_{i,2}^0(Y_k):\, i,k\geq 0\}$ constitutes an orthonormal  basis for $L^{2}(\mathbb{R}^2,\mathbb{R}_{2}),$ we may write $\psi_{2N,2}^{k,c}(x)=\sum_{i,k'=0}^{\infty}\barC_{2i,2}^0(Y_{k'})(x)b_{i,N,m}^{k'}$. But it can be concluded that the index $k'$ simplifies only to one $k$. The reason for the latter claim is that the functions $Y_{k}(x)$ are orthonormal. 
An application of  \eqref{twice_application_bonnet} yields
\begin{align}
L_{c}f_{N,k}(x)&=L_{c}\big(\sum_{i=0}^{\infty}\bar C_{2i,m}^0(Y_{k}^{j})(x)\alpha_{i,N,m}^{k}\big)\nonumber\hspace*{8.5cm}\\
%&=\sum_{i=0}^{\infty}L_{c}\big(\barC_{2i,2}^0(Y_k)(x)\big)b_{i,N,m}^{k}\nonumber\\
&=\sum_{i=0}^{\infty}(L_{0}+4\pi^2c^2\vert x\vert^{2})\barC_{2i,m}^0(Y_{k}^{j})(x)\alpha_{i,N,m}^{k}\nonumber\\
%&=\sum_{i=0}^{\infty}\big[ L_{0}(\barC_{2i,2}^0(Y_k)(x))+4\pi^2c^2\vert x\vert^{2}(\barC_{2i,2}^0(Y_k)(x))\big]b_{i,N,m}^{k}\nonumber\\
&=\sum_{i=0}^{\infty}\big[ C(0,2i,m,k)\barC_{2i,m}^0(Y_{k}^{j})(x)+4\pi^2c^2\vert x\vert^{2}\barC_{2i,m}^0(Y_{k}^{j})(x)\big]\alpha_{i,N,m}^{k}\nonumber\\
&=\sum_{i=0}^{\infty}\big[4\pi^2c^2a_{i}\barC_{2i+2,m}^0(Y_{k}^{j})(x)+\big(C(0,2i,m,k)+4\pi^2c^2b_{i}\big)\barC_{2i,m}^0(Y_{k}^{j})(x)\nonumber\\
&\qquad\qquad +4\pi^2c^2c_{i}\barC_{2i-2,m}^0(Y_{k}^{j})(x)\big]\alpha_{i,N,m}^{k}\nonumber\\
&=\sum_{i=0}^{\infty}\barC_{2i,m}^0(Y_{k}^{j})(x)\big[ (4\pi^2c^2 a_{i-1})\alpha_{i-1,N,m}^{k}+\big(C(0,2i,m,k)+4\pi^2c^2b_{i}\big)\alpha_{i,N,m}^{k}\nonumber\\
&\qquad\qquad +(4\pi^2c^2c_{i+1})\alpha_{i+1,N,m}^{k}\big]\nonumber\\
&=\chi_{2N,m}^{k,c}f_{N,k}(x)=\chi_{2N,m}^{k,c}\sum_{i=0}^{\infty}\barC_{2i,m}^0(Y_{k}^{j})(x)\alpha_{i,N,m}^{k},
\end{align}
with the understanding that 
$\alpha_{i,N,m}^{k}=0$
if
$i<0$.
The orthogonality of the Clifford-Legendre polynomials enables us to equate the coefficients on both sides of this equation to obtain the following recurrence formula for $\{\alpha_{i,N,m}^{k}\}_{i=0}^{\infty}$:
$$(4\pi^2c^2 a_{i-1})\alpha_{i-1,N,m}+\big(C(0,2i,2,k)+4\pi^2c^2b_{i}-\chi_{2N,2}^{k,c}\big)\alpha_{i,N,m}+(4\pi^2c^2c_{i+1})\alpha_{i+1,N,m}^{k}=0.$$
This recurrence formula is true for all $i,N\geq 0$ so the problem reduces to finding the eigenvectors $\alpha_{i,N,m}^{k}$ and the eigenvalues $\chi_{2N,m}^{k,c}$ of the doubly-infinite matrix $M_{k,m}^{e}$ with the following entries:
%\begin{align*}
%	(4\pi^2c^2 c_{i+1})=M_{k}^{e}(i,i+1)&=-4\pi^2 c^2\dfrac{(i+1)(\sqrt{k+2i+3})(k+i+1)}{(\sqrt{k+2i+1})(k+2i+3)(k+2i+2)},\\
%	C(0,2i,2,k)+4\pi^2c^2b_{i}=M_{k}^{e}(i,i)&=\big(4i(k+i+1)+4\pi^2c^2(\dfrac{(i+k+1)^2}{(2i+k+2)(2i+k+1)}\\
%	&+\dfrac{i^2}{(2i+k+1)(2i+k)})\big),\\
%	(4\pi^2c^2 a_{i})=M_{k}^{e}(i+1,i)&=-4\pi^2 c^2\dfrac{(i+1)(\sqrt{k+2i+3})(k+i+1)}{(\sqrt{k+2i+1})(k+2i+3)(k+2i+2)}.
%\end{align*}
% This is the 2-dimensional version
%\begin{align*}
%M_k^e(i,j)=\begin{cases}
%-4\pi^2 c^2\dfrac{(i+1)(k+i+1)}{\sqrt{k+2i+1}\sqrt{k+2i+3}(k+2i+2)}&\text{ if $j=i+1$}\\
%4i(k+i+1)+\dfrac{4\pi^2c^2}{2i+k+1}\left(\dfrac{(i+k+1)^2}{2i+k+2}+\dfrac{i^2}{(2i+k)})\right)&\text{ if $i=j$}\\
%-4\pi^2 c^2\dfrac{(i+1)(k+i+1)}{\sqrt{k+2i+1}\sqrt{k+2i+3}(k+2i+2)}&\text{ if $i=j+1$}\\
%0&\text{ else.}
%\end{cases}
%\end{align*}
\begin{eqnarray*}
	M_{k,m}^{e}(i,j)=\left\lbrace \begin{array}{ll}
		-4\pi^{2}c^{2}\frac{i(k+i+\frac{m}{2}-1)}{(k+2i+\frac{m}{2}-1)\sqrt{(k+2i+\frac{m}{2})(k+2i+\frac{m}{2}-2)}},\;\;\;\textnormal{if}\; i\geq 1,\; j=i-1,  \\
		4i(k+i+\frac{m}{2})+\frac{4\pi^2c^2}{(k+2i+\frac{m}{2})}\big[\frac{(k+i+\frac{m}{2})^2}{(k+2i+\frac{m}{2}+1)}+\frac{i^2}{(k+2i+\frac{m}{2}-1)}\big]\;\;\;\textnormal{if}\; i=j\geq 0,\\
		-4\pi^{2}c^{2}\frac{(i+1)(k+i+\frac{m}{2})}{(k+2i+\frac{m}{2}+1)\sqrt{(k+2i+\frac{m}{2}+2)(k+2i+\frac{m}{2})}},\;\;\;\textnormal{if}\; i\geq 0,\; j=i+1,  \\
		0\;\;\;\;\;\;\;\; else.
	\end{array} \right.
\end{eqnarray*}
We see directly that $M_{k,m}^{e}$ is symmetric and banded in the sense that $M_{k,m}^{e}(i,j)=0$ if $|i-j|>1$. The symmetry of $M_{k,m}^{e}$ is due to the self-adjointness of $L_c$.
%We may see that $a_{i}=c_{i+1}$ so by taking $b_{-1,N,m}^{k}=0,$ we can see that $=M_{k}^{even}(i,i+1)= M_{k}^{even}(i+1,i).$ The reason that the matrix should be self-adjoint is this, let
%$\{f_{n}\}$ be an orthonormal basis for $L^{2}(\mathbb{R}^{2},\mathbb{R}_{2})$ for an arbitrary function $f$ we can write $f=\sum_{n=0}^{\infty}f_{n}a_{n}$ where $a_{n}=\langle f,f_{n}\rangle.$
%So
%$$L_{c}f=\sum_{n}(L_{c}f_{n})a_{n}=\sum_{n}(\sum_{m=0}f_{m}M_{mn})a_{n}=\sum_{m}f_{m}(\sum_{n=0}M_{mn}a_{n})=\sum_{m}f_{m}(Ma)_{m}.$$
%And,
%$$\langle f_{p},L_{c}f\rangle=\langle f_{p},\sum_{m}f_{m}(Ma)_{m}\rangle=\sum_{m}\langle f_{p},f_{m}\rangle (Ma)_{m}=(Ma)_{p},$$
%also we have that
%$L_{c}f_{n}=\sum_{m}f_{m}a_{mn}.$
%Therefore, since
%$L_{c}$
%is self-adjoint, then 
%\begin{align*}
%	\langle f_{p},L_{c}f_{n}\rangle&=\langle f_{p},\sum_{m}f_{m}a_{mn}\rangle=\sum_{m}\langle f_{p},f_{m}\rangle a_{mn}=a_{pn}\\
%	&=\langle L_{c}f_{p},f_{n}\rangle=\langle \sum_{m}f_{m}a_{mp},f_{n}\rangle=\sum_{m}\overline{a_{mp}}\langle f_{m},f_{n}\rangle=\overline{a_{np}}.
%\end{align*}
With $g_{N,k}$ as in (\ref{eigenfunction form}), we find that if $g_{N,k}$ is an eigenfunction of $L_c$ with eigenvalue $\chi_{2N+1,m}^{k,c}$ then the vector $\beta_n^k$ with $i$-th entry $\beta_{i,N}^k$ is an eigenvector of the doubly-infinite matrix $M_{k,m}^{o}$ which has $(i,j)$-th entry
% This is the 2-dimensional version
%\begin{align*}
%M_k^o(i,j)=\begin{cases}
%-4\pi^2 c^2\dfrac{(i+1)(k+i+2)}{\sqrt{k+2i+4}\sqrt{k+2i+2}(k+2i+3)}&\text{ if $j=i+1$}\\
%4(i+1)(i+k+1)+\dfrac{4\pi^2c^2}{2i+k+2}\left(\dfrac{(i+k+1)^2}{2i+k+1}+\dfrac{(i+1)^2}{2i+k+3}\right)&\text{ if $i=j$}\\
%-4\pi^2 c^2\dfrac{(i+1)(k+i+2)}{\sqrt{k+2i+4}\sqrt{k+2i+2}(k+2i+3)}&\text{ if $j=i+1$}\\
%0&\text{ else.}
%\end{cases}
%\end{align*}
\begin{eqnarray*}
	M_{k,m}^{o}(i,j)=\left\lbrace \begin{array}{ll}
		-4\pi^{2}c^{2}\frac{i(k+i+\frac{m}{2})}{(k+2i+\frac{m}{2})\sqrt{(k+2i+\frac{m}{2}+1)(k+2i+\frac{m}{2}-1)}},\;\;\;\;\;\;\textnormal{if}\; i\geq 1,\; j=i-1,  \\
		4(i+1)(k+i+\frac{m}{2})+\frac{4\pi^2c^2}{(k+2i+\frac{m}{2}+1)}\big[\frac{(k+i+\frac{m}{2})^2}{(k+2i+\frac{m}{2})}+\frac{(i+1)^2}{(k+2i+\frac{m}{2}+2)}\big]\;\;\;\textnormal{if}\; i=j\geq 0,\\
		-4\pi^{2}c^{2}\frac{(i+1)(k+i+\frac{m}{2}+1)}{(k+2i+\frac{m}{2}+2)\sqrt{(k+2i+\frac{m}{2}+3)(k+2i+\frac{m}{2}+1)}},\;\;\;\textnormal{if}\; i\geq 0,\; j=i+1,  \\
		0\;\;\;\;\;\;\;\; else.
	\end{array} \right.
\end{eqnarray*}
\begin{Remark}\label{Relationship_Coefficients_CPSWFs}
We can check directly that the off-diagonal elements of $M_{k,m}^{e}$ and $M_{k,m}^{o}$ satisfy
\begin{align}
	M_{k+1,m}^{e}(i-1,i)&=M_{k+1,m}^{e}(i,i+1)\notag\\
	&=-4\pi^{2} c^{2}\dfrac{(i+1)(k+i+\frac{m}{2}+1)}{(k+2i+\frac{m}{2}+1)\sqrt{k+2i+\frac{m}{2}+3}\sqrt{k+2i+\frac{m}{2}+1}}\label{even-odd off-diagonal}\\
	%&=-4\pi^2 c^2\dfrac{(i+1)(\sqrt{k+2i+2})(k+i+2)}{(\sqrt{k+2i+4})(k+2i+2)(k+2i+3)}\\
	&=M_{k,m}^{o}(i-1,i)=M_{k,m}^{o}(i,i+1).\notag
\end{align}
However, for the diagonal elements we have 
\begin{equation}
	M_{k,m}^{o}(i,i)-M_{k+1,m}^{e}(i,i)=4k+2m.\label{even-odd diagonal}
\end{equation}
In light of (\ref{even-odd off-diagonal}) and (\ref{even-odd diagonal}) we have
$$M_{k}^{o}=M_{k+1}^{e}+bI,$$
where 
$b=4k+2m$.
Therefore,  if 
$v$
is the eigenvector of matrix of $M_{k,m}^{e}$ with eigenvalue
$\lambda$
then
$v$
is also an eigenvector of $M_{k+1,m}^{o}$ with eigenvalue
$\lambda +b$.
\end{Remark}
%\begin{Remark}\label{Relationship_Coefficients_CPSWFs}
%It's possible to obtain the even and odd matrices in $m$ dimensions and conclude that
%$$ M^{e}_{k+1,m}(i-1,i)=M^{e}_{k+1,m}(i,i+1)=M^{o}_{k,m}(i-1,i)=M^{o}_{k,m}(i,i+1),$$
%and,
%$$M_{k,m}^{o}(i,i)=M_{k,m}^{e}(i,i)+bI$$
%where $b=4k+2m$.
%\end{Remark}

The matrices $M_{k,m}^{e}$ and $M_{k,m}^{o}$ have eigenvectors $\alpha_N$ and $\beta_N$ $(N\geq 0)$ respectively, leading to  eigenfunctions $\psi_{2N,m}^{k,c}$ and $\psi_{2N+1,m}^{k,c}$ of $L_c$ arising from (\ref{eigenfunction form}):
\begin{equation}
\psi_{2N,m}^{k,c}(x)=\sum_{j=0}^\infty C_{2i,m}^0(Y_k)(x)\alpha_{i,N,m}^k;\quad \psi_{2N+1,m}^{k,c}(x)=\sum_{i=0}^\infty C_{2i+1,m}^0(Y_k)(x)\beta_{i,N,m}^k,\label{CPSWF def}
\end{equation}
where $(\alpha_N)_i=\alpha_{i,N,m}^{k}$ and $(\beta_N)_i=\beta_{i,N,m}^{k}$ are the $i$-th entries of the vectors $\alpha_N$ and $\beta_N$ respectively. The relationship between $M_{k,m}^{e}$ and $M_{k,m}^{o}$ outlined above leads to the following relationship between the even and off CPSWFs.
\begin{Th}\label{relation of two dim cl prolate with each other}
	Let	
	$\psi_{2N,2}^{k,c}(x),$
	and,
	$\psi_{2N+1,2}^{k,c}(x),$
	be $L^2$-normalised two-dimensional CPSWFs. Then we have
	\begin{equation}
	\psi_{2N+1,2}^{k,c}(x)=-e_{1}\psi_{2N,2}^{k+1,c}(x).
	\end{equation}
\end{Th}
\begin{proof} We recall equation (\ref{CPSWF def}) and the fact that 
%	By the way that we constructed we have that
%	$$\psi_{2N+1,2}^{k,c}(x)=\sum_{i}b_{i,odd}^{N,k}C_{2i+1,2}^{0}(Y_{k})(x),$$
%	where 
%	$b_{i,odd}^{N,k}$
%	is the eigenvector of the odd matrix. But we know that
	$\beta_{N,m}^k=\alpha_{N,m}^{k+1}$.
	An application of Theorem \ref{Even and Odd Clifford Legendre} yields
$$\psi_{2N+1,2}^{k,c}(x)=\sum_{i=0}^\infty \barC_{2i+1,2}^0(Y_k)(x)\beta_{i,N}^k=\sum_{i=0}^\infty -e_1\barC_{2i,2}^{0}(Y_{k+1})(x)\alpha_{i,N}^{k+1}
		%&=-e_{1}\sum_{i}b_{i,even}^{N,k+1}C_{2i,2}^{0}(Y_{k+1})(x)\\
		=-e_{1}\psi_{2N,2}^{k+1,c}(x).$$
	This completes the proof.
\end{proof}
As we have seen, each even CPSWF has the form $\psi_{2N,2}^{k,c}(x)=R_N^k(|x|^2)Y_k(x)$ with $R_N^k$ a function defined on $[0,\infty )$ and known as the {\it radial part} of $\psi_{2N,2}^{k,c}$. Similarly, the odd CPSWFs have the form $\psi_{2N+1,2}^{k,c}(x)=xS(x)Y_k(x)$ and we define its radial part to be $|x|S(|x|^2)$. Here we plot the radial parts of the some CPSWFs. In Figures \ref{radialktwon}, we see the proficiency of Theorem \ref{relation of two dim cl prolate with each other}. 
%%%%%%%%%%%%%%
\begin{figure}[h]\label{radialk_one_n_one_two}
	\centering
	\begin{tabular}{cc}
		\begin{minipage}[h]{0.5\textwidth}
			\includegraphics[width=5.5cm]{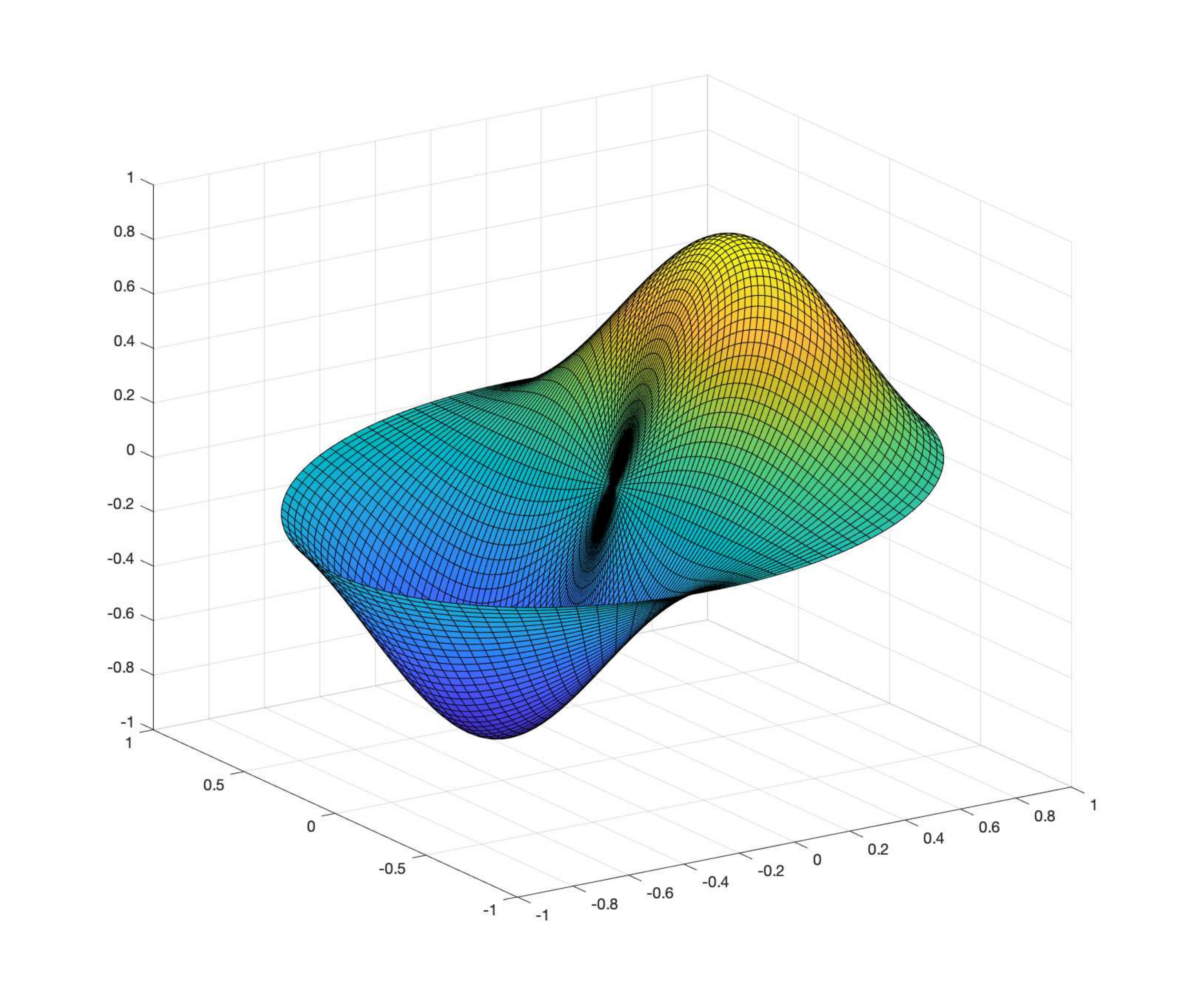}
		\end{minipage} &
		\begin{minipage}[h]{0.5\textwidth}
			\includegraphics[width=5.5cm]{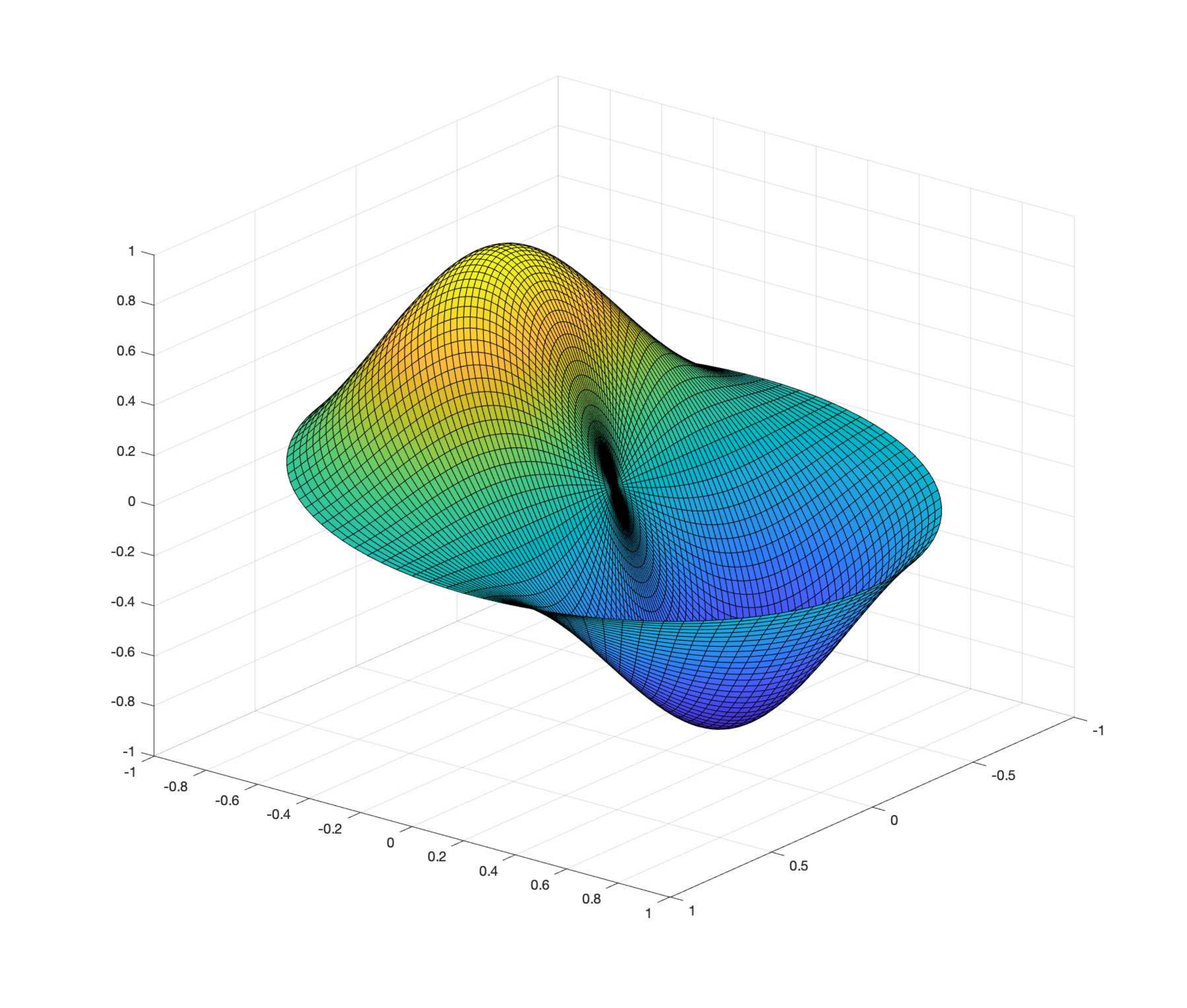}
		\end{minipage}\\
		\begin{minipage}[h]{0.5\textwidth}
			(a) Graph of $e_1$ part of $\psi_{0,2}^{1,1}$ on $B(1)$.
		\end{minipage} & 
		\begin{minipage}[h]{0.5\textwidth}
			(b) Graph of $e_2$ part of $\psi_{0,2}^{1,1}$ on $B(1)$.
		\end{minipage}\\
	\end{tabular}
	\caption{Plots of CPSWFs for $k=1,\; c=1,\; n=0$.}
\end{figure}

%%%%%%%%%%%%%%
\begin{figure}[h]\label{radialktwon}
	\centering
	\begin{tabular}{cc}
		\begin{minipage}[h]{0.5\textwidth}
			\includegraphics[width=5.5cm]{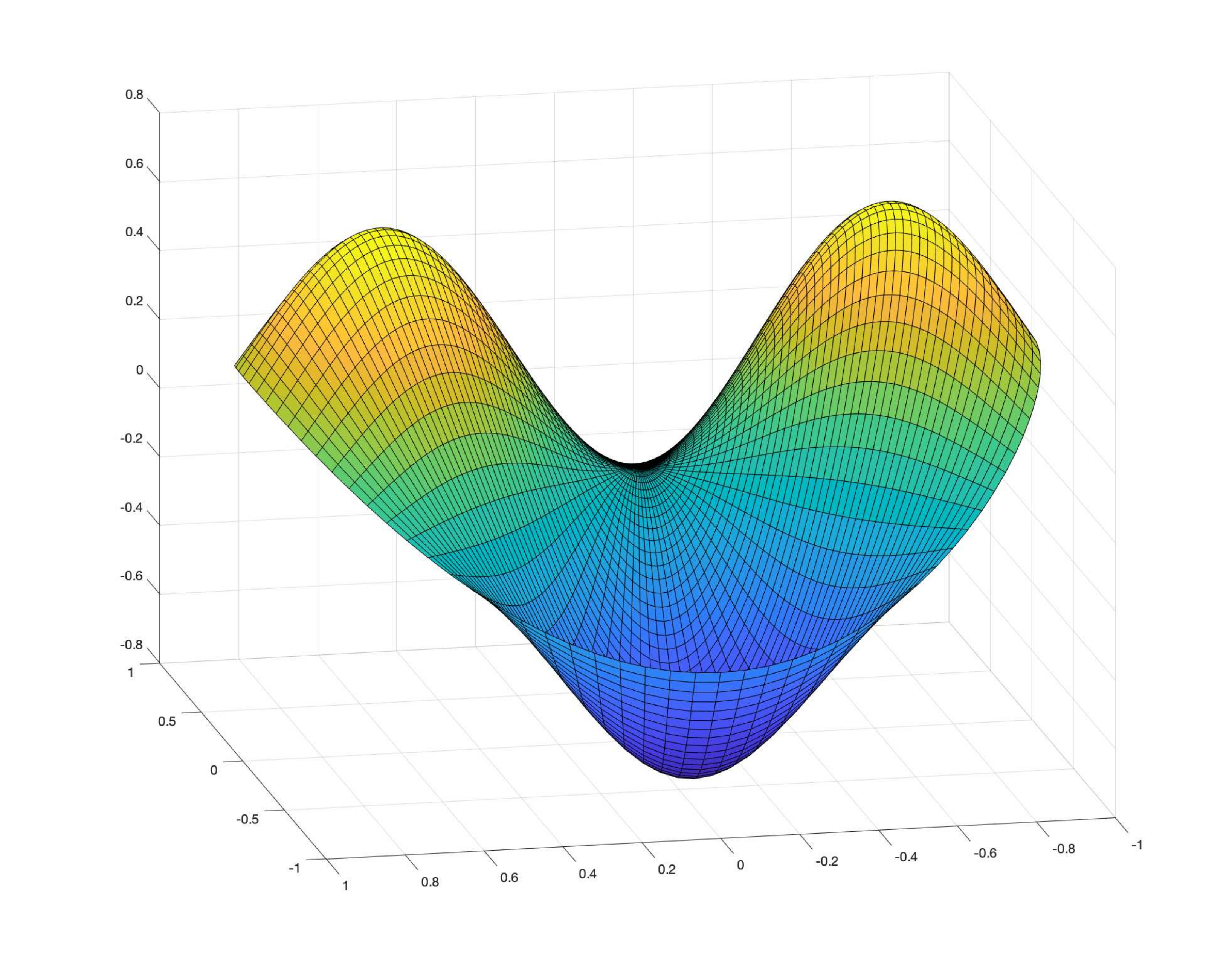}
		\end{minipage} &
		\begin{minipage}[h]{0.5\textwidth}
			\includegraphics[width=5.5cm]{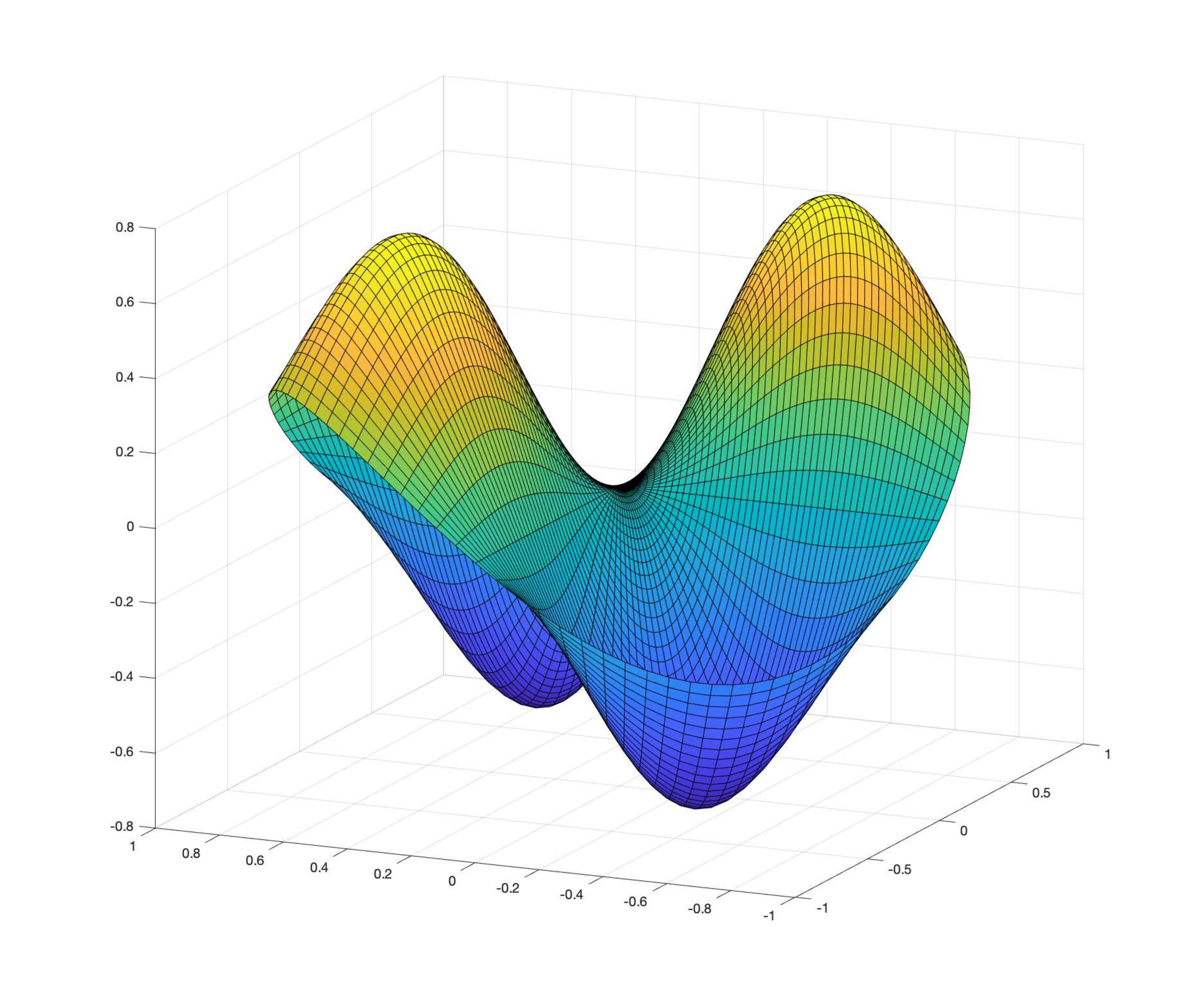}
		\end{minipage}\\
		\begin{minipage}[h]{0.45\textwidth}
			(a) Graph of $e_0$ part of $\psi_{1,2}^{1,1}$ on $B(1)$.
		\end{minipage} & 
		\begin{minipage}[h]{0.45\textwidth}
			(b) Graph of $e_{12}$ part of $\psi_{1,2}^{1,1}$ on $B(1)$.
		\end{minipage}\\
	\end{tabular}
	\caption{Plots of CPSWFs for $k=1,\; c=1,\; n=1$.}
	\label{Figure1}
\end{figure}
We end this section with an observation about the eigenvalues of the operator $L_c$.

\begin{Proposition} The eigenvalues of $L_c$ are positive real numbers.
\end{Proposition}
\begin{proof}
We have already seen that $L_c$ is self-adjoint. Furthermore, the eigenvalues of $L_c$ are the eigenvalues of the symmetric matrices $M_k^e$ and $M_k^o$, and are therefore real numbers. On the other hand, $L_c$ is a positive operator in the sense that 
$$\langle L_cf,f\rangle =[(L_cf,f)]_0\geq 0.$$
To see this, note that for each continuously differentiable ${\mathbb C}_m$-valued function defined on $\Omega\subset{\mathbb R}^m$ we have $\overline{\partial f}=-\overline{f}\partial$. Therefore, by the Clifford-Stokes theorem,
\begin{align*}
\langle L_0f,f \rangle&=\int\limits_{B(1)}\overline{\partial[(1-|x|^2)\partial f(x)]}f(x)\, dx\\
&=-\int\limits_{B(1)}(\overline{[(1-|x|^2)\partial f(x)]}\partial )f(x)\, dx\\
&=-\int\limits_{\partial B(1)}(1-|x|^2)\overline{\partial f(x)}\frac{x}{|x|}f(x)+\int\limits_{B(1)}(1-|x|^2)\overline{\partial f(x)}\partial f(x)\, dx\\
&=\int\limits_{B(1)}(1-|x|^2)\overline{\partial f(x)}\partial f(x)\, dx.
\end{align*}
Consequently,
\begin{align*}
\langle L_cf,f\rangle&=\int\limits_{B(1)}(1-|x|^2)[\overline{\partial f(x)}\partial f(x)]_0\, dx
+4\pi^2c^2\int\limits_{B(1)}|x|^2[\overline{f(x)}f(x)]_0\, dx\\
&=\int\limits_{B(1)}(1-|x|^2)|\partial f(x)|^2\, dx+4\pi^2\int\limits_{B(1)}|x|^2|f(x)|^2\, dx\geq 0.
\end{align*}
 If $\lambda $ is a (real) eigenvalue of $L_0$ associated with an eigenfunction $f$, then 
 $$\lambda\|f\|_2^2=\lambda\langle f,f\rangle =\langle \lambda f,f\rangle=\langle L_cf,f\rangle\geq 0,$$
 from which we conclude that $\lambda\geq 0$. 
\end{proof}

%%%%%%%%%%%%%%
\section{The Sturm-Liouville Properties of the CPSWFs}
In \cite{propertiesofcliffordlegendre} it was shown that
$C^{0}_{2N,m}(Y_{k})(x)=P_{N,k,m}(\vert x\vert^{2})Y_{k}(x),$
where 
$P_{N,k,m}(\vert x\vert^{2})$ 
(the radial part of the Clifford-Legendre polynomial) is real-valued. From \eqref{differential_equation_of_CL_Poly} we have
$$L_{0}C_{n,m}^{0}(Y_{k})(x)=C(0,n,m,k)C_{n,m}^{0}(Y_{k})(x),$$
Putting  $R_{N,k,m}(s)=P_{N,k,m}(\vert x\vert^{2})$ with $\vert x\vert^{2}=\dfrac{s+1}{2}$ ($s\in (-1,1)$)
gives
$$S_0R_{N,k,m}=-\frac{C(0,2N,m,k)}{4}R_{N,k,m},$$
where $S_0$ is the differential operator
\begin{equation}\label{S_0 def}
S_0=(1-s^2)\frac{d^2}{ds^2}+\bigg[\bigg(k+\frac{m}{2}-1\bigg)-\bigg(\frac{m}{2}+k+1\bigg)s\bigg]\frac{d}{ds}.
\end{equation}
We have a similar result for eigenfunctions $\psi_{2N,m}^{k,c}(x)=P_{N,m}^{k,c}(|x|^2)Y_k(x)$ of the operator $L_c=L_0+M_c$, namely,
%$L_c\psi_{2N,m}^{k,c,i}(x)=L_0\psi_{2N,m}^{k,c,i}(x)+4\pi^2c^2\vert x\vert^2\psi_{2N,m}^{k,c,i}(x),$ where 
%$\psi_{2N,m}^{k,c,i}(x)=P_{N,m}^{k,c}(\vert x\vert^{2})Y_{k}^{i}(x).$
%So we have that 
$$t(1-t)\frac{d^{2}}{dt^{2}}P_{N,m}^{k,c}(t)+\frac{1}{2}(m+2k-t(2+m+2k))\frac{d}{dt}P_{N,m}^{k,c}(t)-\pi^2c^2tP_{N,m}^{k,c}(t)=-\frac{\chi_{2N,m}^{k,c}}{4}P_{N,m}^{k,c}(t),$$
where $\vert x\vert^{2}=t$. Putting $\tilde{P}_{N,m}^{k,c}(s)=P_{N,m}^{k,c}(2t-1)$ $(s\in (-1,1))$, gives 
%$$(1-s^{2})Q''_{N,k,m}(s)+[-s(k+\frac{m}{2}+1)+(k+\frac{m}{2}-1)]Q'_{N,k,m}(s)-\pi^{2}c^2(\frac{s+1}{2})Q_{N,k,m}(s)+\frac{\chi_{2N,m}^{k,c}}{4}Q_{N,k,m}(s)=0,$$
%where $Q_{N,k,m}(s)=P_{N,m}^{k,c}(t),$ that is
\begin{equation}
S_c\tilde P_{N,m}^{k,c}(s):=S_0\tilde P_{N,m}^{k,c}(s)-\frac{\pi^2c^2}{2}(s+1)\tilde P_{N,m}^{k,c}(s)=-\frac{\chi_{2N,m}^{k,c}}{4}\tilde P_{N,m}^{k,c}(s).\label{nonSL_for_CPSWFs}
\end{equation}
%(1-s^{2})\frac{d^2}{ds^2}\tilde Q_{N,m}^{k,c}(s)&+\left[-s\left(k+\frac{m}{2}+1\right)+\left(k+\frac{m}{2}-1\right)\right]\frac{d}{ds}\tilde Q_{N,m}^{k,c}(s)\notag\\
%&\qquad +(-\frac{\pi^{2}c^2}{2}s-\frac{\pi^2c^2}{2}+\frac{\chi_{2N,m}^{k,c}}{4})Q_{N,m}^{k,c}(s)=0.\label{nonSL_for_CPSWFs}
%\end{align}
The operator determined by equation 
\eqref{nonSL_for_CPSWFs}
is not self-adjoint. In order to make it so, we multiply both sides by  $y=y(s)=(s+1)^{k+\frac{m}{2}-1}$ 
%for which
%$$((1-s^2)y(s))'=\left[-s\left(k+\frac{m}{2}+1\right)+\left(k+\frac{m}{2}-1\right)\right]y(s).$$
%Such a function is $y(s)=(s+1)^{k+\frac{m}{2}-1}$.
%So for finding that we have 
%\begin{eqnarray*}
%	&\Rightarrow& -2sy(s)+(1-s^2)y'(s)=[-s(k+\frac{m}{2}+1)+(k+\frac{m}{2}-1)]y(s)\\
%	&\Rightarrow& (1-s^2)y'(s)=[-s(k+\frac{m}{2}-1)-(k+\frac{m}{2}-1)]y(s)\\
%	&\Rightarrow& (1-s^2)y'(s)=y(s)(k+\frac{m}{2}-1)(1-s)\\
%	&\Rightarrow& \ln \vert y(s)\vert=\ln \vert 1+s\vert^{(k+\frac{m}{2}-1)}\\
%	&\Rightarrow& y(s)=(s+1)^{(k+\frac{m}{2}-1)}
%\end{eqnarray*}
%As a result 
%$$y(s)=(s+1)^{(k+\frac{m}{2}-1)}.$$
%olution for finding $y(s).$)
%Now, we multiply \eqref{nonSL_for_CPSWFs} by  
%$y(s)$ 
to obtain
$$\frac{d}{ds}\left(p(s)\frac{d}{ds}\tilde P_{N,m}^{k,c}\right)-\frac{\pi^2c^2}{2}(s+1)^{k+\frac{m}{2}}\tilde P_{N,m}^{k,c}(s)=-\frac{\chi_{2N,m}^{k,c}}{4}(s+1)^{k+\frac{m}{2}-1}\tilde P_{N,m}^{k,c}(s).$$
%$$p(s)\frac{d^2}{ds^2}Q_{N,m}^{k,c}(s)+p'(s)\frac{d}{ds}Q_{N,m}^{k,c}(s)+\left(-\frac{\pi^{2}c^2}{2}s-\frac{\pi^2c^2}{2}+\frac{\chi_{2N,m}^{k,c}}{4}\right)y(s)Q_{N,m}^{k,c}(s)=0,$$
%where 
%$p(s)=y(s)(1-s^2)=(1+s)^{(k+\frac{m}{2})}(1-s)$
%Equivalently
%\begin{align}\label{final_SL_for_CPSWFs}
%(1+s)^{k+\frac{m}{2}}(1-s)\frac{d^2}{ds^2}Q_{N,m}^{k,c}(s)&+\left[-s\left(k+\frac{m}{2}+1\right)+\left(k+\frac{m}{2}-1\right)\right](s+1)^{k+\frac{m}{2}-1}\frac{d}{ds}Q_{N,m}^{k,c}(s)\nonumber\\
%&+\left(-\frac{\pi^{2}c^2}{2}s-\frac{\pi^2c^2}{2}+\frac{\chi_{2N,m}^{k,c}}{4}\right)(s+1)^{k+\frac{m}{2}-1}Q_{N,m}^{k,c}(s)=0,\nonumber\\
%\end{align}
which is a Sturm-Liouville differential equation. From Definition \ref{def_SL} and Theorem \ref{firstSL} we can conclude the following theorem.

\begin{Th}\label{th: Tctildeoperator}
	The differential operator 
	$\tilde{T_{c}}$
	defined by
	\begin{equation}\label{Tctildeoperator}
	\tilde{T_{c}}f(s)=\big(4(s+1)^{k+\frac{m}{2}}(1-s)f'(s)\big)'-2\pi^2c^2(s+1)^{k+\frac{m}{2}}f(s)
	\end{equation}
	is self-adjoint and the differential equation (\ref{nonSL_for_CPSWFs}) can be written as 
	$$\tilde{T_{c}}\tilde P_{N,m}^{k,c}(s)+\chi_{2N,m}^{k,c}g(s)\tilde P_{N,m}^{k,c}(s)=0,$$
	where
	$g(s)=(s+1)^{k+\frac{m}{2}-1}.$
	This is a singular Sturm-Liouville equation, i.e.,
	with $p(s)=4(s+1)^{k+\frac{m}{2}}(1-s)=0$
	when 
	$s=\pm1.$ Therefore, the eigenvalues of
	$\tilde{T_{c}}$
	are distinct and may be labelled so that 
	$\chi_{0,m}^{k,c}<\chi_{2,m}^{k,c}<\dots,$ 
	and the eigenfunctions $\{\tilde P_{N,m}^{k,c}\}_{N=0}^\infty$ may be normalised so that 
	$$( \tilde P_{N,m}^{k,c},\tilde P_{N',m}^{k,c})_{g(s)}=\int\limits_{-1}^{1}\tilde P_{N,m}^{k,c}(s)\tilde P_{N',m}^{k,c}(s)g(s)\, ds=\delta_{N,N'}.$$
\end{Th}

\begin{Lemma}\label{orthogonality_of_even_and_even_CPSWFs}
For each integer $k\geq 0$, let $\{Y_k^i\}_{i=1}^{d_k}$ be an orthonormal basis for $M_l^+(k)$. The even eigenfunctions of the operator
$L_{c},$ i.e, $\{\psi_{2N,m}^{k,c,i}:\, k\geq 0,\ 1\leq i\leq d_k\}$ are orthogonal.
\end{Lemma}	
\begin{proof} Note that 
\begin{align*}
	\int\limits_{B(1)}\overline{\psi_{2N,m}^{k,c,i}(x)}\psi_{2N',m}^{k',c,i'}(x)\, dx
	&=\int\limits_{{B(1)}}P_{N,m}^{k,c}(\vert x\vert^{2})P_{N',m}^{k',c}(\vert x\vert^{2})\overline{Y_{k}^{i}(x)}Y_{k'}^{i'}(x)\, dx\\
	&=\int\limits_{0}^{1}P_{N,m}^{k,c}(r^{2})P_{N',m}^{k',c}(r^{2})r^{m+k+k'-1}\int\limits_{S^{m-1}}\overline{Y_{k}^{i}(\omega)}Y_{k'}^{i'}(\omega )\, d\omega\, dr\\
	&=\int\limits_{0}^{1}P_{N,m}^{k,c}(r^{2})P_{N',m}^{k,c}(r^{2})r^{m+2k-1}\,dr\delta_{kk'}\delta_{ii'}\\
	&=\frac{1}{2}\int\limits_{0}^{1}P_{N,m}^{k,c}(t)P_{N',m}^{k,c}(t)t^{k+\frac{m}{2}-1}\, \,dt\delta_{kk'}\delta_{ii'}\\
	&=\frac{1}{4}\int\limits_{-1}^{1}Q_{N,m}^{k,c}(s)Q_{N',m}^{k,c}(s)\left(\frac{s+1}{2}\right)^{k+\frac{m}{2}-1}\, ds\, \delta_{kk'}\delta_{ii'}\\
%	&=\frac{1}{2^{k+\frac{m}{2}+1}}\langle Q_{N,m}^{k,c}(s),Q_{N',m}^{k,c}(s)\rangle_{g(s)}\delta_{kk'}\delta_{ii'}\\
	&=\frac{1}{2^{k+\frac{m}{2}+1}}\delta_{N,N'}\delta_{kk'}\delta_{ii'}.
\end{align*}
\end{proof}
Here, we mention that we can obtain a Sturm-Liouville problem for the functions $P_{N,m}^{k,c}$ $(N\geq 0)$ without the change of variable.

\begin{Lemma}\label{even_radial_orthogonal_without_change_variable}
The radial part of $P_{N,m}^{k,c}(|x|^2)$ of the CPSWF $\psi_{2N,2}^{k,c,i}(x),$  satisfies
\begin{equation}
4t(1-t)\frac{d^{2}}{dt^{2}}P_{N,m}^{k,c}(t)+2(m+2k-t(2+m+2k))\frac{d}{dt}P_{N,m}^{k,c}(t)-4\pi^2c^2tP_{N,m}^{k,c}(t)+\chi_{2N,m}^{k,c}P_{N,m}^{k,c}(t)=0,\label{SL even}
\end{equation}
which becomes a  Sturm-Liouville differential equation after multiplying by  $g(t)=t^{k+\frac{m}{2}-1}.$ Therefore, $\{P_{N,m}^{k,c}\}_{N=0}^\infty$ may be normalised so that 
$$\int\limits_{0}^{1}P_{N,m}^{k,c}(t)P_{M,m}^{k,c}(t)t^{k+\frac{m}{2}-1}dt=(P_{N,m}^{k,c},P_{M,m}^{k,c})_{g(t)}=\delta_{MN}.$$
Furthermore, for each fixed integer $k\geq 0$, the collection $\{P_{N,m}^{k,c}:\, N\geq 0\}$ is complete in the weighted $L^2$ space $L^2([0,1], t^{k+\frac{m}{2}-1})$ of measurable functions $f:[0,1]\to{\mathbb R}$ for which $\int\limits_0^1|f(t)|^2t^{k+\frac{m}{2}-1}\, dt <\infty$.
\end{Lemma}

\begin{Lemma}\label{span_even_CPSWFs}
If $f$ is as in \eqref{Representation_f_all_monogenics_equation} and is supported on ${B(1)}$ and $(\psi_{2N,m}^{k,c,i},f)=0$ for all $N\geq 0,\; k\geq 0,\; 1\leq j\leq d_{k},$ where $d_{k}$ is the dimension of $M_{l}^{+}(k).$ Then $f_{k}^{(i)}=0,$ for all $k\geq 0,\; 1\leq j\leq d_{k}.$ 
\end{Lemma}
\begin{proof}
Suppose $f\in L^2({\mathbb R}^m,{\mathbb R}_m)$ admits the expansion  \eqref{Representation_f_all_monogenics_equation} and is supported on ${B(1)}.$ By  Lemmas \ref{property of two monogenic and x between} and \ref{support_Lemma_f_summand}, we have
\begin{align*}
0=(\psi_{2N,m}^{k,c,i},f)&=\int\limits_{{B(1)}}\overline{\psi_{2N,m}^{k,c,i}(x)}f(x)\, dx\\
&=\int\limits_{{B(1)}}P_{N,m}^{k,c}(\vert x\vert^{2})\overline{Y_{k}^{i}(x)}f(x)\, dx\\
&=\int\limits_{{B(1)}}P_{N,m}^{k,c}(\vert x\vert^{2})\overline{Y_{k}^{i}(x)}\sum_{k'=0}^{\infty}\sum_{l=0}^{d_{k'}}f_{k'}^{(l)}(\vert x\vert)Y_{k'}^{(l)}(x)\, dx\\
&=\sum_{k'=0}^{\infty}\sum_{l=1}^{d_{k'}}\int\limits_{0}^{1}P_{N,m}^{k,c}(r^{2})f_{k'}^{(l)}(r)r^{k+k'+m-1}\int\limits_{S^{m-1}}\overline{Y_{k}^{i}(\omega)}Y_{k'}^{(l)}(\omega)\, d\omega\, dr\\
&=\int\limits_{0}^{1}P_{N,m}^{k,c}(r^{2})f_{k}^{(i)}(r)r^{2k+m-1}\, dr\\
&=2\int\limits_{0}^{1}P_{N,m}^{k,c}(t)f_{k}^{(i)}(\sqrt{t})t^{k+\frac{m}{2}-1}\, dt.
\end{align*}
Hence by Lemma \ref{even_radial_orthogonal_without_change_variable} $f_{k}^{(i)}=0,$ for all $k$ and $i$.
\end{proof}
We may conclude similar results for the odd CPSWFs. 
\begin{Th}
	The radial parts $Q_{M,n}^{k,c}$ of the odd CPSWFs $C_{2N+1,m}^{k,c}(Y_k)$ satisfy the following Sturm-Liouville differential equation
\begin{align*}
&(1-s)^2\frac{d^2}{ds^2}Q_{N,m}^{k,c}(s)+\left[-s\left(k+\frac{m}{2}+2\right)+\left(k+\frac{m}{2}\right)\right]\frac{d}{ds}Q_{N,m}^{k,c}(s)\\
&\qquad\qquad +\left[-\left(\frac{m}{2}+k\right)-\frac{\pi^{2}c^2}{2}(1+s)+\frac{\chi_{2N+1,m}^{k}}{4}\right]Q_{N,m}^{k,c}(s)=0.\end{align*}
	In fact, the differential operator 
	$\tilde{U_{c}}$
	defined by
	\begin{equation*}
	\tilde{U_{c}}f(s)=\big(4(s+1)^{k+\frac{m}{2}+1}(1-s)f'(s)\big)'+4(s+1)^{k+\frac{m}{2}}(-\frac{\pi^{2}c^{2}}{2}s-\frac{\pi^{2}c^{2}}{2})f(s).
	\end{equation*}
	is self-adjoint and can be written,
\begin{equation}\label{SL_odd}
\tilde{T_{c}}Q_{N,m}^{k,c}(s)+\chi_{2N+1,m}^{k,c}g(s)Q_{N,m}^{k,c}(s)=0,
\end{equation}
	where
	$g(s)=(s+1)^{k+\frac{m}{2}}.$
	The differential equation in (\ref{SL_odd}) is singular Sturm-Liouville with 
	$p(s)=4(s+1)^{k+\frac{m}{2}+1}(1-s)$
	which takes the value $0$ when 
	$s=\pm1.$
	Therefore, the eigenvalues of the
	$\tilde{U_{c}}$
	are distinct  and may be ordered so that 
	$\chi_{1,m}^{k,c}<\chi_{3,m}^{k,c}<\dots$ 
Furthermore, the eigenfunctions $Q_{N,m}^{k,c}$ may be normalised so that 
	$$(Q_{N,m}^{k,c},Q_{N',m}^{k,c})_{g(s)}=\int\limits_{-1}^{1}\overline{Q_{N,m}^{k,c}(s)}Q_{N',m}^{k,c}(s)g(s)\, ds=\delta_{N,N'},$$
	and are complete in the weighted Lebesgue space $L^2([-1,1],g)$ of measurable functions $f:[-1,1]\to{\mathbb R}$ for which $\int\limits_0^1|f(s)|^2g(s)\, ds<\infty$.
\end{Th} 
\begin{Lemma}\label{odd_radial_orthogonal_without_change_variable}
	The radial part $Q_{N,m}^{k,c}(|x|^2)$ of the odd CPSWFs $\psi_{2N+1,2}^{k,c,i}(x),$ satisfy 
\begin{align}
4t(1-t)\frac{d^{2}}{dt^{2}}Q_{N,m}^{k,c}(t)&+2(m+2k+2-t(4+m+2k))\frac{d}{dt}Q_{N,m}^{k,c}(t)\notag\\
&-(4\pi^2c^2t+2(m+2k))Q_{N,m}^{k,c}(t)+\chi_{2N+1,m}^{k,c}Q_{N,m}^{k,c}(t)=0,
\end{align}
	which is a Sturm-Liouville differential equation after multiplication of both sides by  $g(t)=t^{k+\frac{m}{2}}.$ Therefore,
	$$\int\limits_{0}^{1}Q_{N,m}^{k,c}(t)Q_{M,m}^{k,c}(t)t^{k+\frac{m}{2}}dt=(Q_{N,m}^{k,c},Q_{M,m}^{k,c})_{g}=\delta_{MN},$$
	and for each integer $k\geq 0$, the collection $\{Q_{N,m}^{k,c}:\, N\geq 0\}$ is complete in the weighted Lebesgue space $L^2([0,1], t^{k+\frac{m}{2}})$ of measurable functions $f:[0,1]\to{\mathbb R}$ for which $\int\limits_0^1|f(t)|^2t^{k+\frac{m}{2}}\, dt<\infty$.
\end{Lemma}

\begin{Th}\label{orthogonality_of_odd_and_odd_CPSWFs}
	The odd eigenfunctions of
	$L_{c},$ i.e, $\{\psi_{2N+1,m}^{k,c,i}:\, N\geq 0,\ k\geq 0,\ 1\leq i\leq d_k\}$ are orthogonal in $L^2(B(1), {\mathbb R}_m)$.
\end{Th}

\begin{Lemma}\label{span_odd_CPSWFs}
	If $f$ is as in \eqref{Representation_f_all_monogenics_equation} and is supported on ${B(1)}$ and $(\psi_{2N+1,m}^{k,c,i},f)=0$ for all $N\geq 0,\; k\geq 0,\; 1\leq j\leq d_{k}$, then $g_{k}^{(i)}=0,$ for all $k\geq 0,\; 1\leq i\leq d_{k}.$ 	
\end{Lemma}
Using the fact that $\psi_{2N,m}^{k,c,i}=P_{N,m}^{k,c}(\vert x\vert^{2})Y_{k}^{i}(x)$ and $\psi_{2N+1,m}^{k,c,i}=Q_{N,m}^{k,c}(\vert x\vert^{2})xY_{k}^{i}(x)$, we conclude form Lemma \ref{property of two monogenic and x between}, we can conclude that the  the even and odd CPSWFs span orthogonal spaces in $L^2(B(1),{\mathbb R}_m)$.

\begin{Proposition}\label{orthogonality_of_even_and_odd_CPSWFs}
The even and odd eigenfunctions of the $L_{c},$ are orthogonal, i.e.,
$$\int\limits_{B(1)}\overline{\psi_{2N,m}^{k,c,i}(x)}\psi_{2N'+1,m}^{k',c,i'}(x)\, dx=0,$$
for all $N,N'\geq 0$, $k,k'\geq 0$, $1\leq i\leq d_k$, $1\leq i'\leq d_{k'}$.
\end{Proposition}
Now we prove that the CPSWFs are not only orthonormal but also a basis for $L^{2}({B(1)},\mathbb{R}_{m}).$

\begin{Th}\label{CPSWFs_are_basis_using_SL}
	The set $\{\psi_{2N,m}^{k,c,i}:\  N\geq 0,\, k\geq 0, \, 1\leq j\leq  d_k \}\cup \{\psi_{2N+1,m}^{k,c,i}:\,  N\geq 0,\, k\geq 0, \, 1\leq j\leq  d_k \},$ is an orthonormal basis for $L^{2}({B(1)},\mathbb{R}_{m}).$
\end{Th}
\begin{proof}
From Proposition \ref{orthogonality_of_even_and_odd_CPSWFs} and Lemmas \ref{orthogonality_of_even_and_even_CPSWFs} and \ref{orthogonality_of_odd_and_odd_CPSWFs}, we see that $\{\psi_{n,m}^{k,c,i}:\, n\geq 0,\ k\geq 0,\ 1\leq i\leq d_k\}$ is orthogonal in $L^2(B,{\mathbb R}_m)$. Completeness is a consequence of  Lemmas \ref{span_even_CPSWFs} and \ref{span_odd_CPSWFs}.
\end{proof}

\begin{Th}
	The radial part $P_{N,m}^{k,c}(|x|^2)$ of 
	$\psi_{2N,m}^{k,c,i}(x)$ and the radial part $Q_{N,m}^{k,c}(|x|^2)$ of $\psi_{2N+1,m}^{k,c,i}(x)$ satisfy 
	$$P_{N,m}^{k,c}(1)\neq 0\neq Q_{N,m}^{k,c}(1).$$
\end{Th}
\begin{proof}
With $Q=Q_{N,m}^{k,c}$, we assume that
	$Q(1)=0$. Putting $s=1$ in (\ref{nonSL_for_CPSWFs}) gives $-2Q'(1)=0$ from which we conclude that $Q'(1)=0$.
%	so we have that
%	\begin{align}
%	(1-s^{2})Q''_{N,k,m}(s)&+[-s(k+\frac{m}{2}+1)+(k+\frac{m}{2}-1)]Q'_{N,k,m}(s)\nonumber\hspace*{5cm}\\
%	&+(\frac{-\pi^{2}c^{2}s}{2}-\frac{\pi^{2}c^{2}}{2}+\frac{\lambda}{4})Q_{N,k,m}(s)=0,
%	\end{align}
%	if
%	$Q_{N,k,m}(1)=0$
%	then for 
%	$s=1,$
%	we have 
%	$(1-s^2)=0,$ so
%	$$\big(-(k+\frac{m}{2}+1)+(\frac{m}{2}+k-1)\big)Q'_{N,k,m}(1)=0,$$
%	$$\Rightarrow Q'_{N,k,m}(1)=0.$$
	Taking the $n$-th derivative of (\ref{nonSL_for_CPSWFs}) and evaluating both sides at $s=1$ gives $-(2n+2)Q^{(n+1)}(1)=0$, from which we conclude that $Q^{(n+1)}(1)=0$. Hence, by induction we have that $Q^{(\ell )}(1)=0$ for all $\ell \geq 0$. 
But  
	$Q$
	is a solution of (\ref{nonSL_for_CPSWFs}) which has the analytic solution,
	$$Q(s)=\sum_{n=0}^{\infty}\frac{Q^{(n)}(1)}{n!}(s-1)^{n},$$
	on $(-1,1)$. Since,
	$Q^{(n)}(1)=0$
	for all $n\geq 0$, we have 
	$Q(s)\equiv 0$.
	We conclude that
	$Q(1)\neq 0$.
	The proof that $P(1)=P_{N,m}^{k,c}(1)\neq 0$ is similar.
\end{proof}
\begin{Proposition}
	The even eigenvalues of the operator 
	$L_{c},$
	i.e.,
	$\chi_{2N,m}^{k,c}$
	are also the eigenvalues of
	$\tilde{T_{c}},$
	which is defined on 
	\eqref{Tctildeoperator}. This is true for the odd case too.
\end{Proposition}
\begin{proof} 
This is essentially the content of Theorem \ref{th: Tctildeoperator}.	
%We just go through the even case. So
%	$\psi_{2N,m}^{k,c,i}(x)=P_{N,m}^{k,c}(\vert x\vert^2)Y_{k}^i(x).$
%	Thus,
%	$$L_{c}(\psi_{N,m}^{k,c,i}(x))=L_{c}(P_{N,m}^{k,c}(\vert x\vert^2)Y_{k}^i(x))=\tilde{T_{c}}(P_{N,m}^{k,c}(\vert x\vert^2))Y_{k}^i(x)=\chi_{2N,m}^{k,c}P_{N,m}^{k,c}(\vert x \vert^2)Y_{k}^i(x).$$
%%	But we know that
%%	$$\int\limits_{\partial \overline{B(r)}}f(y)d\sigma(y)=r^{m-1}\int\limits_{S^{m-1}}f(r\omega)d\omega,$$
%%	for
%%	$0<r<1,\; y=r\omega,\; r=\vert y\vert.$
%%	So in 
%	Multiplying both sides of the equation $\tilde{T_{c}}(P_{N,m}^{k,c}(\vert x\vert^2))Y_{k}^i(x)=\chi_{2N,m}^{k,c}P_{N,m}^{k,c}(\vert x \vert^2)Y_{k}^i(x)$
%	 by
%	$\overline{Y_k^i(x)}$
%	and integrating
%	\begin{eqnarray*}
%		&&\int\limits_{\partial \overline{B(r)}}\tilde{T_{c}}(P_{N,m}^{k,c}(\vert x\vert^2))\vert Y_{k}(x)\vert^{2}dx=\chi_{2N,m}^{k,c}\int\limits_{\partial \overline{B(r)}}P_{N,m}^{k,c}(\vert x\vert^2)\vert Y_{k}(x)\vert^{2}dx\\
%		&\Rightarrow& r^{m-1}\int\limits_{S^{m-1}}\tilde{T_{c}}(P_{N,m}^{k,c}(r^2))r^{2k}\vert Y_{k}(\omega)\vert^{2}d\omega=\chi_{2N,m}^{k,c}\int\limits_{S^{m-1}}P_{N,m}^{k,c}(r^2)r^{2k}\vert Y_{k}(\omega)\vert^{2}d\omega.
%	\end{eqnarray*}
%	Therefore, when
%	$r\neq 0,$
%	then
%	$$\tilde{T_{c}}(P_{N,m}^{k,c}(r^2))=\chi_{2N,m}^{k,c}P_{N,m}^{k,c}(r^2).$$
\end{proof}

According to the \ref{def_SL}, \ref{firstSL}, and, \ref{secondSL}, we may conclude some of the features of the SL for the radial parts of the even and odd CPSWFs.
\begin{Corollary}
For any fixed $k$, the eigenvalues $\{\chi_{2N,m}^{k,c}\}_{N=0}^\infty$ are distinct, as are the eigenvalues $\{\chi_{2N+1,m}^{k,c}\}_{N=0}^\infty$.
\end{Corollary}
As a consequence of Sturm-Liouville theory
\begin{Corollary}
For any fixed $k$, the zeros of the $\{\psi_{2N,m}^{k,c,i}\}_{N=0}^\infty$ are interlacing as are the zeroes of $\{\psi_{2N+1,m}^{k,c,i}\}_{N=0}^\infty$.
\end{Corollary}
%\begin{Remark}
%In section 6, we will prove that the all two dimension CPSWFs are orthogonal.
%\end{Remark}

\section{CPSWFs as the eigenfunctions of the Finite Fourier Transformation}
In this section, we will present the finite Fourier transformation $\mathcal{G}_{c}$ and then prove that the CPSWFs are the eigenfunctions of this operator.
\begin{Definition}
	We define
	$\mathcal{G}_{c}$
	from 
	$L^{2}({B}(1),\mathbb{C}_{m})$
	to
	$L^{2}({\mathbb R}^m,\mathbb{R}_{m})$
	by
	\begin{equation}\label{Definition_of_gc}
	\mathcal{G}_{c}f(x)=\chi_{{B(1)}}(x)\int\limits_{{B(1)}}e^{2\pi ic\langle x,y\rangle}f(y)\, dy,
	\end{equation}
	where $ \chi_{{B(1)}} $ is the characteristic function of $B(1)$.
%	where 
%	$B(1)$
%	is the closed unit ball in
%	$\mathbb{R}^{m}$
%	and 
%	$\partial_{x}$
%	is the Dirac Operator. 
The adjoint ${\mathcal G}_c^*$ of $\mathcal{G}_{c}$ is given by
	$$\mathcal{G}_{c}^{\ast}f(x)=\chi_{{B}(1)}(x)\int\limits_{{\mathbb R}^m}e^{-2\pi ic \langle x,y\rangle}f(y)\, dy.$$
\end{Definition}
\begin{Definition}
	The ``space-limiting'' operator
	$Q:L^{2}(\mathbb{R}^{m},\mathbb{R}_{m})\to L^{2}({B(1)},\mathbb{R}_{m})$
	is given by
	$$Qf(x)=\chi_{B(1)}(x)f(x),$$
	and the ``bandlimiting'' operator 
	$P_{c}: L^{2}(\mathbb{R}^{m},\mathbb{R}_{m})\to PW_{c}$
	is given by
	$$P_{c}f(x)=\int\limits_{{B(1)}}\mathcal{F}f(\xi )e^{2\pi i\langle \xi, x\rangle}d\xi .$$
	Here, $PW_{c}$ is the Paley-Wiener space of functions with bandlimit $c$, i.e.,
	$$PW_c=\{f\in L^{2}(\mathbb{R}^{m},\mathbb{R}_{m}):\, \hat f(\xi )=0\text{ if }|\xi |>c\}.$$
\end{Definition}
\begin{Th}
	The operators of 
	$L_{c},$
	and
	$\mathcal{G}_{c}$ 
	commute, i.e., 
	$L_{c}\mathcal{G}_{c}=\mathcal{G}_{c}L_{c}.$
\end{Th}
\begin{proof} Since $L_c$ is self-adjoint,
\begin{align}{\mathcal G}_cL_cf(x)&=\int\limits_{B(1)}L_cf(y)e^{2\pi ic\langle cx,y\rangle}\, dy\notag\\
&=\int\limits_{B(1)}\overline{e^{-2\pi ic\langle x,y\rangle}}L_cf(y)\, dy=\int\limits_{B(1)}\overline{L_c(e^{-2\pi ic\langle x,\cdot})(y)}f(y)\, dy.\label{comm1}
\end{align}
On the other hand,
\begin{equation}L_c{\mathcal G}_cf(x)=L_c\int\limits_{B(1)}e^{2\pi ic\langle x,y\rangle}f(y)\, dy=\int\limits_{B(1)}L_c(e^{2\pi ic\langle\cdot ,y\rangle})(x)f(y)\,dy.\label{comm2}
\end{equation}
By comparing (\ref{comm1}) and (\ref{comm2}) we see that it is sufficient to show that 
$$L_c(e^{2\pi ic\langle\cdot ,y\rangle})(x)=\overline{L_c(e^{-2\pi ic\langle x,\cdot\rangle})(y)}.$$
By direct calculation, we find that 
\begin{align*}
L_c(e^{2\pi ic\langle\cdot ,y\rangle})(x)&=[-4\pi icxy+4\pi^2c^2(|x|^2+|y|^2-|x|^2|y|^2)]e^{2\pi ic\langle x,y\rangle}\\
&=\overline{L_c(e^{-2\pi ic\langle x,\cdot\rangle})(y)},
\end{align*}
and the proof is complete.
\end{proof}
\begin{Corollary}
	$L_{c},\; \mathcal{G}_{c}$
	and
	$\mathcal{G}_{c}^{\ast}$
	satisfy  the following commutation relations:
	\begin{itemize}
		\item[(a)]
		$\mathcal{G}_{c}^{\ast}L_{c}=L_{c} \mathcal{G}_{c}^{\ast}$
		\item[(b)]
		$L_{c}(\mathcal{G}_{c}^{\ast}\mathcal{G}_{c})=(\mathcal{G}_{c}^{\ast}\mathcal{G}_{c})L_{c}$
%		\item[(c)]
%		$\mathcal{G}_{c}^{\ast}\mathcal{G}_{c},$
		is self-adjoint. 
	\end{itemize}
\end{Corollary}

The operators 
$Q$
and 
$P_{c}$
are self-adjoint on
$L^{2}(\mathbb{R}^{m},\mathbb{C}_{m})$
An equivalent representation of  
$P_{c}$
may be achieved as follows:
\begin{align*}
P_{c}f(x)&=\int\limits_{B(c)}{\mathcal F}f(\xi )e^{2\pi i\langle \xi ,x \rangle}\, d\xi\\
&=\int\limits_{B(c)}\bigg(\int\limits_{\mathbb{R}^{m}}f(s)e^{-2\pi i\langle s,\xi\rangle}\, ds\bigg)e^{2\pi i \langle \xi ,x\rangle}\, d\xi\\
&=c^{m}\int\limits_{\mathbb{R}^{m}}f(s)\bigg(\int\limits_{{B(1)}}e^{2\pi ic\langle \xi ,x-s\rangle}\, d\xi\bigg)\, ds=c^{m}\int\limits_{\mathbb{R}^{m}}f(s)K_{c}(x-s)ds,
\end{align*}
where
\begin{equation}\label{definition_Besinc}
K_{c}(x)=\int\limits_{{B(1)}}e^{2\pi ic\langle \xi ,x\rangle}\, d\xi=\left(\frac{2\pi}{c\vert x\vert}\right)^{\frac{m}{2}}J_{\frac{m}{2}}(c\vert x\vert), 
\end{equation}
where $J_{m}(x)$ is the Bessel function. Also, we see that
$$QP_{c}f(x)=c^{m}\chi_{{B(1)}}(x)\int\limits_{\mathbb{R}^{m}}f(t)K_{c}(x-t)dt.$$
$K_{c}$
is real-valued and even. Now we show that the operator 
$QP_{c}:L^{2}({B}(1),\mathbb{C}_{m})\to L^{2}({B}(1),\mathbb{C}_{m}),$
is self-adjoint. In fact, if $f,g\in L^2(B(1),{\mathbb C}_m)$ then 
$$(f,(QP_{c})^{\ast}g)=(QP_cf,g)=(P_{c}f,Qg)=(P_{c}f,g)=(f,P_{c}g)=(Qf,P_{c}g)=(f,QP_{c}g).$$
Furthermore,
\begin{align}
\mathcal{G}_{c}^{\ast}\mathcal{G}_{c}f(x)&=\chi_{B(1)}(x)\int\limits_{{B(1)}}e^{-2\pi ic\langle y,x\rangle}\mathcal{G}_{c}f(y)\, dy\nonumber\\
&=\chi_{{B(1)}}(x)\int\limits_{{B(1)}}e^{-2\pi ic\langle y,x\rangle}\int\limits_{{B(1)}}e^{2\pi ic\langle y,t\rangle}f(t)\, dt\, dy\nonumber\\
&=\chi_{{B(1)}}(x)\int\limits_{{B(1)}}f(t)\int\limits_{{B(1)}}e^{2\pi ic\langle y,t-x\rangle}\, dy\, dt\nonumber\\
&=\chi_{{B(1)}}(x)\int\limits_{{B(1)}}f(t)K_{c}(x-t)\, dt=\frac{1}{c^{m}}QP_{c}f(x).
\end{align}
The operator
$P_{c}Q$ has the representation
$$
P_{c}Qf(x)=c^{m}\int\limits_{\mathbb{R}^{m}}(Qf)(s)K_{c}(x-s)\, ds=c^{m}\int\limits_{{B(1)}}f(s)K_{c}(x-s)\, ds.$$
We have also that
$$QP_{c}f(x)=\chi_{{B(1)}}(x)\int\limits_{{B(1)}}\chi_{{B(1)}}^{\vee}(x-y)f(y)\, dy=\int\limits_{\mathbb{R}^{m}}M_{c}(x,y)f(y)dy,$$
where
$M_{c}(x,y)=c^{m}\chi_{B(1)}(x)\chi_{{B(1)}}(y){\mathcal F}^{-1}\chi_{{B(1)}}(x-y).$
We therefore have
\begin{align*}
\int\limits_{\mathbb{R}^{m}}\int\limits_{\mathbb{R}^{m}}\vert L_{c}(x,y)\vert^{2}\, dx\, dy&\leq c^{2m}\int\limits_{{B(1)}}\int\limits_{\mathbb{R}^{m}}\vert {\mathcal F}^{-1}\chi_{{B}(1)}(x-y)\vert^{2}\, dy\, dx\\
&=c^{2m}\int\limits_{{B(1)}}\int\limits_{\mathbb{R}^{m}}\vert {\mathcal F}^{-1}\chi_{{B(1)}}(x)\vert^{2}\, dx\, dy\\
&=c^{2m}\vert \chi_{{B(1)}}\vert \int\limits_{\mathbb{R}^{m}}\vert \chi_{{B(1)}}(x)\vert^{2}\, dx=c^{2m}\vert {B(1)}\vert^{2}<\infty.
\end{align*}
Therefore, 
$QP_{c}$
is a Hilbert-Schmidt operator, and therefore compact.

%\begin{Corollary}\label{share_eigenbasis_QPc}
%	From Lemma \ref{share_eigenbasis_lemma}, it's possible to conclude that $QP_{c}$ and $L_{c}$ share a common eigenbasis in $\mathbb{R}_{2}$.
%\end{Corollary}

\begin{Proposition}\label{radial_times_monogenic_for_even}
Let $n=2N$ be an even integer, and $\psi_{n,m}^{k,c,i}(x)$ be a CPSWF. Then, 
\begin{itemize}
\item[{(i)}]$\psi_{2N,m}^{k,c,i}(x)=P_{N,m}^{k,c}(\vert x\vert^{2})Y_{k}^i(x),$ where 
$P_{N,m}^{k,c}(\vert x\vert^{2})$
is a radial function.
\item[{(ii)}] $\mathcal{G}_{c}\psi_{2N,m}^{k,c,i}(x)=R_{2N,m}^{k,c}(\vert x\vert)Y_{k}^i(x),$ where 
$R_{2N,m}^{k,c}(\vert x\vert)$
is a radial function.
\item[{(iii)}]
$QP_{c}\psi_{2N,m}^{k,c,i}(x)=q_{2N,m}^{k,c}(\vert x\vert)Y_{k}^i(x),$ where
$q_{2N,m}^{k,c}(\vert x\vert)$
is a radial function.
\end{itemize}
\end{Proposition}

\begin{proof}
	\item[{(i)}]
	We've seen that
	$\psi_{2N,m}^{k,c,i}(x)=\sum\limits_{i=0}^{\infty}b_{i}^{k}C_{2i,m}^{0}(Y_{k}^i)(x),$
	but we know that
	$C_{2N,m}^{0}Y_{k}(x)=P_{N,m}^{k}(\vert x\vert^{2})Y_{k}(x)$ with $P_{N,m}^k$ a polynomial. Consequently, 
	$\psi_{2N,m}^{k,c,i}(x)=P_{N,m}^{k,c}(\vert x\vert^{2})Y_{k}^i(x).$
	\item[{(ii)}]
Since
	$\psi_{2N,m}^{k,c,i}(x)=P_{N,m}^{k,c}(\vert x\vert^{2})Y_{k}^i(x),$
	 we have 
	\begin{align*}
	\mathcal{G}_{c}\psi_{n,m}^{k,c,i}(x)%&%=\mathcal{G}_{c}\big(P_{N,m}^{k,c}(\vert x\vert^{2})Y_{k}(x)\big)\\
	&=\chi_{{B(1)}}(x)\int\limits_{{B(1)}}e^{2\pi ic\langle x,y\rangle}P_{N,m}^{k,c}(\vert y\vert^{2})Y_{k}(y)\, dy\\
	&=\chi_{B(1)}(x)\int\limits_{0}^{1}r^{m-1}\int\limits_{S^{m-1}}e^{2\pi ic\langle  r'\hat{x},r\theta\rangle}P_{N,m}^{k,c}(r^{2})Y_{k}(r\theta)\, d\theta \, dr\\
	&=\chi_{B(1)}(x)\int\limits_{0}^{1}r^{k+m-1}\int\limits_{S^{m-1}}e^{2\pi ic\langle  r'\hat{x},r\theta\rangle}P_{N,m}^{k,c}(r^2)Y_{k}(\theta)\, d\theta\,  dr\\
	&=\chi_{B(1)}(x)\int\limits_{0}^{1}r^{k+m-1}P_{N,m}^{k,c}(r^{2})\dfrac{(2\pi)^{\frac{m}{2}}(-i)^{k}}{(2\pi crr')^{\frac{m-2}{2}}}J_{k+\frac{m}{2}-1}(2\pi crr')Y_{k}(-\hat{x})\, dr\\
	&=\chi_{B(1)}(x)\dfrac{(2\pi)^{\frac{m}{2}}(-i)^{k}Y_{k}(-\hat{x})}{(2\pi cr')^{\frac{m-2}{2}}}\int\limits_{0}^{1}r^{k+\frac{m}{2}}P_{N,m}^{k,c}(r^{2})J_{k+\frac{m}{2}-1}(2\pi crr')\, dr\\
	&=R_{2N,m}^{k,c}(\vert x\vert)Y_{k}(x),
	\end{align*}
	where we have used Lemma 9.10.2 in \cite{andrews1999special}.
	\item[{(iii)}]
	Since ${\mathcal G}_c^*f(x)={\mathcal G}_cf(-x)$,
%	\begin{align*}
%	\mathcal{G}_{c}(f)(x)&=\chi_{{B(1)}}(x)\int\limits_{{B(1)}}e^{2\pi ic\langle x,y\rangle}f(y)\, dy\\
%	&=\chi_{{B(1)}}(x)\int\limits_{{B(1)}}e^{-2\pi ic\langle -x,y\rangle}f(y)\, dy=\mathcal{G}_{c}^{\ast}f(-x),
%	\end{align*}
	we have
	\begin{align*}
	\mathcal{G}_{c}^{\ast}\mathcal{G}_{c}(\psi_{2N,m}^{k,c,i})(x)&=\mathcal{G}_{c}^{\ast}\mathcal{G}_{c}\big(P_{N,m}^{k,c}(\vert x\vert^{2})Y_{k}^i(x)\big)\\
	&=\mathcal{G}_{c}^{\ast}(R_{2N,m}^{k,c}(\vert x\vert)Y_{k}^i(-x))\\
	&=\mathcal{G}_{c}^{\ast}(R_{2N,m}^{k,c}(\vert x\vert)(-1)^{k}Y_{k}^i(x))=q_{2N,m}^{k,c}(\vert x\vert)Y_{k}^i(x).
	\end{align*}
\end{proof}
In the similar way, we may conclude analogous results for the odd CPSWFs.

\begin{Proposition}\label{radial_times_monogenic_for_odd}
	Let
	$n=2N+1$
	be odd. Then	
	\begin{itemize}
		\item[\textbf{(i)}]
		$\psi_{2N+1,m}^{k,c,i}(x)=Q_{N,m}^{k,c}(\vert x\vert^{2})xY_{k}^i(x),$
		where
		$Q_{N,m}^{k,c}(\vert x\vert^{2})$
		is a radial function.
		\item[\textbf{(ii)}]  $\mathcal{G}_{c}\psi_{2N+1,m}^{k,c,i}(x)=S_{2N,m}^{k,c}(\vert x\vert)xY_{k}^i(x),$
		where
		$S_{N,m}^{k,c}(\vert x\vert)$
		is a radial function.
		\item[\textbf{(iii)}]
		$QP_{c}\psi_{2N+1,m}^{k,c,i}(x)=T_{2N,m}^{k,c}(\vert x\vert)xY_{k}^i(x),$ 
		where
		$T_{2N,m}^{k,c}(\vert x\vert)$
		is a radial function.
	\end{itemize}
\end{Proposition}
\begin{proof}
	The proof of part (i) and (iii) are similar to the proof of Proposition \ref{radial_times_monogenic_for_even} so we prove only part (ii). From part (i), we know that 
	$\psi_{2N+1,m}^{k,c,i}(x)=Q_{N,m}^{k,c}(\vert x\vert^{2})xY_{k}^{i}(x).$ Therefore, 
	\begin{align*}
	\mathcal{G}_{c}\psi_{2N+1,m}^{k,c,i}(x)%&=\int\limits_{{B(1)}}\psi_{2N+1,m}^{k,c,i}(y)e^{2\pi i c \langle x,y\rangle}dy\\
	&=\int\limits_{B(1)}Q_{N,m}^{k,c}(y)(\vert y\vert^{2})yY_{k}(y)e^{2\pi i c \langle x,y\rangle}\, dy\\
	&=\int\limits_{0}^{1}\int\limits_{S^{m-1}}Q_{N,m}^{k,c}(r^{2})r\omega Y_{k}(r\omega)e^{2\pi i c rs \langle \omega,\xi\rangle}r^{m-1} \, d\omega\, dr.
	\end{align*}
	Now we want to calculate 
	$\int\limits_{S^{m-1}}\omega Y_{k}(\omega)e^{2\pi i c rs \langle \omega,\xi\rangle} \, d\omega$. For $x\in{\mathbb R}^m$, we put
	$$F(x):=\int\limits_{S^{m-1}} Y_{k}(\omega)e^{2\pi i c r \langle \omega,x \rangle }d\omega.$$
	Applying the Dirac operator from the left on both sides of this equation yields
$$\partial_{x}F(x):=2\pi i c r\int\limits_{S^{m-1}} \omega Y_{k}(\omega)e^{2\pi i c r \langle \omega,x \rangle }d\omega.$$
	From Lemma 9.10.2 in \cite{andrews1999special} we have
\begin{equation}
\int\limits_{S^{m-1}}Y_{k}(\omega)e^{2\pi i c r\langle \omega,x \rangle}d\omega=\frac{(2\pi)(-i)^{k}}{r^{\frac{m}{2}-1}}\vert x\vert^{-(k+\frac{m}{2}+1)}J_{k+\frac{m}{2}-1}(2\pi c r\vert x\vert)Y_{k}(x),\label{HF}
\end{equation}
	and therefore,
	\begin{align*}
\int\limits_{S^{m-1}} \omega Y_{k}(\omega)e^{2\pi i c r \langle \omega,x \rangle }d\omega&=\frac{(-1)^k(-i)^{k}}{cr^{\frac{m}{2}}}\partial_{x}[\vert x\vert^{-(k+\frac{m}{2}+1)}J_{k+\frac{m}{2}-1}(2\pi c r\vert x\vert)Y_{k}(x)]\hspace*{4cm}\\
%	&=\frac{(2\pi)(-i)^{k}}{r^{\frac{m}{2}-1}}\sum_{j}e_{j}\frac{\partial}{\partial x_{j}}[J_{k+\frac{m}{2}-1}(2\pi c r\vert x\vert)\vert x\vert^{-(k+\frac{m}{2}+1)}Y_{k}(x)]\\
%	&=\frac{(2\pi)(-i)^{k}}{r^{\frac{m}{2}-1}}\sum_{j}e_{j}\frac{\partial}{\partial x_{j}}[J_{k+\frac{m}{2}-1}(2\pi c r\vert x\vert)\vert x\vert^{-(k+\frac{m}{2}+1)}]Y_{k}(x)\\
%	&=\frac{(2\pi)(-i)^{k}}{r^{\frac{m}{2}-1}}(-2\pi c r)x [J_{k+\frac{m}{2}}(2\pi c r\vert x\vert)\vert x\vert^{-(k+\frac{m}{2}+2)}]Y_{k}(x),
&=\frac{(2\pi)(-1)^{k}i^{k-1}}{r^{\frac{m}{2}-1}}xJ_{k+\frac{m}{2}}(2\pi c r\vert x\vert)\vert x\vert^{-(k+\frac{m}{2}+2)}Y_{k}(x)
	\end{align*}
%	where we have applied (\cite{gradshteyn2007ryzhik}, 8.472). Therefore,
%	$$\int\limits_{S^{m-1}} \omega Y_{k}(\omega)e^{2\pi i c r \langle \omega,x \rangle }d\omega=\frac{(2\pi)(-1)^{k}i^{k-1}}{r^{\frac{m}{2}-1}}xJ_{k+\frac{m}{2}}(2\pi c r\vert x\vert)\vert x\vert^{-(k+\frac{m}{2}+2)}Y_{k}(x).$$
from which we conclude
	$$\mathcal{G}_{c}\psi_{2N+1,m}^{k,c,i}(x)=2\pi(-1)^ki^{k-1}\left(\int\limits_{0}^{1}Q_{N,m}^{k,c}(r^{2})r^{\frac{m}{2}+1}J_{k+\frac{m}{2}}(2\pi c r\vert x\vert)\, dr\right) \frac{x}{|x|^{k+\frac{m}{2}+2}}Y_{k}^i(x).$$
\end{proof}

\begin{Th}\label{CPSWF eig of QPc}
	The even CPSWFs
	$\psi_{2N,2}^{k,c,i}(x)$
	are eigenfunctions of
	$QP_{c}.$ The eigenvalues are positive.
\end{Th}

\begin{proof} Note that $\psi_{2N,m}^{k,c.i}(x)=P_{N,m}^k(|x|^2)Y_k^i(x)$ with $P_{N,m}^k$ real-valued. Also,
\begin{equation}
(\psi_{2N,m}^{k,c,i},\psi_{N,m}^{k,c,i})=\int\limits_{B(1)}\overline{\psi_{N,m}^{k,c,i}(x)}\psi_{N,m}^{k,c,i}(x)\, dx=\int\limits_0^1r^{m+2k-1}P_{N,m}^k(r^2)\, dr >0.\label{pos1}
\end{equation}
Further,
\begin{align*}
{\mathcal F}\psi_{2N,m}^{k,c,i}(y)&=\int\limits_{B(1)}e^{-2\pi i\langle x,y\rangle}\psi_{2N,m}^{k,c,i}(x)\, dx\\
&=\int\limits_0^1P_{N,m}^k(r^2)r^{m+k-1}\int\limits_{S^{m-1}}e^{-2\pi ir|y|\langle\omega ,y/|y|\rangle}Y_k^i(\omega )\, d\omega\\
&=\frac{2\pi (-i)^k}{|y|^{\frac{m}{2}+k-1}}\left(\int\limits_0^1P_{N,m}^k(r^2)r^{\frac{m}{2}+k}J_{k+\frac{m}{2}-1}(2\pi r|y|)\, dy\right)Y_k^i(y)\\
&=(-i)^kS_{N,m}^k(|y|)Y_k^i(y),
\end{align*}
where $S_{N,m}^k$ is real-valued. Therefore,
\begin{align*}
P_c\psi_{2N,m}^{k,c,i}(x)&=\int\limits_{B(c)}{\mathcal F}\psi_{2N,m}^{k,c,i}(y)e^{2\pi i\langle x,y\rangle}\, dy\\
&=(-i)^k\int\limits_{B(c)}S_{N,m}^k(|y|)Y_k^i(y)e^{2\pi i\langle x,y\rangle}\, dy\\
&=(-i)^k\int\limits_0^cr^{m+k-1}S_{N,m}^k(r)\int\limits_{S^{m-1}}Y_k^i(\omega )e^{2\pi ir|x|\langle\frac{x}{|x|},\omega\rangle}\, d\omega\, dr\\
&=\frac{2\pi}{|x|^{\frac{m}{2}+k-1}}\left(\int\limits_0^cr^{\frac{m}{2}+k}S_{N,m}^k(r)J_{k+\frac{m}{2}-1}(2\pi r|x|)\, dr\right) Y_k^i(x)\\
&=T_{N,m}^k(|x|)Y_k^i(x),
\end{align*}
with $T_{N,m}^k$ real-valued. Hence,
\begin{align*}
(P_c\psi_{2N,m}^{k,c,i},P_c\psi_{2N,m}^{k,c,i})&=\int\limits_{{\mathbb R}^m}T_{N,m}^k(|x|)^2\overline{Y_k^i(x)}Y_k^i(x)\, dx\\
&=\int\limits_0^1r^{2k+m-1}T_{N,m}^k(r)^2\, dr\geq 0 .
\end{align*}
On the other hand,
\begin{align*}
(P_c\psi_{2N,m}^{k,c,i},P_c\psi_{2N,m}^{k,c,i})=0&\iff T_{N,m}^k\equiv 0\\
&\iff P_c\psi_{2N,m}^{k,c,i}\equiv 0\\
&\iff{\mathcal F}\psi_{2N,m}^{k,c,i}(y)=0\text{ for }|y|\leq c .
\end{align*}
But $\psi_{2N,m}^{k,c,i}$ is non-trivial and compactly supported, so its Fourier transform cannot vanish on $B(c)$. We conclude that 
\begin{equation}
(P_c\psi_{2N,m}^{k,c,i},P_c\psi_{2N,m}^{k,c,i})>0.\label{pos2}
\end{equation}
Note that 
$$L_c(QP_c)\psi_{2N,m}^{k,c,i}=(QP_c)L_c\psi_{2N,m}^{k,c,i}=\chi_{2N}^{k,c}(QP_c)\psi_{2N,m}^{k,c,i},$$
and therefore $QP_c\psi_{2N,m}^{k,c,i}\in E_{\chi_{2N}^{k,c}}(L_c)$ (the $\chi_{2N}^{k,c}$ eigenspace of $L_c$). Note also that $\psi_{2N,m}^{k,c,i}\in{\mathcal H}_k^i$ and
$$QP_c\psi_{2N,m}^{k,c,i}(x)=\chi_{B(1)}T_{N,m}^k(|x|)Y_k^i(x)\in{\mathcal H}_k^i,$$
so that $QP_c\psi_{2N,m}^{k,c,i}\in E_{\chi_{2N}^{k,c}}(L_c\big|_{{\mathcal H}_k^i})$ which is one-dimensional and contains $\psi_{2N,m}^{k,c,i}$. We conclude that there exists $\lambda\in{\mathbb C}_m$ for which
$$QP_c\psi_{2N,m}^{k,c,i}=\psi_{2N,m}^{k,c,i}\lambda,$$
i.e., each CPSWF $\psi_{2N,m}^{k,c,i}$ is an eigenfunction of $QP_c$. Also,
\begin{equation}(P_c\psi_{2N,m}^{k,c,i},P_c\psi_{2N,m}^{k,c,i})=(\psi_{2N,m}^{k,c,i},QP_c\psi_{2N,m}^{k,c,i})=(\psi_{2N,m}^{k,c,i},\psi_{2N,m}^{k,c,i})\lambda\label{pos3}
\end{equation}
Applying (\ref{pos1}) and (\ref{pos2}) to (\ref{pos3}) then gives $\lambda >0$

\end{proof}

\begin{Th}
	The even CPSWFs
	$\psi_{2N,m}^{k,c,i}$
	are eigenfunctions of
	$\mathcal{G}_{c}.$
\end{Th}
\begin{proof}
	The proof is similar to that of Theorem \ref{CPSWF eig of QPc}, and uses the commutation relation
	$\mathcal{G}_{c}L_{c}=L_{c}\mathcal{G}_{c}.$
%	Suppose that
%	$\psi_{2N,m}^{k,c,i}(x)$ is the eigenfunction of the
%	$L_{c},$
%	i.e.,
%	$L_{c}\psi_{2N,m}^{k,c,i}(x)=\chi_{2N,m}^{k,c} \psi_{2N,m}^{k,c,i}(x).$
%	So we can write
%	$$L_{c}\mathcal{G}_{c}\psi_{2N,m}^{k,c,i}(x)=\mathcal{G}_{c}L_{c}\psi_{2N,m}^{k,c,i}(x)=\mathcal{G}_{c}\big(\chi_{2N,m}^{k,c} \psi_{2N,m}^{k,c,i}(x)\big)=\chi_{2N,m}^{k,c}\big(\mathcal{G}_{c}\psi_{2N,m}^{k,c,i}(x)\big).$$
%	So $\mathcal{G}_{c}(x)\in E_{\chi_{2N,m}^{k,c}}(L_c).$ The rest of the proof is similar to the proof of \ref{CPSWF eig of QPc}.
\end{proof}
Similarly, we have the following result.
\begin{Th}\label{odd CPSWFs eigenfunctions of QP_c}
	The odd CPSWFs
	$\psi_{2N+1,m}^{k,c,i}(x)$
	are the eigenfunction of both	$QP_{c}$ and ${\mathcal G}_c$.
\end{Th}
%\begin{Th}
%	The odd CPSWFs
%	$\psi_{2N+1,m}^{k,c,i}(x)$
%	are also eigenfunctions of the
%	$\mathcal{G}_{c}.$
%\end{Th}
%According to the Corollary \ref{share_eigenbasis_QPc} we can conclude the following statement for CPSWFs which have been defined on two dimension disc.

Combining Theorems \ref{CPSWFs_are_basis_using_SL}, \ref{CPSWF eig of QPc}and \ref{odd CPSWFs eigenfunctions of QP_c} gives the following result.
\begin{Corollary}
	The CPSWFs $\{\psi_{n,m}^{k,c,i}:\, n\geq 0,\ k\geq 0,\ 1\leq i\leq d_k\}$ is an orthonormal basis for $L^2(B(1),{\mathbb R}_m)$ consisting of eigenfunctions of $QP_{c}$.
\end{Corollary}
%Now using Theorem \ref{CPSWFs_are_basis_using_SL} will give the following fundamental result.
%\begin{Corollary}
%The CPSWFs $\psi_{n,2,m}^{k,c,i}$ are eigenbasis for the operator $QP_{c}.$
%\end{Corollary}
%\begin{Remark}
%Right now, we are not aware of the extended version of Theorem \ref{spectral_theorem} so we can't obtain Lemma \ref{share_eigenbasis_lemma} for extended version. However, we proved independently from spectral theorem \ref{spectral_theorem}, that the eigenfunctions of the \ref{Definition_of_Lc_operator} are also eigenfunctions of \ref{Definition_of_gc} which are basis. Nevertheless, we can prove that CPSWFs, in $2-$dimension, are eigenbasis of the operators $L_{c}$ and $\mathcal{G}_{c}$ using spectral theorem.
%\end{Remark}
\section{Some Features of the CPSWFs}
Let's assume that 
$\mathcal{G}_{c}\psi_{j}(x)=\mu_{j}\psi_{j}(x)$
where 
$\psi_{j}(x)$
is a time-limited Clifford Prolate Spheroidal Wave Functions. We also know that 
$QP_{c}\psi_{j}(x)=\lambda_{j}\psi_{j}(x).$
We also see that
$$P_{c}Q(P_{c}\psi_{j})(x)=\lambda_{j}(P_{c}\psi_{j})(x),$$
showing that the 
$QP_{c},$
and
$P_{c}Q$
have the same eigenvalue 
$\lambda_{j}$
with eigenfunctions
$\psi_{j}$
and
$P_{c}\psi_{j}$
respectively. Now, we are going to show a relationship between the eigenvalues of 
$\mathcal{G}_{c},$
and
$QP_{c}.$
Now we have that 
\begin{align*}
\vert\mu_{j}\vert^{2}=\vert\mu\vert^{2}\langle \psi_{j},\psi_{j}\rangle&=\langle \mu_{j}\psi_{j}, \mu_{j}\psi_{j} \rangle\\
&=\langle \mathcal{G}_{c}\psi_{j}, \mathcal{G}_{c}\psi_{j} \rangle\\
&=\langle \mathcal{G}_{c}^{\ast}\mathcal{G}_{c}\psi_{j}, \psi_{j} \rangle\\
&=\langle \frac{1}{c^{m}}QP_{c}\psi_{j}, \psi_{j} \rangle
%&=\langle \frac{1}{c^{m}}QP_{c}\psi_{j}, \psi_{j} \rangle\\
=\frac{1}{c^{m}}\langle \lambda_{j}\psi_{j}, \psi_{j} \rangle=\frac{1}{c^{m}}\lambda_{j}.
\end{align*}
So 
$\vert \mu_{j}\vert^{2}=\frac{1}{c^{m}}\lambda_{j}.$
We know that 
$\mu_{j}$
is a complex-valued number. So
\begin{equation}\label{relation_eigenvalues}
\mu_{j}=\alpha_{j}\frac{\sqrt{\lambda_{j}}}{c^{\frac{m}{2}}},
\end{equation}
where 
$\alpha_{j}$
is a complex-valued and 
$\vert \alpha_{j}\vert=1.$
Now, we have already seen that
\begin{align*}
\mathcal{G}_{c}\psi_{2N,m}^{k,c,j}(x)&=i^{k}G_{2N}(\vert x\vert)Y_{k}^{j}(x)=\mu_{2N,m}^{k,c}\psi_{2N,m}^{k,c,j}(x),\\
\mathcal{G}_{c}\psi_{2N+1,m}^{k,c,j'}(x)&=i^{k+1}F_{2N+1}(\vert x\vert)xY_{k}^{j'}(x)=\mu_{2N+1,m}^{k,c}\psi_{2N+1,m}^{k,c,j'}(x),\\
\end{align*}
as a result, I can write
\begin{equation}
\mu_{j}=\epsilon_{j}i^{k+n}\frac{\sqrt{\lambda_{j}}}{c^{\frac{m}{2}}}.
\end{equation}
Therefore, we have found the relationship between the eigenvalues of the 
$\mathcal{G}_{c},\; QP_{c},$
and,
$P_{c}Q.$
At this part of the section, we are going to investigate some features of the CPSWFs related to spectral concentration problem. For this, we firstly are going to prove the non-degeneracy of the eigenvalues of the operator $P_c Q$ for a fixed $k$. 

\begin{Th}\label{Eigenvalues_Distinct}
For a fixed $k$ the eigenvalues $\{\lambda_{n,m}^{k,c}\}$ are distinct.
\end{Th}

\begin{proof}
First suppose that $\lambda_{2M,m}^{k,c}=\lambda_{2N,m}^{k,c}$ with $M\neq N$. Then
$$\lambda_{2N,m}^{k,c}\psi_{2M,m}^{k,c,j}(x)=\int\limits_{B(1)}K_c(x-y)\psi_{2M,m}^{k,c,j}(y)\, dy,$$
so that
$$\lambda_{2N,m}^{k,c}\overline{\Delta\psi_{2M,m}^{k,c,j}(x)}=\int\limits_{B(1)}\Delta K_c(x-y)\overline{\psi_{2M,m}^{k,c,j}(y)}\, dy,$$
and as a consequence we have
\begin{equation}
\lambda_{2N,m}^{k,c}\int\limits_{B(1)}\overline{\Delta\Psi_{2M,m}^{k,c,j}(x)}\psi_{2N,m}^{k,c,i}(x)\, dx=\int\limits_{B(1)}\int\limits_{B(1)}\Delta K_c(x-y)\overline{\psi_{2M,m}^{k,c,j}(y)}\psi_{2N,m}^{k,c,i}(x)\, dy\, dx.\label{non-deg 1}
\end{equation}
Similarly, we have
\begin{equation}
\lambda_{2N,m}^{k,c}\int\limits_{B(1)}\overline{\psi_{2M,m}^{k,c,j}(x)}\Delta\psi_{2N,m}^{k,c,i}(x)\, dx=\int\limits_{B(1)}\int\limits_{B(1)}\overline{\psi_{2M,m}^{k,c,j}(x)}\Delta K_c(x-y)\psi_{2N,m}^{k,c,i}(y)\, dy\, dx .\label{non-deg 2}
\end{equation}
Subtracting (\ref{non-deg 2}) from (\ref{non-deg 1}) and applying Stokes' theorem gives
\begin{align}
0&=\lambda_{2N,m}^{k,c}\int\limits_{B(1)}\overline{\Delta\psi_{2M,m}^{k,c,j}(x)}\psi_{2N,m}^{k,c,i}(x)-\overline{\psi_{2M,m}^{k,c,j}(x)}\Delta\psi_{2N,m}^{k,c,i}(x)\, dx\notag\\
&=\lambda_{2N,m}^{k,c}\int\limits_{B(1)}\overline{\psi_{2M,m}^{k,c,j}(x)}\partial_x^2\psi_{2N,m}^{k,c,i}(x)-\overline{\partial_x^2(\psi_{2M,m}^{k,c,j}(x))}\psi_{2N,m}^{k,c,i}(x)\, dx\notag\\
&=\int\limits_{S^{m-1}}\overline{\psi_{2M,m}^{k,c,j}(x)}x(\partial_{x}\psi_{2N,m}^{k,c,i}(x))+\partial_{x}\overline{\psi_{2M,m}^{k,c,j}(x)}x\psi_{2N,m}^{k,c,i}(x)\, d\sigma (x).\label{non-deg 3}
\end{align}
However, since $\psi_{2M,m}^{k,c,j}(x)=P_{M,m}^{k,c}(|x|^2)Y_k^j(x)$, we have $\partial_{x}\psi_{2M}^{k,j}(x)=2x(P_{M,m}^{k,c})'(|x|^2)Y_k^j(x)$ and from (\ref{non-deg 3}) and the orthonormality of $\{Y_k^i\}_{i=1}^{d_k}$ in $L^2(S^{m-1},{\mathbb R}^m)$, we have
$$0=(P_{M,m}^{k,c})'(1)P_{N,m}^{k,c}(1)-P_{M,m}^{k,c}(1)(P_{N,m}^{k,c})'(1),$$
or equivalently,
\begin{equation}
\frac{(P_{N,m}^{k,c})'(1)}{P_{N,m}^{k,c}(1)}=\frac{(P_{M,m}^{k,c})'(1)}{P_{M,m}^{k,c}(1)}.\label{non-deg 4}
\end{equation}
However, from (\ref{SL even}), $(P_{N,m}^{k,c})'(1)=\left(\dfrac{\chi_{2N}^k}{4}-\pi^2c^2\right)P_{N,m}^{k,c}(1)$ and we conclude from (\ref{non-deg 4}) that $\chi_{2N,m}^{k,c}=\chi_{2M,m}^{k,c}$, which contradicts the eigenvalues $\{\chi_{2N,m}^{k,c}\}_{N=0}^\infty$ of $L_c$ are distinct. We conclude that the eigenvalues $\{\lambda_{2N,m}^{k,c}\}_{N=0}^\infty$ of $QP_c$  are distinct.
\vskip0.2in
Suppose now that $\lambda_{2N+1,m}^{k,c}=\lambda_{2M+1,m}^{k,c}$. Then we have
\begin{equation}
0=\int\limits_{S^{m-1}}\overline{\psi_{2M+1,m}^{k,c,j}(x)}x(\partial_{x}\psi_{2N+1,m}^{k,c,i}(x)+\overline{\partial_{x}\psi_{2M+1,m}^{k,c,j}(x)}x\psi_{2N+1,m}^{k,c,i}(x))\, d\sigma (x).\label{non-deg 5}
\end{equation}
However, since $\psi_{2N+1,m}^{k,c,i}(x)=xQ_{N,m}^{k,c}(|x|^2)Y_k^i(x)$, by the monogenicity and homogeneity of $Y_k^i$ we have
\begin{equation}
\partial_{x}\psi_{2N+1,m}^{k,c,i}(x)=[(m+2k)Q_{N,m}^{k,c}(|x|^2)-2|x|^2(Q_{N,m}^{k,c})'(|x|^2)]Y_k^i(x).\label{non-deg 6}
\end{equation}
Substituting (\ref{non-deg 6}) into (\ref{non-deg 5}), the orthonormality of $\{Y_k^i\}_{i=1}^{d_k}$ in $L^2(S^{m-1})$ yields
$$0=Q_{M,m}^{k,c}(1)[-(m+2k)Q_{N,m}^{k,c}(1)+2(Q_{N,m}^{k,c})'(1)]-Q_{N,m}^{k,c}(1)[-(m+2k)Q_{M,m}^{k,c}(1)+2(Q_{N,m}^{k,c})'(1)]$$
which we rearrange to find
\begin{equation}
\frac{(m+2k)Q_{M,m}^{k,c}(1)+2(Q_{M,m}^{k,c})'(1)}{Q_{M,m}^{k,c}(1)}=\frac{(m+2k)Q_{N,m}^{k,c}(1)+2(Q_{N,m}^{k,c})'(1)}{Q_{N,m}^{k,c}(1)}.\label{non-deg 7}
\end{equation}
The functions $\{Q_{N,m}^{k,c}\}_{N=0}^\infty$ satisfy (\ref{SL_odd}) and therefore,
\begin{equation}
2(Q_{N,m}^{k,c})'(1)+(m+2k)Q_{N,m}^{k,c}(1)=\left(\frac{m}{2}+k-\frac{\pi^2c^2}{2}-\frac{\chi_{2N+1}^k}{4}\right),\label{non-deg 8}
\end{equation}
and applying (\ref{non-deg 8}) to (\ref{non-deg 7}) yields $\chi_{2N+1,m}^{k,c}=\chi_{2M+1,m}^{k,c}$, which contradicts the fact that the eigenvalues $\{\chi_{2N+1,m}^{k,c}\}_{N=0}^\infty$ are distinct. We conclude that $\lambda_{2N+1}^k\neq\lambda_{2M+1,m}^{k,c}$ unless $N=M$.
\vskip0.2in

Finally, we suppose $\lambda_{2N+1}^k=\lambda_{2M}^k$. Then with an application of Stokes' theorem,  we have
\begin{align}
\lambda_{2N,m}^{k,c}\partial_{x}\psi_{2N,m}^{k,c,i}(x)&=\int\limits_{B(1)}\partial_{x}(K_c(x-y))\psi_{2N,m}^{k,c,i}(y)\, dy\notag\\
&=-\int\limits_{B(1)}\partial_{y}(K_c(x-y))\psi_{2N,m}^{k,c,i}(y)\, dy\notag\\
&=-\left\{\int\limits_{S^{m-1}}K_c(x-y)y\psi_{2N,m}^{k,c,i}(y)\, d\sigma (y)-\int\limits_{B(1)}K_c(x-y)\partial_{y}\psi_{2N,m}^{k,c,i}(y)\, dy\right\}.\label{non-deg 9}
\end{align}
Multiplying both sides of (\ref{non-deg 9}) on the left by $\overline{\psi_{2M+1,m}^{k,c,j}(x)}$ and integrating over $B(1)$ gives
\begin{align}
\lambda_{2N,m}^{k,c}\int\limits_{B(1)}\overline{\psi_{2M+1,m}^{k,c,j}(x)}(\partial_{x}\psi_{2N,m}^{k,c,i})(x)\, dx&=\int\limits_{B(1)}\int\limits_{B(1)}\overline{\psi_{2M+1,m}^{k,c,j}(x)}K_c(x-y)(\partial_{x}\psi_{2N,m}^{k,c,i})(y)\, dy\, dx\notag\\
&-\lambda_{2M+1,m}^{k,c}\int\limits_{S^{m-1}}\overline{\psi_{2M+1,m}^{k,c,j}(y)}y\psi_{2N,m}^{k,c,i}(y)\,d\sigma (y),\label{non-deg 10}
\end{align}
where we have used the fact that 
\begin{equation}
\int\limits_{B(1)}K_c(x-y)\psi_{2M+1,m}^{k,c,j}(x)\, dx=\lambda_{2M+1,m}^{k,c}\psi_{2M+1,m}^{k,c,j}(y).\label{non-deg 11}
\end{equation}
On the other hand, conjugating both sides of (\ref{non-deg 11}), multiplying on the right by $\partial_{x}\psi_{2N,m}^{k,c,i}(x)$ and integrating over $B(1)$ gives
\begin{equation}
\lambda_{2M+1}^k\int\limits_{B(1)}\overline{\psi_{2M+1,m}^{k,c,j}(x)}(\partial_{x}\psi_{2N,m}^{k,c,i})(x)
=\int\limits_{B(1)}\int\limits_{B(1)}\overline{\psi_{2M+1,m}^{k,c,j}(y)}K_c(x-y)(\partial_{x}\psi_{2N,m}^{k,c,i})(x)\, dy\, dx .\label{non-deg 12}
\end{equation}
If $\lambda_{2n,m}^k=\lambda_{2M+1}^k$, then subtracting (\ref{non-deg 12}) from (\ref{non-deg 10}) and using the orthonormality of $\{Y_k^i\}_{i=1}^{d_k}$ gives
$$0=\int\limits_{S^{m-1}}\overline{\psi_{2M+1,m}^{k,c,j}(y)}y\psi_{2N,m}^{k,c,i}(y)\, d\sigma (y)=Q_M^{k}(1)P_N^k(1)\delta_{ij},$$
which contradicts the fact that $P_{N,m}^{k,c}(1)Q_{N,m}^{k,c}(1)\neq 0$. We conclude that $\lambda_{2N,m}^{k,c}\neq\lambda_{2M+1,m}^{k,c}$ for all $N,M\geq 0$.
\end{proof}

\begin{Th}\label{Continuity_Eigenvalues_c}
	The eigenvalues  $\lambda_{n,m}^{k,c}$ of $QP_c$ are continuous functions of $c$.
\end{Th}
\begin{proof}
	Since $\psi_{2N,m}^{k,c,i}(x)=P_{N,m}^{k,c}(|x|^2)Y_k^i(x)$, we have
	\begin{align*}
		{\mathcal G}_c\psi_{2N,m}^{k,c,i}(x)&=\int\limits_{B(1)}e^{2\pi ic\langle x,y\rangle}P_{N,m}^{k,c}(|y|^2)Y_k^i(y)\, dy\\
		&=\int\limits_0^1r^{m+k-1}P_{N,m}^{k,c}(r^2)\int\limits_{S^{m-1}}e^{2\pi icr\langle x,\omega\rangle}Y_k^i(\omega )\, d\omega\, dr\\
		&=\frac{2\pi i^k}{(c|x|)^{\frac{m}{2}-1}}\int\limits_0^1r^{\frac{m}{2}+k}P_{N,m}^{k,c}(r^2)J_{k+\frac{m}{2}-1}(2\pi cr|x|)\, dr\, Y_k^i\left(\frac{x}{|x|}\right),
	\end{align*}
	where we have used (\ref{HF}). Similarly,
	\begin{align}
		{\mathcal G}_c^*{\mathcal G}_c\psi_{2N,m}^{k,c,i}(y)&=\int\limits_{B(1)}e^{-2\pi ic\langle x,y}{\mathcal G}_c\psi_{2N,m}^{k,c,i}(x)\, dx\notag\\
		&=\frac{2\pi i^k}{c^{\frac{m}{2}-1}}\int\limits_0^1s^{\frac{m}{2}}\int\limits_0^1r^{\frac{m}{2}+k}J_{k+\frac{m}{2}-1}(2\pi crs)\int\limits_{S^{m-1}}e^{-2\pi ics\langle y,\omega\rangle}Y_k^i(\omega )\, d\omega\, dr\, ds\notag\\
		&=\frac{4\pi^2}{c^{m-2}|y|^{\frac{m}{2}-1}}\int\limits_0^1\int\limits_0^1sr^{\frac{m}{2}+k}J_{k+\frac{m}{2}-1}(2\pi crs)J_{k+\frac{m}{2}-1}(2\pi cs|y|)\, dr\, ds\, Y_k^i\left(\frac{y}{|y|}\right)\notag\\
		&=\frac{2\pi}{c^{m-1}|y|^{\frac{m}{2}+k-1}}\int\limits_0^1r^{\frac{m}{2}+k}P_n^k(r^2)M_c(|y|,r)\, dr\, Y_k^i(y),\label{G^*G}
	\end{align}
	where $M_c(r,s)$ is the symmetric kernel
	\begin{align}
		M_c(r,s)&=2\pi c\int\limits_0^1sJ_{k+\frac{m}{2}-1}(2\pi crs)J_{k+\frac{m}{2}-1}(2\pi cs|y|)\, ds\notag\\
		&=\frac{sJ_{k+\frac{m}{2}-2}(2\pi cs)J_{k+\frac{m}{2}-1}(2\pi c r)-rJ_{k+\frac{m}{2}-2}(2\pi cr )J_{k+\frac{m}{2}-1}(2\pi cs)}{r^2-s^2}.\label{M kernel}
	\end{align}
	In calculating (\ref{M kernel}), we have used (6.521) from \cite{gradshteyn2007ryzhik}. 
	The (finite) diagonal values $M_c(s,s)$ of the kernel $M_c$ may be obtained by taking the limit $\lim_{r\to s}M_c(s,s)$. Since ${\mathcal G^*}_c{\mathcal G}_c\psi_{2N,m}^{k,c,i}(x)=c^{-m}QP_c\psi_{2N}^{k,i}(x)=\lambda_{2N,m}^{k,c}\psi_{2N,m}^{k,c,i}(x)$, equation (\ref{G^*G}) may be written as
	$$\frac{\lambda_{2N,m}^{k,c}(c)}{4\pi^2}s^{\frac{m}{2}+k-1}P_{N,m}^{k,c}(s^2)=c\int\limits_0^1r^{\frac{m}{2}+k}P_{N,m}^{k,c}(r^2)M_c(s,r)\, dr,$$
	or equivalently
	$$\frac{\lambda_{2N,m}^{k,c}(c)}{2\pi^2}s^{\frac{m}{4}+\frac{k}{2}-\frac{1}{2}}P_{N,m}^{k,c}(s)=\int\limits_0^1r^{\frac{m}{4}+\frac{k}{2}-\frac{1}{2}}P_{N,m}^{k,c}(r)N_c(s,r)\, dr,$$
	i.e., $r^{\frac{m}{4}+\frac{k}{2}-\frac{1}{2}}P_{N,m}^{k,c}(r)$ is an eigenfunction of the integral operator on $[0,1]$ which acts by integration against the kernel $N_c(s,r)=cM_c(\sqrt s,\sqrt r)$. Since $N$ is symmetric and continuous in $r,\ s$ and $c$, we conclude from Chapter III, Section 8.4 of \cite{courant1954methods} that the eigenvalues $\lambda_{2N,m}^{k,c}$ vary continuously with $c$.
\end{proof}

\begin{Th}
\label{thm: SL asymptotics} Let the real constants $\chi_{n,m}^{k,0}$, $\chi_{n,m}^{k,c}$, $\alpha_{n,k}$, $\beta_{n,k}$, $\gamma_{n,k}$ be given by
\begin{align*}
L_0\bar C_{n,m}^0(Y_k)&=\chi_{n,m}^{k,0}\barC_{n,m}^{k,0},\\
L_c\psi_{n,m}^{k,c}&=\chi_{n,m}^{k,c}\psi_{n,m}^{k,c},\\
x^2\barC_{n,m}^0(Y_k)&=\alpha_{n,k,m}\barC_{n+2,m}^0(Y_k)+\beta_{n,k,m}\barC_{n,m}^0(Y_k)+\gamma_{n,k,m}\barC_{n-2,m}^0(Y_k).
\end{align*}
Then the asymptotic behaviours of $\chi_{n,m}^{k,c}$ and $\psi_{n,m}^{k,c}$ are as follows:
\begin{equation}
\chi_{n,m}^{k,c}=\chi_{n,m}^{k,0}-4\pi^2c^2\beta_{n,k,m}+O(c^4).\label{chi asympt}
\end{equation}
and
\begin{equation}
\psi_{n,m}^{k,c,i}=\barC_{n,m}^0(Y_k^i)-4\pi^2c^2\left(\frac{\alpha_{n,k,m}}{\chi_{n,m}^{k,0}-\chi_{n+2,m}^{k,0}}\barC_{n+2,m}^0(Y_k^i)+\frac{\gamma_{n,k,m}}{\chi_{n,m}^{k,0}-\chi_{n-2,m}^{k,0}}\barC_{n-2,m}^0(Y_k^i)\right)+O(c^4).\label{psi asympt}
\end{equation}
\end{Th}

\begin{proof}
We assume an asymptotic expansion of the form
\begin{equation}
\psi_{n,m}^{k,c,i}(x)=\barC_{n,m}^0(Y_k^i)(x)+c^2f(x)+O(c^4).\label{asymp 0}
\end{equation}
Since $\|\psi_{n,m}^{k,c,i}\|_2=\|\barC_{n,m}^0(Y_k^i)\|_2=1$, we have $\langle C_n^0(Y_k^i),f\rangle=0$.
We then have
\begin{align}
&\chi_{n,m}^{k,c}\psi_{n,m}^{k,c,i}=L_c\barC_{n,m}^0(Y_k^i)+c^2L_cf+O(c^4)\\&=L_0\barC_{n,m}(Y_k^i)+4\pi^2c^2|x|^2\barC_{n,m}^0(Y_k^i) +c^2L_0f+O(c^4)\notag\\
&=[\chi_{n,m}^{k,0}-4\pi^2c^2\beta_{n,k}]\barC_{n,m}^0(Y_k^i)-4\pi^2c^2[\alpha_{n,k}\barC_{n+2,m}^0(Y_k^i)+\gamma_{n,k}\barC_{n-2,m}^0(Y_k^i)]\notag\\
&\qquad\qquad+c^2L_0f+O(c^4).\label{asymp 1}
\end{align}
On the other hand,
\begin{equation}
\chi_{n,m}^{k,c}\psi_{n,m}^{k,c,i}=\chi_{n,m}^{k,c}[\barC_{n,m}^0(Y_k^i)+c^2f]+O(c^4).\label{asymp 2}
\end{equation}
Subtracting (\ref{asymp 2}) and (\ref{asymp 1}) gives
\begin{align}
&(\chi_{n,m}^{k,0}-4\pi^2c^2\beta_{n,k}-\chi_{n,m}^{k,c})\barC_{n,m}^0(Y_k^i)+c^2(L_0f-\chi_{n,m}^{k,c}f)\notag\\
&\qquad \qquad\qquad -4\pi^2c^2[\alpha_{nk}\barC_{n+2,m}^0(Y_k^i)+\gamma_{nk}\barC_{n-2,m}^0(Y_k^i)]=O(c^4).\label{asymp 3}
\end{align}
Since $\langle L_0f,\barC_{n,m}^0(Y_k^i)\rangle=\langle f,L_0\barC_n^0(Y_k^i)\rangle = \chi_{n,m}^{k,0}\langle f,\barC_{n,m}^0(Y_k^i)\rangle=0$, taking the inner product of both sides on (\ref{asymp 3}) against $\barC_{n,m}^0(Y_k^i)$ gives (\ref{chi asympt}).

We now aim to determine the function $f$ in (\ref{asymp 0}). Combining (\ref{chi asympt}), (\ref{asymp 0}) and (\ref{asymp 1}) gives
\begin{equation}
L_0f=\chi_{n,m}^{k,0}f-4\pi^2[\alpha_{n,k,m}\barC_{n+2,m}^0(Y_k^i)+\gamma_{n,k,m}\barC_{n-2,m}^0(Y_k^i)],\label{asymp 4}
\end{equation}
which has solutions of the form 
\begin{equation}
f=\barC_{n,m}^0(Y_k^i)A+4\pi^2\left[\frac{\alpha_{nk}}{\chi_{n+2,m}^{k,0}-\chi_{n,m}^{k,0}}\barC_{n+2,m}^0(Y_k^i)+\frac{\gamma_{nk}}{\chi_{n-2,m}^{k,0}-\chi_n^{k,0}}\barC_{n-2,m}^0(Y_k^i)\right], \label{asymp 5}
\end{equation}
where $A$ is an arbitrary Clifford constant. Substituting (\ref{asymp 5}) into (\ref{asymp 0}) gives
\begin{align}
\psi_{n,m}^{k,c,i}=\barC_{n,m}^0(Y_k^i)(1+Ac^2)&-4\pi^2c^2\bigg(\frac{\alpha_{n,k}}{\chi_{n,m}^{k,0}-\chi_{n+2,m}^{k,0}}\barC_{n+2,m}^0(Y_k^i)\nonumber\\
&+\frac{\gamma_{n,k}}{\chi_{n,m}^{k,0}-\chi_{n-2,m}^{k,0}}\barC_{n-2,m}^0(Y_k^i)\bigg)+O(c^4),\label{asymp 6}
\end{align}
However, applying $L_c$ to both sides of (\ref{asymp 6}) and applying (\ref{chi asympt}) gives
$$L_c\psi_{n,m}^{k,c,i}-\chi_{n,m}^{k,c}\psi_{n,m}^{k,c,i}=\barC_{n,m}^0(Y_k^i)Ac^2+O(c^4),$$
from which we conclude that $A=0$. Putting $A=0$ in (\ref{asymp 6}) gives (\ref{psi asympt}).
\end{proof}

For all $n\geq 0$ we say $R_{n,m}^{k,c}$ is the radial part of $\psi_{n,m}^{k,c,i}$, i.e., $R_{2N,m}^{k,c}=P_{N,m}^{k,c}$ and $R_{2N+1,m}^{k,c}=Q_{N,m}^{k,c}$.

\begin{Th}	
For fixed $k\geq 0$, the eigenvalues $\{\lambda_n^k\}_{n=0}^\infty$ of $P_{c}Q$ satisfy
$\lambda_{n+1,m}^{k,c}<\lambda_{n,m}^{k,c}$.
\end{Th}
\begin{proof}
Let's assume
$$\lambda_{n,m}^{k,c}\psi_{n,m}^{k,c,i}(x)=\int\limits_{B(1)}K_{c}(x-y)\psi_{n,m}^{k,c,i}(y)dy,$$
and,
$$\lambda_{n+1,m}^{k,c}\psi_{n+1,m}^{k,c,i}(y)=\int\limits_{B(1)}K_{c}(y-x)\psi_{n+1,m}^{k,c,i}(x)dx.$$
Therefore,
\begin{equation}\label{nth prolate dirac}
\lambda_{n,m}^{k,c}\overline{\partial_{x}\psi_{n,m}^{k,c,i}(x)}=\int\limits_{B(1)}\partial_{x}K_{c}(x-y)\overline{\psi_{n,m}^{k,c,i}(y)}dy,
\end{equation}
and,
\begin{equation}\label{(n+1)th prolate dirac}
\lambda_{n+1,m}^{k,c}\partial_{y}\psi_{n+1,m}^{k,c,i}(y)=\int\limits_{B(1)}\partial_{y}(K_{c}(y-x))\psi_{n+1,m}^{k,c,i}(x)dx.
\end{equation}
multiply
\eqref{nth prolate dirac}
by
$\lambda_{n+1,m}^{k,c}\psi_{n+1,m}^{k,c,i}(x)$
from the right so we have
\begin{align}\label{multiply nth prolate dirac}
&\lambda_{n+1,m}^{k,c}\lambda_{n,m}^{k,c}\int\limits_{B(1)}\overline{\partial_{x}\psi_{n,m}^{k,c,i}(x)}\psi_{n+1,m}^{k,c,i}(x)dx\nonumber\\
&\qquad=\lambda_{n+1,m}^{k,c}\int\limits_{B(1)}\int\limits_{B(1)}(\overline{\psi_{n,m}^{k,c,i}(y)})(\overline{\partial_{x}(K_{c}(x-y))})\psi_{n+1,m}^{k,c,i}(x)dy\; dx\nonumber\\
&\qquad=\lambda_{n+1,m}^{k,c}\int\limits_{B(1)}\int\limits_{B(1)}(\overline{\psi_{n,m}^{k,c,i}(y)})(-\partial_{x}(K_{c}(x-y)))\psi_{n+1,m}^{k,c,i}(x)dy\; dx\nonumber\\
&\qquad=\lambda_{n+1,m}^{k,c}\int\limits_{B(1)}\int\limits_{B(1)}(\overline{\psi_{n,m}^{k,c,i}(y)})(\partial_{y}(K_{c}(x-y)))\psi_{n+1,m}^{k,c,i}(x)dy\; dx\nonumber\\
&\qquad=-\lambda_{n+1,m}^{k,c}\int\limits_{B(1)}\int\limits_{B(1)}(\overline{\psi_{n,m}^{k,c,i}(y)})(\partial_{y}(K_{c}(y-x)))\psi_{n+1,m}^{k,c,i}(x)dy\; dx.\;\;\;\;\;\;\;\;\;\;\;\;   
\end{align}
where we've used the fact that
$\partial_{x}S(-x)=-\partial_{x}S(x)$.
Now, we multiply 
\eqref{(n+1)th prolate dirac}
by
$\lambda_{n,m}^{k,c}\psi_{n,m}^{k,c,i}(y)$
from the left side. So
\begin{equation}\label{multiply (n+1)th prolate dirac}
\lambda_{n+1,m}^{k,c}\lambda_{n,m}^{k,c}\int\limits_{B(1)}\overline{\psi_{n,m}^{k,c,i}(y)}(\partial_{y}\psi_{n+1,m}^{k,c,i}(y))\,dy=\lambda_{n,m}^{k,c}\int\limits_{B(1)}\int\limits_{B(1)}(\overline{\psi_{n,m}^{k,c}(y)})(\partial_{y}(K_{c}(y-x)))\psi_{n+1,m}^{k,c}(x)dx\,dy
\end{equation}
Now we use \eqref{multiply (n+1)th prolate dirac} and \eqref{multiply nth prolate dirac},
\begin{align*}
\lambda_{n+1,m}^{k,c}\lambda_{n,m}^{k,c}&\int\limits_{B(1)}\big[\overline{\psi_{n,m}^{k,c,i}(y)}(\partial_{y}\psi_{n+1,m}^{k,c,i}(y))+(\overline{\partial_{y}\psi_{n,m}^{k,c,i}(y)})\psi_{n+1,m}^{k,c,i}(y)\,dy\big]\hspace*{4cm}\\
&=[\lambda_{n,m}^{k,c}-\lambda_{n+1,m}^{k,c}]\int\limits_{B(1)}\int\limits_{B(1)}(\overline{\psi_{n,m}^{k,c,i}(y)})(\partial_{y}(K_{c}(y-x)))\psi_{n+1,m}^{k,c,i}(x)dx\,dy\\
&=[\lambda_{n,m}^{k,c}-\lambda_{n+1,m}^{k,c}]\lambda_{n+1,m}^{k,c}\int\limits_{B(1)}(\overline{\psi_{n,m}^{k,c,i}(y)})(\partial_{y}\psi_{n+1,m}^{k,c,i}(y))dy,
\end{align*}
By \eqref{psi asympt}, lemma \ref{partial_Cl_sum_of_Cl}, and the orthogonality of Clifford Legendre polynomials, we can see that 
\begin{equation}\label{invertible_partial_psi_psi}
\int\limits_{B(1)}(\overline{\psi_{n,m}^{k,c,i}(y)})(\partial_{y}\psi_{n+1,m}^{k,c,i}(y))dy=4(n+1)(n+k+\frac{m}{2})+O(c^2),
\end{equation}
The integral on the left hand side of \eqref{invertible_partial_psi_psi} is real, and is nonzero for sufficiently small $c$. Therefore we may write
\begin{equation}\label{equation_lambda_decrease_fraction_Cl}
\lambda_{n,m}^{k,c}-\lambda_{n+1,m}^{k,c}=\lambda_{n,m}^{k,c}\bigg[1+\frac{\int\limits_{B(1)}(\overline{\partial_{y}\psi_{n,m}^{k,c,i}(y)})\psi_{n+1,m}^{k,c,i}(y)\;dy}{\int\limits_{B(1)}(\overline{\psi_{n,m}^{k,c,i}(y)})(\partial_{y}\psi_{n+1,m}^{k,c,i}(y))dy}\bigg].
\end{equation}
Applying  Lemma \ref{partial_Cl_sum_of_Cl} and the orthogonality of the Clifford-Legendre polynomials on $B(1)$, we see that the numerator in the fraction on the right hand side of (\ref{equation_lambda_decrease_fraction_Cl}) goes to zero as $c\to 0$. Hence, for $c$ sufficiently small, from (\ref{equation_lambda_decrease_fraction_Cl}) we have 
$$\lambda_{n,m}^{k,c}-\lambda_{n+1,m}^{k,c}>\frac{\lambda_{m,n}^{k,c}}{2}>0,$$ 
and we conclude that for $c$ sufficiently small, \begin{equation}\label{Eigenvalue_QPc_Decrease}
\lambda_{n,m}^{k,c}>\lambda_{n+1,m}^{k,c}>0.
\end{equation}
Suppose now that for some $c_1>0$ and some non-negative integer $n$,  $\lambda_{n,m}^{k,c_1}>\lambda_{n+1,m}^{k,c_1}$. We know that there is a (small) values $c_0>0$ for which $\lambda_{n,m}^{k,c_0}>\lambda_{n+1,m}^{k,c_0}$.  Since the eigenvalues $\lambda_{n,m}^{k,c}$ are continuous functions of $c$, by the Intermediate Value Theorem, there exists $c_2\in (c_0,c_1)$ for which $\lambda_{n,m}^{k,c_1}=\lambda_{n+1,m}^{k,c_1}$, which contradicts Theorem \ref{Eigenvalues_Distinct}. This completes the proof.
\end{proof}
Because of \eqref{relation_eigenvalues}, we have a similar conclusion for the absolute values of eigenvalues of $\mathcal{G}_{c}.$
\begin{Corollary}\label{mu_decreases_as_n_increases}
	For a fixed $k$ the eigenvalues of the $\mathcal{G}_{c},$ $\mu_{N,m}^{k,c}$ are non-degenerate, and 
	$$\vert \mu_{0,m}^{k,c}\vert>\vert \mu_{1,m}^{k,c}\vert>\cdots>\vert\mu_{n,m}^{k,c}\vert>\vert\mu_{n+1,m}^{k,c}\vert>\cdots .$$
\end{Corollary}
%\begin{Th}
%	The CPSWFs $\psi_{n,m}^{k,c,i}$ are orthogonal. 
%\end{Th}
%\begin{proof}
%	Let $\psi_{n,m}^{k,c,i}(x)$ is the eigenfunction $QP_{c}$ with the eigenvalue $\lambda_{n,m}^{k,c}$. We know that for any fixed $k$ we have that
%	$$\lambda_{0,m}^{k,c}>\lambda_{1,m}^{k,c}>\lambda_{2,m}^{k,c}>\cdots,$$
%	then we have that
%	\begin{align*}
%	\lambda_{n,m}^{k,c}\langle \psi_{N,m}^{k,c},\psi_{N',m}^{k,c} \rangle&=\langle QP_{c}\psi_{N,m}^{k,c,i},\psi_{N',m}^{k,c,j} \rangle\\
%	&=\langle \psi_{N,m}^{k,c,i},QP_{c}\psi_{N',m}^{k,c,j} \rangle\\
%	&=\lambda_{n',m}^{k,c}\langle \psi_{N,m}^{k,c,i},\psi_{N',m}^{k,c,j} \rangle,
%	\end{align*}
%	then we can write
%	$$(\lambda_{n,m}^{k,c}-\lambda_{n',m}^{k,c})\langle \psi_{N,m}^{k,c,i},\psi_{N',m}^{k,c,j} \rangle=0,$$
%	since $\lambda_{n,m}^{k,c}-\lambda_{n',m}^{k,c}\neq 0,$ this will complete the proof of the theorem.
%\end{proof}
\begin{Th}
The eigenvalues of $\mathcal{G}_{c}$ are given by
\begin{equation}\label{calculation_eigenvalue_even_CPSWFs}
\mu_{2N,m}^{k,c}=\frac{\alpha_{0,N}^{k}(-1)^{m-1}\sqrt{2k+m}\,i^{k}\pi^{k+\frac{m}{2}}c^{k}}{\Gamma(k+\frac{m}{2}+1)P_{N,m}^{k,c}(0)},
\end{equation}
and
\begin{equation}\label{calculation_eigenvalue_odd_CPSWFs}
\mu_{2N+1,m}^{k,c}=\frac{\beta_{0,N}^{k}(-1)^{m}\sqrt{2k+2+m}\,i^{k+1}\pi^{k+\frac{m}{2}+1}c^{k+1}}{\Gamma(k+\frac{m}{2}+2)Q_{N,m}^{k,c}(0)}.
\end{equation}
where $\alpha_{0,N}^{k}$ and $\beta_{0,N}^{k}$ have been introduced in \eqref{eigenfunction form} and $P_{N,m}^{k,c}(x),$ and $Q_{N,m}^{k,c}(x)$ are the radial parts of the even and odd CPSWFs respectively, at proposition \ref{radial_times_monogenic_for_even}.
\end{Th}
\begin{proof}
Using the same steps in Theorem 4.6 in \cite{propertiesofcliffordlegendre}, we can see that
\begin{equation}\label{finite_Fourier_Cl_Polynomial}
\mathcal{G}_{c}(\overline{C_{2N,m}^{0}(Y_{k}^{j})})(\xi)=(-1)^{m-1}i^{k}\sqrt{2k+4N+m}\frac{J_{k+\frac{m}{2}+2N}(2\pi c\vert\xi\vert)}{\vert \xi\vert^{\frac{m}{2}+k}c^{\frac{m}{2}}}Y_{k}^{j}(\xi).
\end{equation}
Since $\mathcal{G}_{c}(\psi_{2N,m}^{k,c,j})(\xi)=\mu_{2N,m}^{k,c} \psi_{2N,m}^{k,c,j}(\xi)$ for all $\xi\in B(1)$ 
so we compute $\mu_{2N,m}^{k,c}=\lim\limits_{\xi\to 0}\frac{\mathcal{G}_{c}(\psi_{2N,m}^{k,c,j})(\xi)}{\psi_{2N,m}^{k,c,j}(\xi)}.$ From \eqref{finite_Fourier_Cl_Polynomial} and the fact that 
$$J_{\nu}(x)=\sum\limits_{j=0}^{\infty}\frac{(-1)^{j}}{j!\Gamma(j+\nu+1)}(\frac{x}{2})^{2j+\nu}, $$
we see that
$\mathcal{G}_{c}(\psi_{2N,m}^{k,c,j})(\xi)=\big[\frac{\alpha_{0,N}^{k}(-1)^{m-1}\sqrt{2k+m}i^{k}\pi^{k+\frac{m}{2}}c^{k}}{\Gamma(k+\frac{m}{2}+1)}+O(\vert \xi\vert^{2})\big]Y_{k}^{i}(\xi)$
and
$\psi_{2N,m}^{k,c,j}(\xi)=P_{N,m}^{k,c}(\vert \xi\vert^{2})Y_{k}^{i}(\xi).$ So
\begin{align*}
\frac{\mathcal{G}_{c}(\psi_{2N,m}^{k,c,j})(\xi)}{\psi_{2N,m}^{k,c,j}(\xi)}&=\frac{\big[\alpha_{0,N}^{k}(-1)^{m-1}\sqrt{2k+m}i^{k}\pi^{k+\frac{m}{2}}c^{k}{\Gamma(k+\frac{m}{2}+1)}+O(\vert \xi\vert^{2})\big]Y_{k}^{i}(\xi)}{P_{N,m}^{k,c}(\vert \xi\vert^{2})Y_{k}^{i}(\xi)}\\
&\to \frac{\alpha_{0,N}^{k}(-1)^{m-1}\sqrt{2k+m}\,i^{k}\pi^{k+\frac{m}{2}}c^{k}}{\Gamma(k+\frac{m}{2}+1)P_{N,m}^{k,c}(0)}.
\end{align*}
as $\xi\to 0.$ Hence we can get \eqref{calculation_eigenvalue_even_CPSWFs}. The calculation of the $\mu_{2N+1,m}^{k,c}$ is similar.

\end{proof}
\begin{Th}\label{mu_k_even_equal_k_minus_one_odd}
The eigenvalues of the $\mathcal{G}_{c},$ i.e, $\mu_{n,m}^{k,c}$ enjoy the following relationship
\begin{equation}\label{Relationship_Eigenvalues_even_odd}
\mu_{2N,m}^{k,c}=\mu_{2N+1,m}^{k-1,c}.
\end{equation}
\end{Th}
\begin{proof}
For the proof, we need to say that \eqref{calculation_eigenvalue_even_CPSWFs}, and \eqref{calculation_eigenvalue_odd_CPSWFs} for $k-1,$ are equal. From the Remark \ref{Relationship_Coefficients_CPSWFs}, we can see that $\alpha_{0,N}^{k}=\beta_{0,N}^{k-1}.$ Also, by the normalized version of Theorem 3.12 in \cite{propertiesofcliffordlegendre}, we can see that $P_{N,m}^{k,c}(0)=Q_{N,m}^{k-1,c}(0).$ Therefore, we can conclude the \eqref{Relationship_Eigenvalues_even_odd}.
\end{proof}
\begin{Th}\label{Theorem_eigenvalues_decreases_as_k_increases}
	For a fixed 
	$n,$
	the absolute value of the eigenvalues of the 
	$\mathcal{G}_{c},$i.e.,
	$\mu_{n,m}^{k,c}$
	are decreasing
	in terms of 
	$k=0,1,2,3,\cdots.$i.e,
	\begin{equation}\label{eigenvalues_decreases_as_k_increases}
	\vert\mu_{n,m}^{0,c}\vert>\vert\mu_{n,m}^{1,c}\vert>\vert\mu_{n,m}^{2,c}\vert>\cdots .
	\end{equation}
\end{Th}
\begin{proof}
	Let's assume $n=2N.$ From \eqref{Relationship_Eigenvalues_even_odd} we know that for any $k=1,2,3,\dots$ 
	$$\mu_{2N,m}^{k,c}=\mu_{2N+1,m}^{k-1,c},$$
	On the other side, from corollary \ref{mu_decreases_as_n_increases} we know that $\vert\mu_{2N,m}^{k-1,c}\vert>\vert\mu_{2N+1,m}^{k-1,c}\vert,$ Therefore we can conclude that for $k=1,2,3,\dots .$
	\begin{equation}\label{even_eigenvalues_decreases_as_k_increases}
	\vert\mu_{2N,m}^{k-1,c}\vert > \vert\mu_{2N,m}^{k,c}\vert .
	\end{equation}
	Now, let's assume that $n=2N+1.$ From corollary \ref{mu_decreases_as_n_increases} we know that $\vert\mu_{2N,m}^{k,c}\vert > \vert\mu_{2N+1,m}^{k,c}\vert,$ on the other side from \eqref{Relationship_Eigenvalues_even_odd} we know that for any $k=1,2,3,\dots$
	$$\mu_{2N+1,m}^{k-1,c}=\mu_{2N,m}^{k,c},$$
	Therefore we can conclude that for $k=1,2,3,\dots .$
	\begin{equation}\label{odd_eigenvalues_decreases_as_k_increases}
	\vert\mu_{2N+1,m}^{k-1,c}\vert > \vert\mu_{2N+1,m}^{k,c}\vert .
	\end{equation}
	So from \eqref{even_eigenvalues_decreases_as_k_increases}, and, \eqref{odd_eigenvalues_decreases_as_k_increases}, we can conclude \eqref{eigenvalues_decreases_as_k_increases}.
\end{proof}
Because of corollary \ref{mu_decreases_as_n_increases} and theorem \ref{Theorem_eigenvalues_decreases_as_k_increases}, we can see that for $n\leq n',$ and, $ k\leq k',$
$$ \vert \mu_{n,m}^{k,c}\vert \geq \vert \mu_{n',m}^{k',c}\vert. $$ 
This means 
$$ \vert \mu_{0,m}^{0,c}\vert \geq \vert \mu_{n',m}^{k',c}\vert, $$
for all $k'\geq 0,$ and $n'\geq 0.$ From \eqref{relation_eigenvalues} we can conclude that
$$ \lambda_{0,m}^{0,c} \geq \lambda_{n,m}^{k,c}, $$
if  $(n,k)\neq (0,0)$.

This will be helpful in proving the Spectral Concentration problem completely.
\begin{Remark}
As in one dimension PSWFs, CPSWFs enjoys the dual orthogonality feature and also they are the solution of Spectral Concentration problem. In fact, the way that we constructed $\psi_{n,m}^{k,c,i}(x),$ they are orthonormal spatial-limited functions in $B(1)$. By defining, $\varphi_{n,m}^{k,c}(x)=P_{c}\psi_{n,m}^{k,c,i}(x),$ we obtain the band-limited version of CPSWFs which are orthogonal in $\mathbb{R}^{m}.$ If we assume that $\varphi_{n,m}^{k,c,i}(x)$ are band-limited CPSWFs which are orthogonal on $\mathbb{R}^{m}$ then we can see that 
$\tilde{\varphi}_{n,mw}^{k,c,i}(x)=\frac{1}{\sqrt{\lambda_{n,m}^{k,c}}}\varphi_{n,m}^{k,c,i}(x)$ are orthonormal in $\mathbb{R}^{m}$ and orthogonal in $B(1).$ In fact we have that
\begin{align*}
\langle \tilde{\varphi}_{n,m}^{k,c,i},\tilde{\varphi}_{n',m}^{k,c,i} \rangle_{L^{2}(\mathbb{R}^{m})}
%&= \tilde{\varphi}_{n,m}^{k,c,i},\tilde{\varphi}_{n',m}^{k,c,i} \rangle_{L^{2}(\mathbb{R}^{m})} \\
& = \langle ({\lambda_{n,m}^{k,c}})^{-1/2}P_{c}\psi_{n,m}^{k,c,i} , ({\lambda_{n,m}^{k,c}})^{-1/2} P_{c}\psi_{n',m}^{k,c,i} \rangle_{L^{2}(\mathbb{R}^{m})} \\
& = ({\lambda_{n,m}^{k,c}})^{-1/2}({\lambda_{n',m}^{k,c}})^{-1/2}\langle P_{c}\psi_{n,m}^{k,c,i} ,  \psi_{n',m}^{k,c,i} \rangle_{L^{2}(\mathbb{R}^{m})}\\
& = ({\lambda_{n,m}^{k,c}})^{-1/2}({\lambda_{n',m}^{k,c}})^{-1/2}\langle QP_{c}\psi_{n,m}^{k,c,i} ,  \psi_{n',m}^{k,c,i} \rangle_{L^{2}(B(1))} \\
& =(\lambda_{n,m}^{k,c})^{1/2}({\lambda_{n',m}^{k,c}})^{-1/2}\langle \psi_{n,m}^{k,c,i} ,  \psi_{n',m}^{k,c,i} \rangle_{L^{2}(B(1))} =\delta_{nn'}.
\end{align*}
Also,
\begin{align*}
\langle \tilde{\varphi}_{n,m}^{k,c,i},\tilde{\varphi}_{n',m}^{k,c,i} \rangle_{L^{2}(B(1))} & = \langle Q\tilde{\varphi}_{n,m}^{k,c,i},P_{c}\tilde{\varphi}_{n',m}^{k,c,i} \rangle_{L^{2}(\mathbb{R}^{m})}\\
& = \langle P_{c} Q\tilde{\varphi}_{n,m}^{k,c,i},\tilde{\varphi}_{n',m}^{k,c,i} \rangle_{L^{2}(\mathbb{R}^{m})}\\
& = \lambda_{n,m}^{k,c} \langle \tilde{\varphi}_{n,m}^{k,c,i},\tilde{\varphi}_{n',m}^{k,c,i} \rangle_{L^{2}(\mathbb{R}^{m})} = \lambda_{n,m}^{k,c} \delta_{nn'}.
\end{align*}
So $\psi_{n,m}^{k,c,i}=\frac{1}{\sqrt{\lambda_{n,m}^{k,c}}}Q\tilde{\varphi}_{n,m}^{k,c,i}$ are orthonormal spatial-limited CPSWFs basis in $L^{2}(B(1)).$
Now let's assume that $f\in PW_{c}(\mathbb{R}^m),$ and $\langle f,\tilde{\varphi}_{n',m}^{k,c,i} \rangle_{L^{2}(\mathbb{R}^{m})}=0,$
for all $n.$ We will prove that $f=0.$ So we have
% Fix here minus -1 minus half
\begin{align*}
0=\langle f,\tilde{\varphi}_{n',m}^{k,c,i} \rangle_{L^{2}(\mathbb{R}^{m})}&=(\lambda_{n,m}^{k,c})^{-1}\langle f,P_{c}Q\tilde{\varphi}_{n',m}^{k,c,i} \rangle_{L^{2}(\mathbb{R}^{m})}\\
&=(\lambda_{n,m}^{k,c})^{-1}\langle f,Q\tilde{\varphi}_{n,m}^{k,c,i} \rangle_{L^{2}(\mathbb{R}^{m})}\\
&=(\lambda_{n,m}^{k,c})^{-\frac{1}{2}}\langle Qf,\psi_{n,m}^{k,c,i} \rangle_{L^{2}(\mathbb{R}^{m})}.
\end{align*}
Since $\{\psi_{n,m}^{k,c,i};\, n\geq 0,\, k\geq 0, 1\leq i\leq d_{k} \}$ is an orthonormal basis  for $L^{2}(B(1))$, we have $Qf=0$. 
Since $f\in PW_{c}(\mathbb{R}^{m}),$ it is an analytic function therefore $f=0$ in $\mathbb{R}^{m}$. This means that the collection $\{\tilde{\varphi}_{n,m}^{k,c,i}:\, n,k\geq 0,\ 1\leq i\leq d_k\}$ is an orthonormal basis for $PW_{c}(\mathbb{R}^m)$.

The spectral concentration problem refers to the search for the function in $PW_{c}(\mathbb{R}^m)$ which keeps most energy in a unit ball, i.e, we would like to find a function, $f \in PW_{c}(\mathbb{R}^m)$ (with $\|f\|_{L^2({\mathbb R}^m)}=1$ for which the energy concentration
$$\frac{\int\limits_{B(1)}\vert f(x)\vert^{2}dx}{\int\limits_{\mathbb{R}^{m}}\vert f(x)\vert^{2}dx},$$
is the maximised. Since $f\in PW_{c},$ it admits an expansion of the form $ f(x)=\sum\limits_{n=0}^{\infty}\sum\limits_{k=0}^{\infty}\sum\limits_{i=1}^{d_{k}}c_{n,m}^{k,i}\tilde{\varphi}_{n,m}^{k,c,i}(x).$ Then 
\begin{align*}
\int\limits_{B(1)}\vert f(x)\vert^{2} dx 
%&= \left(\int\limits_{B(1)} \overline{f(x)}f(x) dx\right)_{0}\\
&=\sum\limits_{n=0}^{\infty}\sum\limits_{k=0}^{\infty}\sum\limits_{i=1}^{d_{k}}\vert c_{n,m}^{k,i}\vert^{2}\lambda_{n,m}^{k,c}\\
&\leq \lambda_{0,m}^{0,c}\sum\limits_{n=0}^{\infty}\sum\limits_{k=0}^{\infty}\sum\limits_{i=1}^{d_{k}}\vert c_{n,m}^{k,i}\vert^{2}=\lambda_{0,m}^{0,c}\Vert f\Vert_{L^{2}(\mathbb{R}^{m})}.
\end{align*}
Therefore,
$$\frac{\int\limits_{B(1)}\vert f(x)\vert^{2}dx}{\int\limits_{\mathbb{R}^{m}}\vert f(x)\vert^{2}dx}\leq \lambda_{0,m}^{0,c}= \frac{\int\limits_{B(1)}\vert \tilde{\varphi}_{0,m}^{0,c,1}(x)\vert^{2}dx}{\int\limits_{\mathbb{R}^{m}}\vert \tilde{\varphi}_{0,m}^{0,c,1}(x)\vert^{2}dx}.$$
\end{Remark}
\begin{figure}
	\centering
	\includegraphics[width=.5\linewidth]{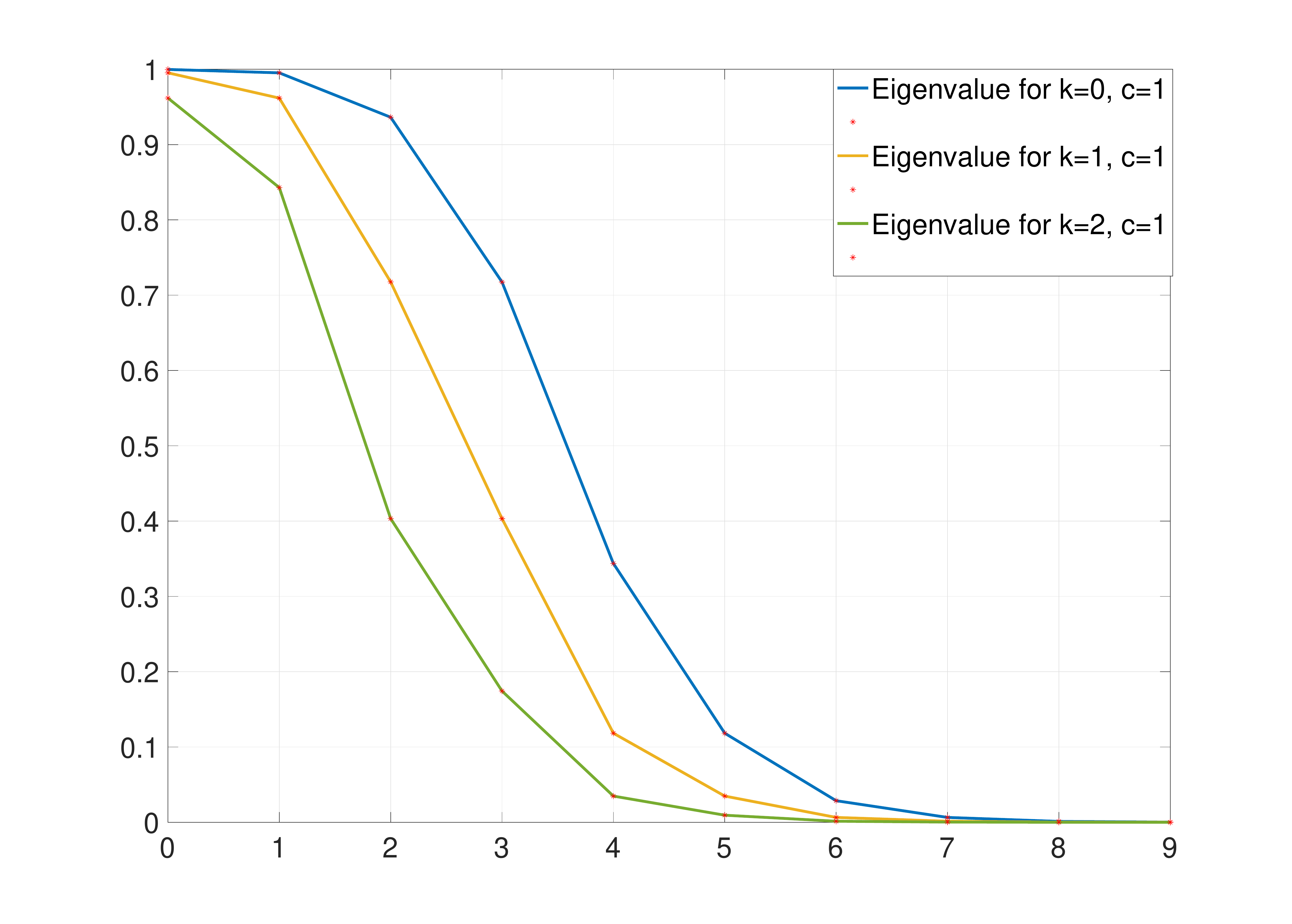}
	\caption{The Graph of Changes of the $\mu_{n,2}^{k,c}$ eigenvalues of the 2-dimension CPSWFs for $k=0,1,2,\; c=1$}
\end{figure}
\begin{figure}
	\centering
	\includegraphics[width=.5\linewidth]{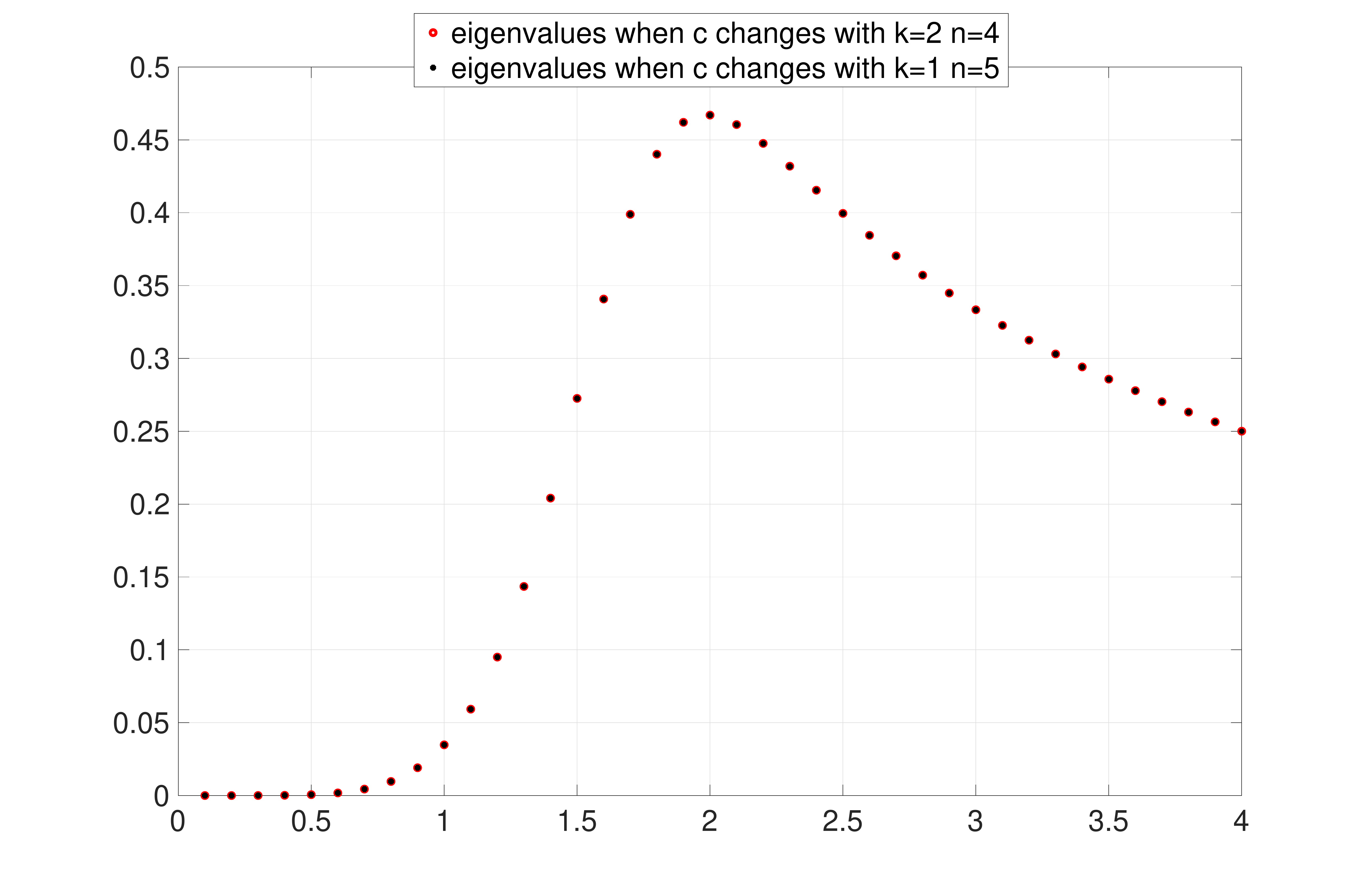}
	\caption{The Graph of Changes of the eigenvalues of the 2-dimension CPSWFs as $c$ changes for different values of the $k=2,1$ and $n=4,5$}
\end{figure}
\section{Spectrum Accumulation}
In this short section, we  present the spectrum accumulation properties of $m-$dimensional CPSWFs. For one-dimensional PSWFs, this property was investigated in \cite{hogan2015frame} and has important applications in the construction of multi-taper channel  estimation algorithms \cite{thomson82}, \cite{hogan2011duration}. %In fact the Spectrum Accumulation tells us that the summation of squared CPSWFs can be manipulatable if we involve the eigenvalues. 

Recall that $\lambda_{n,m}^{k,c}\psi_{n,m}^{k,c,i}(x)=\int\limits_{B(1)}K_{c}(x-y)\psi_{n,m}^{k,c,i}(y)dy$ where $K_{c}(x)=c^{m}\int\limits_{B(1)}e^{2\pi ic\langle \omega,x\rangle}d\omega.$ Then
\begin{align*}
\sum\limits_{k=0}^{\infty}\sum_{n=0}^{\infty}\sum_{i=1}^{d_{k}} \lambda_{n,m}^{k,c}\vert \psi_{n,m}^{k,c,i}(x)\vert^{2}&=\sum\limits_{k=0 }^{\infty}\sum_{n=0}^{\infty}\sum_{i=1}^{d_{k}} \lambda_{n,m}^{k,c}\big[\overline{\psi_{n,m}^{k,c,i}(x)}\,\psi_{n,m}^{k,c,i}(x)\big]_{0}\hspace*{10cm}\\
&=\big[\sum\limits_{k=0}^{\infty}\sum_{n=0}^{\infty}\sum_{i=1}^{d_{k}}(\int\limits_{B(1)}\overline{\psi_{n,m}^{k,c,i}(y)}K_{c}(x-y)dy)\psi_{n,,m}^{k,c,i}(x)\big]_{0}\\
%&=\big[\sum\limits_{k=0}^{\infty}\sum_{n=0}^{\infty}\sum_{i=1}^{d_{k}}(\int\limits_{B(1)}\overline{\psi_{n,m}^{k,c,i}(y)}K_{c}(x-y)dy)\psi_{n,m}^{k,c,i}(x)\big]_{0}\\
&=\big[\sum\limits_{k=0}^{\infty}\sum_{n=0}^{\infty}\sum_{i=1}^{d_{k}}\langle \psi_{n,m}^{k,c,i},K_{c}(x-\cdot) \rangle\psi_{n,m}^{k,c,i}(x)\big]_{0}=K_c(0)=c^m \vert B(1)\vert,
\end{align*} 
On the other hand, if instead of summing over all $n,k\geq 0$ in the above calculation, we instead perform a truncated sum by restricting the values of $n$ and $k$ so that $0\leq n\leq 2N+1$ and $0\leq k\leq K$, then we have
\begin{align*}
&\sum\limits_{k=0}^{K}\sum_{n=0}^{2N+1}\sum_{i=1}^{d_{k}} \lambda_{n,m}^{k,c}\vert \psi_{n,m}^{k,c,i}(x)\vert^{2}\\
	&\qquad =\sum\limits_{k=0}^{K}\sum_{n=0}^{N}\sum_{i=1}^{d_{k}} \lambda_{2n,m}^{k,c}\vert \psi_{2n,m}^{k,c,i}(x)\vert^{2}+\sum\limits_{k=0}^{K}\sum_{n=0}^{N}\sum_{i=1}^{d_{k}} \lambda_{2n+1,m}^{k,c}\vert \psi_{2n+1,m}^{k,c,i}(x)\vert^{2}\\
&\qquad =\sum\limits_{k=0}^{K}\sum_{n=0}^{N} \lambda_{2n,m}^{k,c}\vert P_{n,m}^{k,c}(\vert x\vert^2)\vert^{2} \vert x\vert^{2k}\sum_{i=1}^{d_{k}}\vert Y_{k}^{i}(\frac{x}{\vert x\vert})\vert^{2}\\
&\qquad +\sum\limits_{k=0}^{K}\sum_{n=0}^{N} \lambda_{2n+1,m}^{k,c}\vert Q_{n,m}^{k,c}(\vert x\vert^{2})\vert^{2}\vert x\vert^{2k+2}\sum_{i=1}^{d_{k}}\vert Y_{k}^{i}(\frac{x}{\vert x\vert})\vert^{2}\\
%&=\sum\limits_{k=0}^{K}\sum_{n=0}^{N} \lambda_{2n,m}^{k,c}\vert P_{n,m}^{k,c}(\vert x\vert^2)\vert^{2} \vert x\vert^{2k}\sum_{i=1}^{d_{k}}\big[ \overline{Y_{k}^{i}(\frac{x}{\vert x\vert})}Y_{k}^{i}(\frac{x}{\vert x\vert})\big]_{0}\\
%&+\sum\limits_{k=0}^{K}\sum_{n=0}^{N} \lambda_{2n+1,m}^{k,c}\vert Q_{n,m}^{k,c}(\vert x\vert^{2})\vert^{2}\vert x\vert^{2k+2}\sum_{i=1}^{d_{k}}\big[ \overline{Y_{k}^{i}(\frac{x}{\vert x\vert})}Y_{k}^{i}(\frac{x}{\vert x\vert})\big]_{0}\\
&\qquad =\sum\limits_{k=0}^{K}\sum_{n=0}^{N} \lambda_{2n,m}^{k,c}\vert P_{n,m}^{k,c}(\vert x\vert^2)\vert^{2} \vert x\vert^{2k}\frac{(k+m-2)}{\vert S^{m-1}\vert (m-2)}\\
&\qquad +\sum\limits_{k=0}^{K}\sum_{n=0}^{N} \lambda_{2n+1,m}^{k,c}\vert Q_{n,m}^{k,c}(\vert x\vert^{2})\vert^{2}\vert x\vert^{2k+2}\frac{(k+m-2)}{\vert S^{m-1}\vert (m-2)}=G(\vert x\vert^{2}),
\end{align*}
where we have used theorem 3.3 from \cite{de2016reproducing}:
$$\sum_{i=1}^{d_{k}}\overline{Y_{k}^{i}(x)} Y_{k}^{i}(y)=\frac{(k+m-2)}{\vert S^{m-1}\vert (m-2)}\vert x\vert^{k}\vert y\vert^{k}C_{k}^{\mu}(t)+(x\wedge y)\vert x\vert ^{k-1}\vert y\vert^{k-1}C_{k-1}^{\mu+1}(t),$$
in which $t=\frac{\langle x,y\rangle}{\vert x\vert \vert y\vert},$ $C_{k}^{\mu}(t)$ is the Gegenbauer polynomials defined on the line, and $\vert S^{m-1}\vert$ is the Lebesgue measure of $S^{m-1}.$
In the figures \ref{spectrumtwodimcone}, and \ref{spectrumtwodimctwo}, we  see numerical computations of teh partial sums in $2$ and $3$ dimensions, demonstrating the convergence of the partial sums to the constant $c^m|B(1)|$.
\begin{figure}
	\centering
	\includegraphics[width=.5\linewidth]{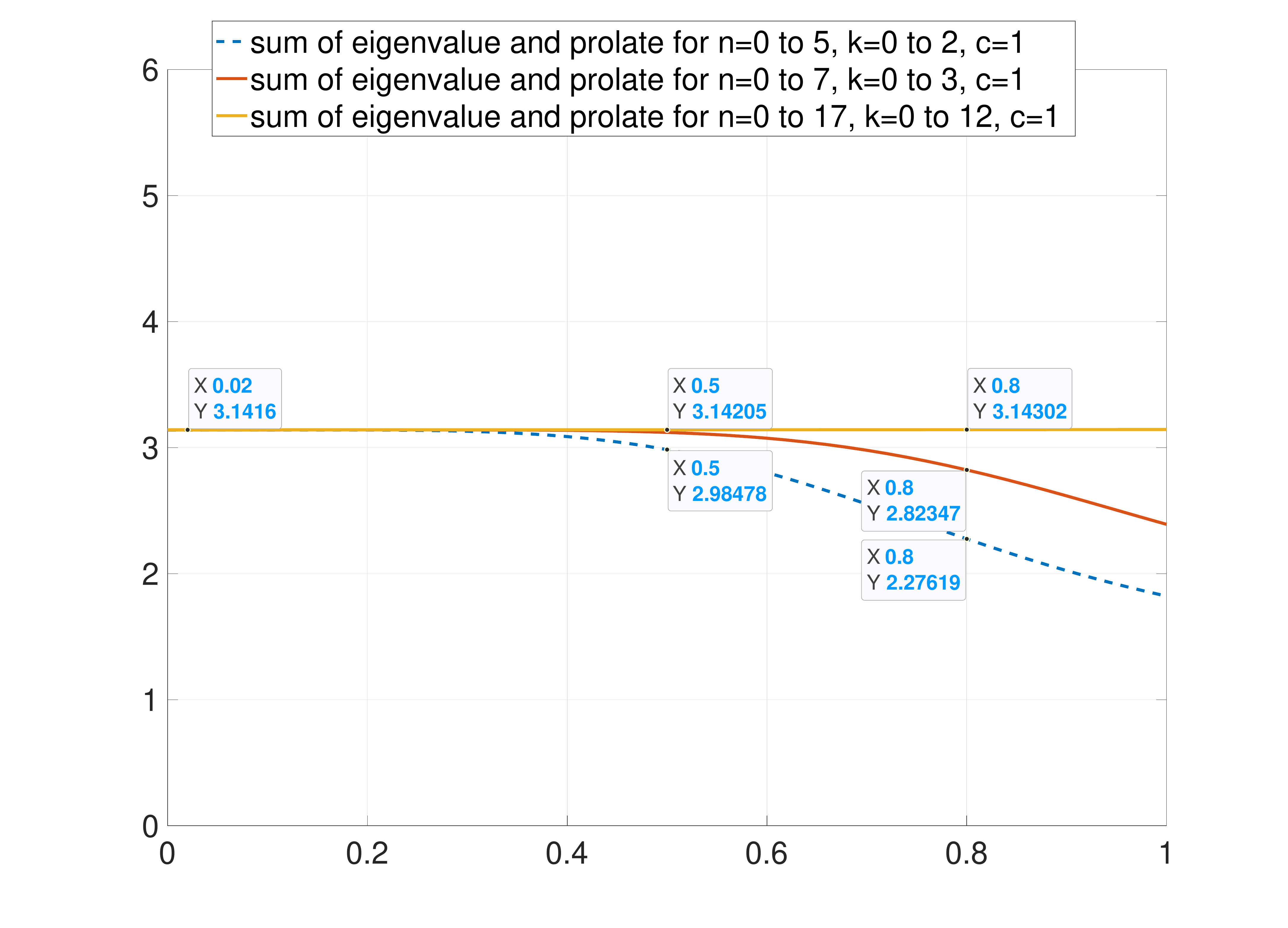}
	\caption{The Radial Graph of numerical computations of spectrum accumulation properties of $2-$ dimension CPSWFs for $c=1.$}\label{spectrumtwodimcone}
\end{figure}
\begin{figure}
	\centering
	\includegraphics[width=.5\linewidth]{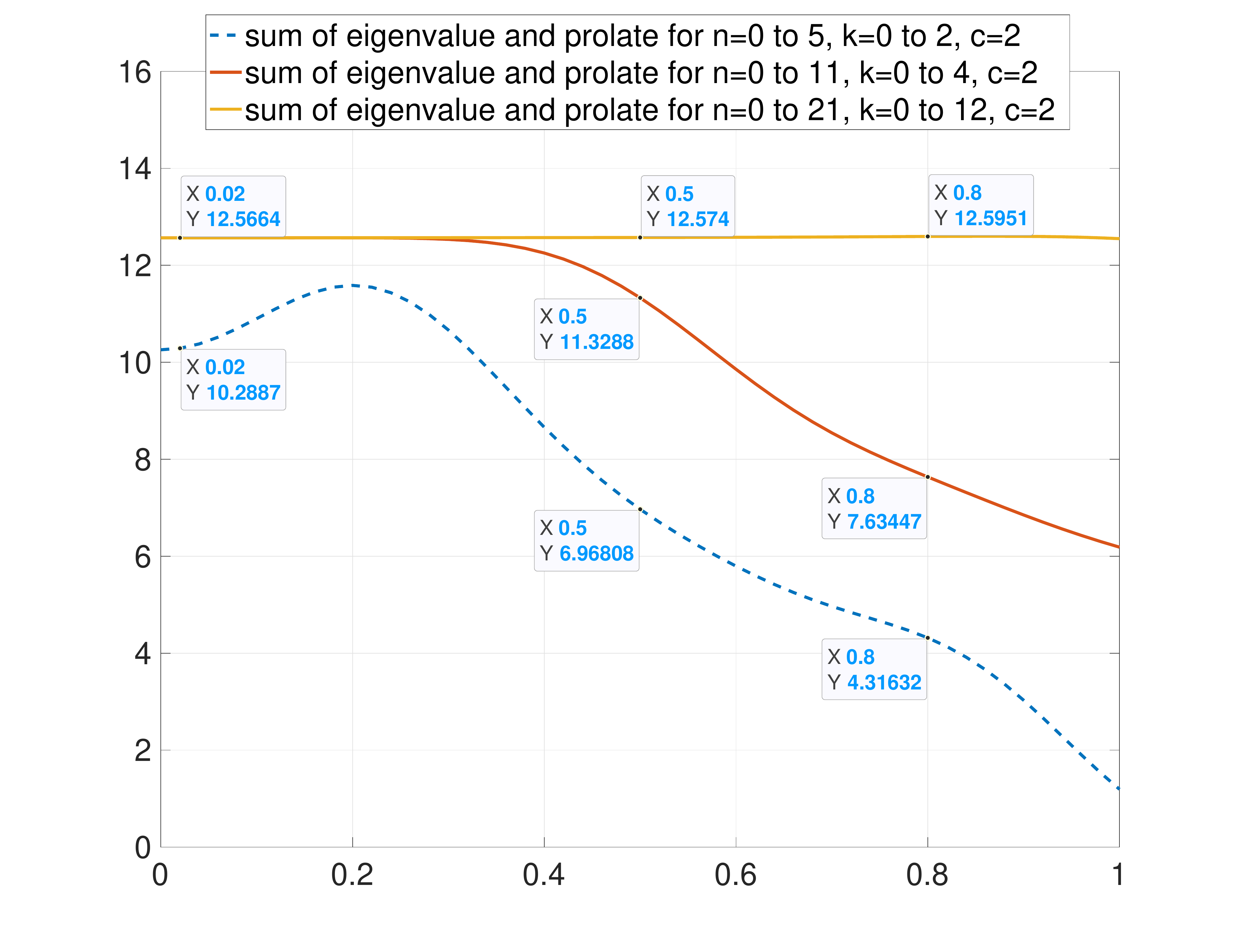}
	\caption{The Radial Graph of numerical computations of spectrum accumulation properties of $2-$ dimension CPSWFs for $c=2.$}\label{spectrumtwodimctwo}
\end{figure}
\begin{figure}
	\centering
	\includegraphics[width=.5\linewidth]{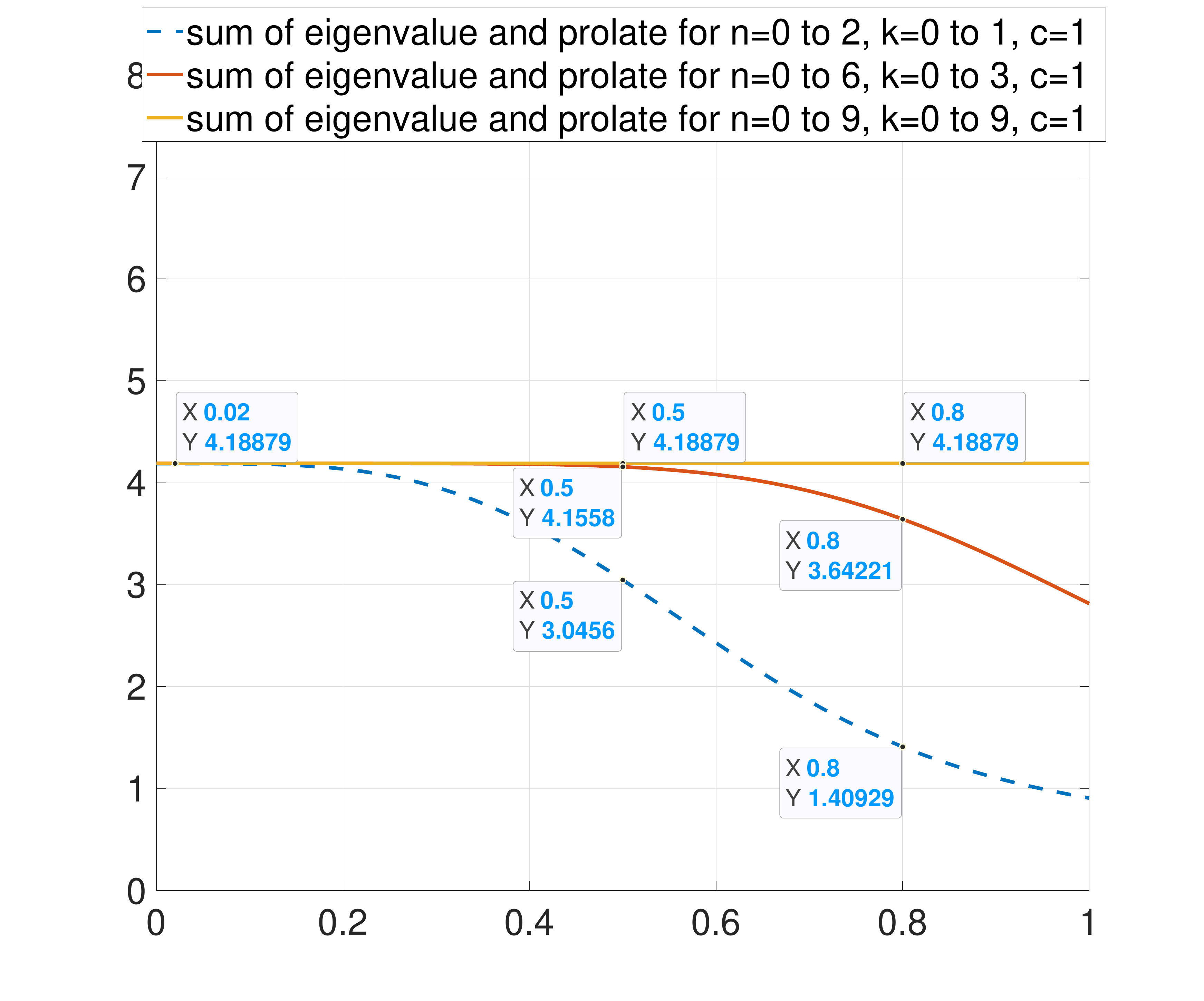}
	\caption{The Radial Graph of numerical computations of spectrum accumulation properties of $3-$ dimension CPSWFs for $c=1.$}\label{spectrumthreedimcone}
\end{figure}
\begin{figure}
	\centering
	\includegraphics[width=.5\linewidth]{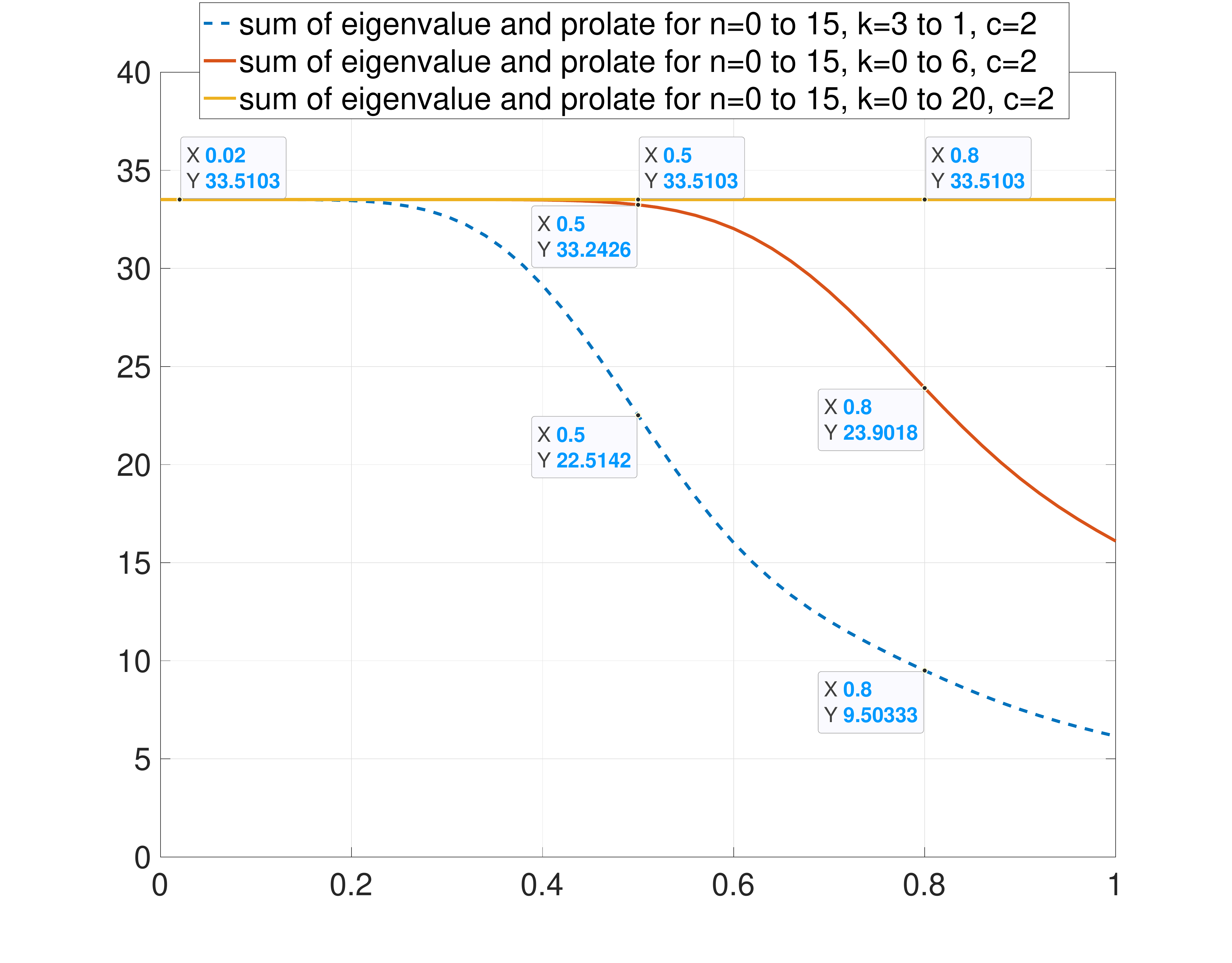}
	\caption{The Radial Graph of numerical computations of spectrum accumulation properties of $3-$ dimension CPSWFs for $c=2.$}\label{spectrumthreedimctwo}
\end{figure}

\begin{Remark}
	All codes related to the calculations of CPSWFs, their plots, and their eigenvalues are available in Github. The codes are written in different platforms, namely Matlab, Maple, Mathematica, Python, Julia, Sagemath. 
\end{Remark}

\section*{Acknowledgment}
\noindent The authors would like to thank the Center for Computer-Assisted Research in Mathematics and its Applications at the University of Newcastle for its continued support. JAH is supported by the Australian Research Council through Discovery Grant DP160101537. Thanks Roy. Thanks HG. Special thanks also to Dr.  AmirHosein Sadeghimanesh for his useful comments about using different programming languages. 
%%%%%%%%%%%%%
\bibliographystyle{siam}
\bibliography{Clifford_Prolate_Spheroidal_Wave_Functions}

\end{document}